\documentclass[11pt]{amsart}
\usepackage{amsmath,amssymb,mathrsfs,hyperref,color}

\usepackage[nameinlink]{cleveref}
\hypersetup{colorlinks={true},linkcolor={blue},citecolor=green}
\crefformat{equation}{(#2#1#3)}

\usepackage{mathtools}
\usepackage[x11names]{xcolor}
\newtagform{blue}{(}{)}

\usepackage{paralist}

\usepackage{subfigure}
\usepackage{float}
\usepackage{times}
\usepackage{tikz}
\usepackage{cases}
\usetikzlibrary{intersections}
\usepackage{xcolor}
\usepackage[latin1]{inputenc}
\usepackage{bbm}
\usetikzlibrary{shadings}
\usetikzlibrary[intersections]
\usetikzlibrary{arrows,shapes}

\makeatletter
\def\setliststart#1{\setcounter{\@listctr}{#1}%
  \addtocounter{\@listctr}{-1}}
\makeatother

\makeatletter
\@addtoreset{figure}{section}
\makeatother

\usepackage{calc}
\newtheorem{theorem}{Theorem}[section]
\newtheorem{lemma}[theorem]{Lemma}
\newtheorem{proposition}[theorem]{Proposition}

\newtheorem{remarks}[theorem]{Remark}
\newtheoremstyle{break}
  {\topsep}{\topsep}%
  {\itshape}{}%
  {\bfseries}{}%
  {\newline}{}%
\theoremstyle{break}

\newtheorem{definition}[theorem]{Definition}
\numberwithin{equation}{section}

\newcommand{\T}{\mathbb{T}}
\newcommand{\R}{\mathbb{R}}

\newcommand{\N}{\mathbb{N}}
\newcommand{\E}{\mathcal{E}}

\newcommand{\PP}{\mathcal{P}}

\newcommand{\C}{\mathcal{C}}

\newcommand{\ds}{\displaystyle}

\DeclareMathOperator*{\argmin}{argmin} 
\DeclareMathOperator*{\supp}{spt}
\DeclareMathOperator*{\esssup}{ess\ sup}
\DeclareMathOperator*{\ddiv}{div}

\DeclareMathOperator*{\eps}{\varepsilon}
\DeclareMathOperator*{\blambda}{\bar\lambda}

\makeatletter
\def\moverlay{\mathpalette\mov@rlay}
\def\mov@rlay#1#2{\leavevmode\vtop{%
   \baselineskip\z@skip \lineskiplimit-\maxdimen
   \ialign{\hfil$\m@th#1##$\hfil\cr#2\crcr}}}
\newcommand{\charfusion}[3][\mathord]{
    #1{\ifx#1\mathop\vphantom{#2}\fi
        \mathpalette\mov@rlay{#2\cr#3}
      }
    \ifx#1\mathop\expandafter\displaylimits\fi}
\makeatother

\title[Ergodic behavior of Control and MFG depending on acceleretion]{Ergodic behavior of Control and Mean Field Games Problems depending on acceleration}
\author{Pierre Cardaliaguet \and Cristian Mendico}
\address{CEREMADE, Universit\'e Paris-Dauphine, PSL University, UMR CNRS 7534, CEREMADE, Place du Maréchal de Lattre de Tassigny - 75775 PARIS Cedex 16}
\email{cardaliaguet@ceremade.dauphine.fr}

\address{GSSI-Gran Sasso Science Institute, Viale F. Crispi 7, 67100  L'Aquila and CEREMADE, Universit\'e Paris-Dauphine, Place du Maréchal de Lattre de Tassigny - 75775 PARIS Cedex 16}
\email{cristian.mendico@gssi.it}

\date{\today}
\subjclass[2010]{}

\begin{document}
\usetagform{blue}
\maketitle
\begin{abstract}
		The goal of this paper is to study the long time behavior of solutions of the first-order mean field game (MFG) systems with a control on the acceleration. The main issue for this is the lack of small time controllability of the problem, which prevents to define the associated ergodic mean field game problem in the standard way. To overcome this issue, we first study the long-time average of optimal control problems with control on the acceleration: we prove that the time average of the value function converges to an ergodic constant and represent this ergodic constant as a minimum of a Lagrangian over a suitable class of closed probability measure. This characterization leads us to define the ergodic MFG problem as a fixed-point problem on the set of closed probability measures. Then we also show that this MFG ergodic problem has at least one solution, that the associated ergodic constant is unique under the standard monotonicity assumption and that the time-average of the value function of the time-dependent MFG problem with control of acceleration converges to this ergodic constant.
	\end{abstract}
\tableofcontents
	\section{Introduction}
	\label{sec:Intro}

The main goal of this paper is to study the asymptotic behavior of mean field games (MFG) system with acceleration. Let us recall that such systems, first introduced in \cite{bib:CM, bib:YA}, aim to describe models with infinite number of interacting agents who  control  their acceleration. The MFG models we study here read as follows
\begin{align}\label{intro:MFGa}
	\begin{split}
	\begin{cases}
	-\partial_{t}u^{T}(t,x,v) + \frac{1}{2}|D_{v}u^{T}(t,x,v)|^{2}-\langle D_{x}u^{T}(t,x,v),v \rangle 	=F(x,v,m^{T}_{t}), & {\rm in}\; [0,T] \times \T^{d} \times \R^{d}
	\\
	\partial_{t}m^{T}_{t}-\langle v, D_{x}m^{T}_{t} \rangle- \ddiv\Big(m^{T}_{t}D_{v}u^{T}(t,x,v) \Big)=0, 	 &  {\rm in}\;  [0,T] \times \T^{d} \times \R^{d}
	\\
	 u^{T}(T,x,v)=g(x,v,m^{T}_{T}), \;  \quad m^{T}_{0}(x,v)=m_{0}(x,v)\; {\rm in}\;  \T^{d} \times \R^{d}. 
	\end{cases}
	\end{split}
\end{align} 
The above coupled system is a particular case of the more general class of MFG systems, which aim to describe the optimal value $u$ and the distribution $m$  of players, in a Nash equilibrium, for  differential games with infinitely many small players.  This models were introduced independently by Lasry and Lions \cite{bib:LL1,bib:LL2,bib:LL3} and Huang, Malham\'e and Caines \cite{bib:HCM1,bib:HCM2} and since these pioneering works the MFG theory has grown very fast:  see for instance  the survey papers and the monographs \cite{bib:NC, bib:DEV, bib:BFY, bib:CD1} and the references therein. In system \eqref{intro:MFGa} the pair $(u^T, m^T)$ can be interpreted as follows: $u^T$ is the value function of a typical small player for an optimal control problem of acceleration in which the cost depends on the time-dependent family of probability measures $(m^T_t)$; on the other hand, $m^T_t$ is, for each time $t$, the distribution of the small players; it evolves in time according to the continuity equation driven by the optimal feedback of the players.  

%

During the last years, the question of the long time behavior of solutions of (standard) MFG systems has attracted a lot of attention. Results describing the long-time average of solutions were obtained in several context: see \cite{bib:CLLP2, bib:CLLP1}, for second order systems on $\T^{d}$, and \cite{bib:CAR, bib:CCMW, bib:CCMW1}, for first order systems on $\T^{d}$, $\R^{d}$ and for state constraint case respectively. Recently, the first author and Porretta studied the long time behavior of solutions for the so-called Master equation associated with a second order MFG system, see \cite{bib:CP2}. In view of the results obtained in these works one would expect the limit of $u^{T}/T$ to be described by the following ergodic system
\begin{align}\label{intro:ergoMFG}
	\begin{split}
	\begin{cases}
	\frac{1}{2}|D_{v}u(x,v)|^{2}-\langle D_{x}u(x,v),v \rangle 	=F(x,v,m), &  (x,v) \in \T^{d} \times \R^{d}
	\\
	-\langle v, D_{x}m \rangle- \ddiv\Big(mD_{v}u(x,v) \Big)=0, 	 &  (x,v) \in  \T^{d} \times \R^{d}
	\\
	 \int_{\T^{d} \times \R^{d}}{m(dx,dv)}=1.
	\end{cases}
	\end{split}
\end{align} 
The main issue of this paper is that this ergodic system makes no sense. Indeed, as we explain below, even for problems without mean field interaction, we cannot expect to have a solution to the corresponding ergodic Hamilton-Jacobi equation (the first equation in \eqref{intro:ergoMFG}). As the drift of the continuity equation (the second equation in \eqref{intro:ergoMFG})  is given in terms of solution to the ergodic Hamilton-Jacobi equation, there is no hope to formulate the problem in this way. As far as we know, this is the first time this kind of problem is faced in the literature. \\

 To overcome the issue just described, we first study the ergodic Hamilton-Jacobi equation without mean field interaction. More precisely, in the first part of the paper we investigate the existence of the limit, as $T$ tends to infinity, of $u^T(0, \cdot,\cdot)/T$, where now $u^T$ solves the Hamilton-Jacobi equation (without mean field interaction)
$$
\left\{\begin{array}{l}		
-\partial_{t} u^T(t,x,v) + \frac{1}{2}|D_{v}u^T(t,x,v)|^{2}- \langle D_{x}u^T(t,x,v), v \rangle = F(x,v), \quad {\rm in}\;  [0,T] \times \T^{d} \times \R^{d}	\\
u^T(T,x,v)= 0 \; {\rm in }\; \T^d\times \R^d. 
\end{array}\right.
$$ 
Here $F:\R^d\times \R^d\to \R$ is periodic in space (the first variable) and coercive in velocity (the second one). 
Following the seminal paper  \cite{bib:LPV}, it is known that the existence of the limit of $u^T/T$ is related with the existence of a corrector, namely to a solution of the ergodic Hamilton-Jacobi equation:  
\begin{equation*}
	-\langle D_{x}u(x,v), v \rangle +\frac{1}{2}|D_{v}u(x,v)|^{2}=F(x,v) + \bar{c}, \quad (x,v) \in \times \T^{d} \times \R^{d},
\end{equation*}
for some constant $\bar c$.  However, we stress again the fact that due to the lack of coercivity and due to the lack of small time controllability of our model, we do not expect the existence of a continuous viscosity solutions of the ergodic equation. This problem has been overcome in several other frameworks: we can quote for instance \cite{bib:MG, bib:CNS, bib:XY, bib:PCE, bib:BA, bib:AL, bib:ABE, bib:ARG, bib:BQR, bib:GLM, bib:GV}, for related problems see also \cite{bib:ABG, bib:KDV}  and the references therein. Following techniques developed in \cite{bib:ARG} we prove in the first part of Theorem \ref{theo:main1} that the limit of $u^T/T$ exists and is equal to a constant. 
 However, this convergence result does not suffice to handle our MFG system of acceleration: indeed, we  also need to understand, when the map $F$ also depends on the extra parameter $m$, how this ergodic constant depends on $m$. For doing so, we follow  ideas from weak-KAM theory (see for instance \cite{bib:FA}) and characterize the ergodic constant in terms of closed probability measures: namely, we prove in the second part of Theorem \ref{theo:main1} that, for any $(x,v)\in\T^d\times \R^d$,  
 $$
 \lim_{T\to+\infty} \frac{u^T(0,x,v)}{T} =  \inf_{\mu \in \C} \int_{\T^{d} \times \R^{d} \times \R^{d}}{\left(\frac{1}{2}|w|^{2} +F(x,v)\right)\ \mu(dx,dv,dw)}
 $$
 where $\C$ is the set of Borel probability measures $\mu$ on $\T^d\times \R^d$ with suitable finite moments and which are closed in the sense that, for any test function $\varphi \in C^{\infty}_{c}(\T^{d} \times \R^{d})$,
		\begin{equation*}
			\int_{\T^{d} \times \R^{d} \times \R^{d}}{\Big(\langle D_{x}\varphi(x,v), v \rangle + \langle D_{v}\varphi(x,v), w \rangle \Big)\ \eta(dx,dv,dw)}=0,
		\end{equation*} 
(see also \Cref{def:closedmeasure}). \\

We now come back to our MFG of acceleration \eqref{intro:MFGa}.  
 In view of the characterization of the ergodic constant for the Hamilton-Jacobi equation without mean field interaction, it is natural to describe an equilibrium for the ergodic MFG problem with acceleration as a fixed-point problem on the Wasserstein space: we say that $(\bar \lambda, \bar \mu)\in \R\times \C$ is a solution of the ergodic MFG problem of acceleration if 
\begin{align*}
	\bar\lambda =\ & \inf_{\mu \in \C} \int_{\T^{d} \times \R^{d} \times \R^{d}}{\left(\frac{1}{2}|w|^{2} +F(x,v, \pi \sharp \bar\mu)\right)\ \mu(dx,dv,dw)}	
\\
=\ & \int_{\T^{d} \times \R^{d} \times \R^{d}}{\left(\frac{1}{2}|w|^{2} +F(x,v, \pi \sharp \bar{\mu})\right)\ \bar{\mu}(dx,dv,dw)},
\end{align*}
where $\pi:\T^d\times \R^d\times \R^d\to \T^d\times \R^d$ is the canonical projection onto the first two variables. 
We show that such an ergodic MFG problem with acceleration has a solution and that the associated ergodic constant $\bar \lambda$ is unique under the following monotonicity condition (first introduced in \cite{bib:LL1, bib:LL2}): there exists a constant $M_{F} > 0$ such that for any $m_{1}$, $m_{2} \in \PP(\T^{d} \times \R^{d})$
\begin{equation*}
\begin{array}{l}
\ds  \int_{\T^{d} \times \R^{d}}{\big(F(x,v,m_{1})-F(x,v,m_{2}) \big) \ (m_{1}(dx,dv)-m_{2}(dx,dv))} 
	\\
\ds \qquad \geq\  M_{F}\int_{\T^{d} \times \R^{d}}{\big(F(x,v,m_{1})-F(x,v,m_{2}) \big)^{2}\ dxdv},
\end{array}
\end{equation*}
see $(1)$ in  \Cref{thm:theo.main2}. The main result of the paper is the fact that, if $(u^T,m^T)$ solves the MFG system of acceleration \eqref{intro:MFGa}, then $u^T(0,x,v)/T$ converges, as $T$ tends to infinity, to the unique ergodic constant $\bar\lambda$ of the ergodic MFG problem, see $(2)$ in \Cref{thm:theo.main2}. The main  technical step for this is to rewrite the MFG system in terms of time-dependent closed measure (a kind of occupation measure in this set-up), see \Cref{thm:lintegral}, and to understand the long-time average of these measures.

\medskip\medskip\medskip 
The rest of this paper is organized as follows.
In \Cref{sec:MainResults}, we introduce the notation, some preliminaries and the main results of this paper. In \Cref{sec:ErgodicBehavior}, we study the long time averaged of the Hamilton-Jacobi equation without mean field interaction. \Cref{sec:AsymptoticMFG} is devoted to the analysis of the ergodic MFG problem and to the asymptotic behavior of the solution of the time dependent MFG system. 	\\
	
{\bf Acknowledgement.} The first author was partially supported by the ANR (Agence Nationale de la Recherche) project  ANR-12-BS01-0008-01, by the CNRS through the PRC grant 1611 and by the Air Force Office for Scientific Research grant FA9550-18-1-0494. The second author was partly supported by Istituto Nazionale di Alta Matematica (GNAMPA 2019 Research Projects).
	
	\section{Main results}
	\label{sec:MainResults}
	\subsection{Notations and preliminaries}
	
We write below a list of symbols used throughout this paper.
\begin{itemize}
	\item Denote by $\mathbb{N}$ the set of positive integers, by $\mathbb{R}^d$ the $d$-dimensional real Euclidean space,  by $\langle\cdot,\cdot\rangle$ the Euclidean scalar product, by $|\cdot|$ the usual norm in $\mathbb{R}^d$, and by $B_{R}$ the open ball with center $0$ and radius $R$.
	\item If $\Lambda$ is a real $n\times n$ matrix, we define the norm of $\Lambda$ by 
\[
\|\Lambda\|=\sup_{|x|=1,\ x\in\mathbb{R}^d}|\Lambda x|.
\]
\end{itemize}

Let $(X,{\bf d})$ be a metric space (in the paper, we use $X=\T^d\times \R^d$ or $X= \T^d\times \R^d\times \R^d$).
\begin{itemize}
\item For a Lebesgue-measurable subset $A$ of $X$, we let $\mathcal{L}(A)$ be the Lebesgue measure of $A$ and  $\mathbf{1}_{A}:X \rightarrow \{0,1\}$ be the characteristic function of $A$, i.e.,
\begin{align*}
\mathbf{1}_{A}(x)=
\begin{cases}
1  \ \ \ &x\in A,\\
0 &x \not\in A.
\end{cases}
\end{align*} 
We denote by $L^p(A)$ (for $1\leq p\leq \infty$) the space of Lebesgue-measurable functions $f$ with $\|f\|_{p,A}<\infty$, where   
\begin{align*}
& \|f\|_{\infty, A}:=\esssup_{x \in A} |f(x)|,
\\& \|f\|_{p,A}:=\left(\int_{A}|f|^{p}\ dx\right)^{\frac{1}{p}}, \quad 1\leq p<\infty.
\end{align*}
For brevity, $\|f\|_{\infty}$  and $\|f\|_{p}$ stand for  $\|f\|_{\infty,X}$ and  $\|f\|_{p,X}$ respectively.
\item $C_b(X)$ stands for the function space of bounded uniformly  continuous functions on $X$. $C^{2}_{b}(X)$ stands for the space of bounded functions on $X$ with bounded uniformly continuous first and second derivatives. 
$C^k(X)$ ($k\in\mathbb{N}$) stands for the function space of $k$-times continuously differentiable functions on $X$, and $C^\infty(X):=\cap_{k=0}^\infty C^k(X)$. 
 $C_c^\infty(X)$ stands for the space of functions in $C^\infty(X)$ with compact support. Let $a<b\in\mathbb{R}$.
  $AC([a,b];X)$ denotes the space of absolutely continuous maps $[a,b]\to X$.
  \item For $f \in C^{1}(X)$, the gradient of $f$ is denoted by $Df=(D_{x_{1}}f, ..., D_{x_{n}}f)$, where $D_{x_{i}}f=\frac{\partial f}{\partial x_{i}}$, $i=1,2,\cdots,d$.
Let $k$ be a nonnegative integer and let $\alpha=(\alpha_1,\cdots,\alpha_d)$ be a multiindex of order $k$, i.e., $k=|\alpha|=\alpha_1+\cdots +\alpha_d$ , where each component $\alpha_i$ is a nonnegative integer.   For $f \in C^{k}(X)$,
define $D^{\alpha}f:= D_{x_{1}}^{\alpha_{1}} \cdot\cdot\cdot D^{\alpha_{d}}_{x_{d}}f$. 
\end{itemize}

We recall here the notations and definitions of Wasserstein spaces and Wasserstein distance, for more details we refer to \cite{bib:CV, bib:AGS}. Here again we denote by  $(X,{\bf d})$  a metric space (having in mind $X=\T^d\times \R^d$ or $X= \T^d\times \R^d\times \R^d$). 
Denote by $\mathcal{B}(X)$ the  Borel $\sigma$-algebra on $X$ and by $\mathcal{P}(X)$ the space of Borel probability measures on $X$.
The support of a measure $\mu \in \mathcal{P}(X)$, denoted by $\supp(\mu)$, is the closed set defined by
\begin{equation*}
\supp (\mu) := \Big \{x \in X: \mu(V_x)>0\ \text{for each open neighborhood $V_x$ of $x$}\Big\}.
\end{equation*}
We say that a sequence $\{\mu_k\}_{k\in\mathbb{N}}\subset \mathcal{P}(X)$ is weakly-$*$ convergent to $\mu \in \mathcal{P}(X)$, denoted by
$\mu_k \stackrel{w^*}{\longrightarrow}\mu$,
  if
\begin{equation*}
\lim_{n\rightarrow \infty} \int_{X} f(x)\,\mu_n(dx)=\int_{X} f(x) \,\mu(dx), \quad  \forall f \in C_b(X).
\end{equation*}

For $p\in[1,+\infty)$, the Wasserstein space of order $p$ is defined as
\begin{equation*}
\mathcal{P}_p(\mathbb{R}^d):=\left\{m\in\mathcal{P}(\mathbb{R}^d): \int_{\mathbb{R}^d} {\bf d}(x_0,x)^p\,m(dx) <+\infty\right\},
\end{equation*}
for some (and thus all) $x_0 \in X$.
Given any two measures $m$ and $m^{\prime}$ in $\mathcal{P}_p(X)$,  define
\begin{equation}\label{def.transportplan}
\Pi(m,m'):=\Big\{\lambda\in\mathcal{P}(X \times X): \lambda(A\times X)=m(A),\ \lambda(X \times A)=m'(A),\ \forall A\in \mathcal{B}(X)\Big\}.
\end{equation}
The Wasserstein distance of order $p$ between $m$ and $m'$ is defined by
    \begin{equation*}\label{dis1}
          d_p(m,m')=\inf_{\lambda \in\Pi(m,m')}\left(\int_{X\times X}d(x,y)^p\,\lambda(dx,dy) \right)^{1/p}.
    \end{equation*}
    The distance $d_1$ is also commonly called the Kantorovich-Rubinstein distance and can be characterized by a useful duality formula (see, for instance, \cite{bib:CV})  as follows
\begin{equation}\label{eq:2-100}
d_1(m,m')=\sup\left\{\int_{X} f(x)\,m(dx)-\int_{X} f(x)\,m'(dx) \ |\ f:X\rightarrow\mathbb{R} \ \ \text{is}\ 1\text{-Lipschitz}\right\},
\end{equation}
for all $m$, $m'\in\mathcal{P}_1(X)$.

	\subsection{Calculus of variation with acceleration} 
	In our first main result we study the large time average of an optimal control problem of acceleration. 
	Let $L: \T^{d} \times \R^{d} \times \R^{d} \to \R $ be the Lagrangian function  defined as
	\begin{equation*}
	L(x,v,w)=\frac{1}{2}|w|^{2} + F(x,v)
	\end{equation*}
	where $F: \T^{d} \times \R^{d} \to \R$ satisfies the following assumptions:
	\begin{itemize}
	\item[({\bf F1})] $F$ is globally continuous with respect to both variables;
	\item[({\bf F2})] there exists $\alpha > 1$ and there exists a constant $c_{F} \geq 1$ such that for any $(x,v) \in \R^{d} \times \R^{d}$
	\begin{equation}\label{eq:Fgrowth}
	\frac{1}{c_{F}}|v|^{\alpha}-c_{F} \leq F(x,v) \leq c_{F}(1+|v|^{\alpha})
	\end{equation}
	and, without loss of generality, we assume $F(x,v) \geq 0$ for an $(x,v) \in \T^{d} \times \R^{d}$;
	\item[({\bf F3})] there exists a constant $C_{F} \geq 0$ such that 
$$
	|D_{x}F(x,v)| + |D_{v}F(x,v)| \leq C_{F}(1+|v|^{\alpha}).
$$
	\end{itemize}

	Let { $\Gamma$ be the set $C^1$ curves $\gamma:[0,+\infty) \to \T^{d}$ (endowed with the local uniform convergence of the curve and its derivative)} and for $(t,x,v) \in [0,T] \times \T^{d} \times \R^{d}$  let $\Gamma_{t}(x,v)$ be the subset of $\Gamma$ such that $\gamma(t)=x$ and $\dot\gamma(t)=v$. Define the functional $J^{t,T}: \Gamma \rightarrow \R$ as 
	\begin{align}\label{eq:functional}
	\begin{split}
	J^{t,T}(\gamma)=\int_{t}^{T}{\left(\frac{1}{2}|\ddot\gamma(s)|^{2}+F(\gamma(s), \dot\gamma(s)) \right)\ ds}, & \quad \text{if}\ \gamma \in H^{2}(0,T; \T^{d}),
	\end{split}
	\end{align}
 and  $J^{t,T}(\gamma)=+\infty$ if $\gamma \not\in H^{2}(0,T; \T^{d})$, and let $V^{T}(t,x,v)$ denote the value function associated with the functional $J^{t,T}$, i.e.
	\begin{equation}\label{eq:value}
	V^{T}(t,x,v)= \inf_{\gamma 	\in \Gamma_{t}(x,v)} J^{t,T}(\gamma).
	\end{equation}

Let $H$ be the Hamiltonian associated with the Lagrangian $L$, that is for any $(x,v,p_{v}) \in \T^{d} \times \R^{d} \times \R^{d}$,
\begin{equation*}
H(x,v, p_{v})=\frac{1}{2}|p_{v}|^{2}-F(x,v),	
\end{equation*}
where $p_{v} \in \R^{d}$ denotes the momentum variable associated with $v \in \R^{d}$.
Then, it is not difficult to see that the value function $V^{T}$ is a continuous viscosity solution of the following Hamilton-Jacobi equation:
$$
\left\{\begin{array}{l}
\ds -\partial_{t}V^{T}(t,x,v) -\langle D_{x}V^{T}(t,x,v), v \rangle +\frac{1}{2}|D_{v}V^{T}(t,x,v)|^{2}=F(x,v), \quad {\rm in}\; [0,T] \times \T^{d} \times \R^{d}, \\
\ds V^T(T,x,v)= 0 	\quad {\rm in}\;  \T^{d} \times \R^{d}. 
\end{array}\right.
$$
Our aim is to characterize the behavior of $V^T(0, \cdot, \cdot)$ as $T\to +\infty$. To state the result, we need the notion of closed measure, which requires another notation: we set
\begin{align*}
\PP_{\alpha,2}(\T^{d} \times \R^{d} \times \R^{d}) =  \left\{\mu \in \PP(\T^{d} \times \R^{d} \times \R^{d}): \int_{\T^{d} \times \R^{d} \times \R^{d}}{\left(|w|^{2}+|v|^{\alpha} \right)\ \mu(dx,dv,dw)} < + \infty \right\}
\end{align*}
endowed with the weak-$^{*}$ convergence. 
	\begin{definition}[{\bf Closed measure}]\label{def:closedmeasure}
		Let $\eta \in \PP_{\alpha,2}(\T^{d} \times \R^{d} \times \R^{d})$. We say that $\eta$ is a closed measure if for any test function $\varphi \in C^{\infty}_{c}(\T^{d} \times \R^{d})$ the following holds
		\begin{equation*}
			\int_{\T^{d} \times \R^{d} \times \R^{d}}{\Big(\langle D_{x}\varphi(x,v), v \rangle + \langle D_{v}\varphi(x,v), w \rangle \Big)\ \eta(dx,dv,dw)}=0.
		\end{equation*}
		We denote by $\C$ the set of closed measures. 
	\end{definition}

\begin{theorem}[{\bf Main result 1}]\label{theo:main1}
	Assume that $F$ satisfies assumptions {\bf (F1)} and {\bf (F2)}. Then, the following limits exist:
	\begin{equation*}
	\lim_{T \to +\infty} \frac{1}{T}V^{T}(0,x,v)=\lim_{T \to +\infty}  \inf_{\gamma \in \Gamma_{0}(x,v)}\frac{1}{T}J^{T}(\gamma)
	\end{equation*}
	and are independent of  $(x, v) \in \T^{d} \times \R^{d}$. 
	Moreover, if $F$ satisfies also {\bf (F3)} then 
	\begin{equation*}
\lim_{T \to \infty} \frac{1}{T}V^{T}(0,x,v) = \inf_{\mu \in \mathcal{C}} \int_{\T^{d} \times \R^{d} \times \R^{d}}{\left(\frac{1}{2}|w|^{2}+F(x,v) \right)\ \mu(dx,dv,dw)}.
\end{equation*}
\end{theorem}

\begin{remarks}{\rm  \begin{enumerate}
\item If we denote by $\bar \lambda$ the above limits, the convergence of $V^T(0,x,v)-\lambda T$ is a completely open problem  in this context. This is related to the lack of solution of the ergodic HJ equation. 
\item The (strong) structure condition on $L$ and the fact that the problem is periodic in the $x$ variable can probably be  relaxed: this would require however more refined and technical estimates and we have chosen to work in this simpler framework. 
\end{enumerate}}\end{remarks}

	\subsection{Mean Field Games of acceleration}
	
	In our second main result, we consider a mean field game problem of acceleration. The Lagrangian function $L: \T^{d} \times \R^{d} \times \R^{d} \times \PP_{1}(\T^{d} \times \R^{d}) \to \R$ now takes the form
	\begin{equation*}
	L(x,v,w,m)=	\frac{1}{2}|w|^{2}+F(x,v,m)
	\end{equation*}
where $F: \T^{d} \times \R^{d} \times \PP_{1}(\T^{d} \times \R^{d}) \to \R$ satisfies the following assumptions:
\begin{itemize}
	\item[({\bf F1'})] $F$ is globally continuous with respect to all the variables;
	\item[({\bf F2'})] there exists $\alpha > 1$ and a constant $c_{F} \geq 1$ such that for any $(x,v,m) \in \R^{d} \times \R^{d} \times \PP_{1}(\T^{d} \times \R^{d} \times \R^{d})$
$$
	\frac{1}{c_{F}}|v|^{\alpha}-c_{F} \leq F(x,v,m) \leq c_{F}(1+|v|^{\alpha})
$$
	and, without loss of generality, we assume $F(x,v,m) \geq 0$ for any $(x,v,m) \in \T^{d} \times \R^{d} \times \PP_{1}(\T^{d} \times \R^{d} \times \R^{d})$;
{	\item[({\bf F3'})] there exists a constant $C_{F} \geq 0$ such that,  for any $(x,v,m) \in \R^{d} \times \R^{d} \times \PP_{1}(\T^{d} \times \R^{d} \times \R^{d})$, 
$$
	|D_{x}F(x,v,m)| + |D_{v}F(x,v,m)| \leq C_{F}(1+|v|^{\alpha}).
$$}
	\end{itemize}
We consider the time-dependent MFG system 
\begin{align}\label{eq:accMFG}
	\begin{split}
	\begin{cases}
	-\partial_{t}u^{T}(t,x,v)-\langle D_{x}u^{T}(t,x,v),v \rangle  + \frac{1}{2}|D_{v}u^{T}(t,x,v)|^{2} 
		=F(x,v,m^{T}_{t}), &   {\rm in }\;  [0,T] \times \T^{d} \times \R^{d}
	\\
	\partial_{t}m^{T}_{t}-\langle v, D_{x}m^{T}_{t} \rangle- \ddiv\Big(m^{T}_{t}D_{v}u^{T}(t,x,v) \Big)=0, 	 &  {\rm in }\;  [0,T] \times \T^{d} \times \R^{d}
	\\
	 u^{T}(T,x,v)=g(x,v,m^{T}_{T}), \; {\rm in} \; \T^{d} \times \R^{d},  \quad m^{T}_{0}=m_{0} \in \PP(\T^{d} \times \R^{d}).
	\end{cases}
	\end{split}
\end{align}
where the terminal condition of the Hamilton-Jacobi equation satisfies the following:
\begin{itemize}
\item[{\bf (G1)} ] $(x,v) \mapsto g(x,v,m)$ belongs to $C^{1}_{b}(\T^{d} \times \R^{d})$ for any $m \in \PP(\T^{d} \times \R^{d})$ and $m \mapsto g(x,v,m)$ is Lipschitz continuous with respect to the $d_{1}$ distance, uniformly in $(x,v) \in \T^{d} \times \R^{d}$.	
\end{itemize}

We recall that $(u^{T}, m^{T})$ is a solution of \eqref{eq:accMFG} if $u^{T}$ is a viscosity solution of the first equation and $m^{T}$ is a solution in the sense of distributions of the second equation. For more details see \cite{bib:CM, bib:YA}.

Our aim is to understand the averaged limit of $u^T$ as $T\to+\infty$. For this we define the ergodic MFG problem, inspired  by the characterization of the limit in Theorem \ref{theo:main1}. Let us recall that the notion of closed measure was introduced in Definition \ref{def:closedmeasure} and that $\C$ denotes the set of closed measures.

\begin{definition}[{\bf Solution of the ergodic MFG problem}]\label{def.ergoMFGeq}
We say that ${ (\bar\lambda, \bar\mu) \in \R \times \C}$ is a solution of the ergodic MFG problem if
\begin{align}
\bar\lambda =\ & \inf_{\mu \in \C} \int_{\T^{d} \times \R^{d} \times \R^{d}}{\left(\frac{1}{2}|w|^{2} +F(x,v, \pi \sharp \bar\mu)\right)\ \mu(dx,dv,dw)}	\notag
\\
=\ & \int_{\T^{d} \times \R^{d} \times \R^{d}}{\left(\frac{1}{2}|w|^{2} +F(x,v, \pi \sharp \bar{\mu})\right)\ \bar{\mu}(dx,dv,dw)}.\label{eq:ergodic}
\end{align}
\end{definition}

\begin{theorem}[{\bf Main result 2}]\label{thm:theo.main2}
Assume that $F$ and $G$ satisfy {\bf (F1')}, {\bf (F2')} and  {\bf (G1)}. 
\begin{enumerate}
\item 	There exists at least one solution ${ (\bar\lambda, \bar\mu) \in \R \times \C}$ of the ergodic MFG problem~\eqref{eq:ergodic}. 
Moreover, if $F$ satisfies the following  monotonicity assumption:  there exists $M_F>0$ such that for $m_{1}$, $m_{2} \in \PP(\T^{d} \times \R^{d})$
\begin{equation}\label{monotonicity}
\begin{array}{l}
\ds  \int_{\T^{d} \times \R^{d}}{\big(F(x,v,m_{1})-F(x,v,m_{2}) \big) \ (m_{1}(dx,dv)-m_{2}(dx,dv))} 
	\\
\ds \qquad \geq\  M_{F}\int_{\T^{d} \times \R^{d}}{\big(F(x,v,m_{1})-F(x,v,m_{2}) \big)^{2}\ dxdv},
\end{array}
\end{equation}
then the ergodic constant is unique: If $(\bar\lambda_{1}, \bar\mu_{1})$ and $(\bar\lambda_{2}, \bar\mu_{2})$ are two solutions of the ergodic MFG problem, then $\bar\lambda_{1}=\bar\lambda_{2}$.
\item Assume in addition that $\alpha=2$, that {\bf (F3')} and \eqref{monotonicity} hold and that the initial distribution $m_{0}$ is in $\mathcal P_2(\T^d\times \R^d)$. 
Let $(u^{T}, m^{T})$ be a solution of the MFG system \cref{eq:accMFG} and let $(\bar\lambda, \bar\mu)$ be a solution of the ergodic MFG problem \cref{eq:ergodic}. Then $T^{-1}u^T(0, \cdot, \cdot)$ converges locally uniformly to $\bar \lambda$ and we have
	\begin{equation*}
	\lim_{T \to +\infty} \frac{1}{T} \int_{\T^{d} \times \R^{d}}{u^{T}(0,x,v)\ m_{0}(dx,dv)}= \bar\lambda.
	\end{equation*}
\end{enumerate}
\end{theorem}


\section{Ergodic behavior of control of acceleration}\label{sec:ErgodicBehavior}
\subsection{Existence of the limit}

Before proving the main result of this section, \Cref{prop:existslimit}, we need a few preliminary lemmas.

\begin{lemma}\label{lem:upperbo}
Assume that  $F$ satisfies {\bf (F1)} and {\bf (F2)}. Then, for any $(x,v) \in \T^{d} \times B_{R}$, with $R \geq 0$,  and for any $T > 0$, we have
\begin{equation*}
\frac{1}{T}V^{T}(0,x,v) \leq c_{F}(1+R^{\alpha}).	
\end{equation*}
\end{lemma}

{
\begin{remarks} The result also holds when $F=F(t,x,v)$ depends also on time, provided that $F$ is continuous and satisfies {\bf (F2)} with a constant $c_F$ independent of $t$. 
\end{remarks}
}

\begin{proof}
Define the curve
$\xi(t)=x + tv,$ for $t \in [0,T].$
Then, by definition of the value function  $V^T$, we have 
\begin{equation*}
V^{T}(0,x,v) \leq J^{T}(\xi) = \int_{0}^{T}{F(x+tv, v)\ dt}	\leq Tc_{F}(1+R^{\alpha}).
\end{equation*}
\end{proof}

 { 
	\begin{lemma}\label{lem:reachablecurve}
	Assume that $F$ satisfies {\bf (F1)} and {\bf (F2)}. Let $\theta\geq 1$,  $(x_{0}, v_{0})$ and $(x,v)$ be in $\T^{d} \times B_{R}$ for some $R\geq 1$. Then, there exists a constant $C_{2} \geq 0$ (depending only the constants $\alpha$ and $c_F$ in {\bf (F2)}) and a curve $\sigma: [0, \theta] \to \R^{d}$ such that $\sigma(0)=x_{0}$, $\dot\sigma(0)=v_{0}$ and $\sigma(\theta)=x$, $\dot\sigma(\theta)=v$ and 
	\begin{align}\label{eq:boundsigma}
		J^{\theta}(\sigma)\leq C_{2}(R^2\theta^{-1}+ R^\alpha \theta).
	\end{align}
	\end{lemma}
{
\begin{remarks} The result also holds when $F=F(t,x,v)$ depends also on time, provided that $F$ is continuous and satisfies {\bf (F2)} with a constant $c_F$ independent of $t$. 
\end{remarks}
}	
\begin{proof}
	Define the following parametric curve
	\begin{equation*}
	\sigma(t)=x_{0}+v_{0}t+Bt^{2}+Ct^{3}, \quad t \in [0,\theta].
	\end{equation*}
Choosing 
\begin{align*}
\begin{cases}
	& B=3 (x-x_{0})-\theta v -2\theta v_{0})\theta^{-2}
	\\
	& C=(-2(x-x_0)+\theta(v+v_0))\theta^{-3},
\end{cases}	
\end{align*}
we have that $\sigma(0)=x_{0}$, $\dot\sigma(0)=v_{0}$ and $\sigma(1)=x$, $\dot\sigma(1)=v$. 

By definition of the functional $J^{\theta}$ we get
\begin{align*}
J^{\theta}(\sigma) = & \int_{0}^{\theta}{\left(\frac{1}{2}|\ddot\sigma(t)|^{2}+F(\sigma(t), \dot\sigma(t)) \right)\ dt}
\\
\leq & { \int_{0}^{\theta} \left(\frac12 | 2B+6Ct|^2 + c_F(1+ |v_0+ 2tB+3t^2C|^\alpha)\right)dt } \; \leq \; C_2(R^2\theta^{-1}+R^\alpha\theta),
\end{align*}
for some constant $C_2$ depending on the constants $\alpha$ and $c_F$ in {\bf (F2)} only. 
\end{proof}

\begin{lemma}\label{lem.jkhzrnedg} Let $T\geq 2$ and $(x,v)\in \T^d\times B_{R_0}$ for some $R_0\geq c_F^{\frac{2}{\alpha}}$. Let $\gamma\in \Gamma(x,v)$ be optimal for $V^T(0,x,v)$. Then for any $\lambda \geq 2$ there exists $\tilde \gamma\in \Gamma(x,v)$ with $\tilde \gamma(T)=x$, $\dot{\tilde \gamma}(T)=v$ and 
$$
J^T(\tilde \gamma) \leq J^T(\gamma)+ C_3(\lambda^2 R_0^2 + R_0^{\alpha}\lambda^{-\alpha}T ) ,
$$
where the constant $C_3$ depends on $\alpha$ and $c_F$ only. 
\end{lemma}

\begin{remarks}\label{rem.jkhzrnedg}  The result also holds when $F=F(t,x,v)$ depends also on time, provided that $F$ is continuous and satisfies {\bf (F2)} with a constant $c_F$ independent of $t$. In addition, by the construction in the proof, there exists $\tau>0$ such that $\tilde \gamma=\gamma$ on $[0,\tau]$ and 
$$
\int_\tau^T (\frac12 |\ddot{\tilde \gamma}(t)|^2+c_F(1+|\dot{\tilde \gamma}(t)|^\alpha))dt\leq C_3(\lambda^2 R_0^2 + R_0^{\alpha}\lambda^{-\alpha}T ).
$$
Finally, the map which associates $\tilde \gamma$ and $\tau$ to $\gamma$ is measurable. 
\end{remarks}

\begin{proof} Let 
$$
\tau:= \left\{\begin{array}{ll}
\sup\{t\geq0, \; |\dot \gamma(t)|\leq \lambda R_0\} & {\rm if}\; |\gamma(T-1)|>\lambda R_0,\\
T-1 & {\rm otherwise}.
\end{array}\right.
$$
If $\tau\geq T-2$, we set 
$$
\tilde \gamma(t)=  \left\{\begin{array}{ll} 
\gamma(t) & {\rm for }\; t\in [0,\tau],\\ 
\sigma(t-\tau) & {\rm for}\; t\in [\tau, T],
\end{array}\right.
$$
where $\sigma$ is the map built in Lemma \ref{lem:reachablecurve} with $\theta = T-\tau$, $\sigma(0)=\gamma(\tau)$, $\dot \sigma(0)=\dot\gamma(\tau)$,  $\sigma(T-\tau)=x$, $\dot \sigma(T-\tau)=v$. If  $\tau< T-2$, then we set 
$$
\tilde \gamma(t)=  \left\{\begin{array}{ll} 
\gamma(t) & {\rm for }\; t\in [0,\tau],\\ 
\sigma_1(t-\tau) & {\rm for}\; t\in [\tau, \tau+1],\\
\sigma_2(t-\tau-1) & {\rm for}\; t\in [\tau+1, T], 
\end{array}\right.
$$
where $\sigma_1$ and $\sigma_2$ are the map built in Lemma \ref{lem:reachablecurve} with $\theta = 1$, $\sigma_1(0)=\gamma(\tau)$, $\dot \sigma_1(0)=\dot\gamma(\tau)$,  $\sigma_1(1)=x$, $\dot \sigma(1)=v$ and $\theta= T-\tau-1$ and $\sigma_2(0)=\sigma_2(T-\tau-1)=x$ and $\dot \sigma_2(0)=\dot \sigma_2(T-\tau-1)=v$ respectively. Note that $\tilde \gamma(T)=x$ and $\dot{\tilde \gamma}(T)=v$. 

In order to estimate $J^T(\tilde \gamma)$, we first show that $\tau$ cannot be too small: Namely we claim that
\begin{equation}\label{eq:ine3BIS}
\tau \geq T \left(1-\frac{c_{F}(1+R_0^{\alpha})}{\frac{1}{c_{F}}(\lambda R_0)^{\alpha}-c_{F}} \right)-1 .
\end{equation}
Indeed, let us first recall that by \Cref{lem:upperbo} we have
\begin{equation*}
J^{T}(\gamma) \leq c_{F}(1+R_0^{\alpha})T.	
\end{equation*}
On the other hand, by assumption {\bf (F2)} and the fact that $|\dot\gamma(t)|>\lambda R_0$ on $[\tau, T-1]$ and that $F\geq 0$, we also have that
\begin{align*}
J^{T}(\gamma) & = \int_{0}^{T}{\left(\frac{1}{2}|\ddot\gamma(t)|^{2}+F(\gamma(t), \dot\gamma(t)) \right)\ dt}
\\
& \geq \int_{\tau}^{T-1}{\left(\frac{1}{c_{F}}|\dot\gamma(t)|^{\alpha}-c_{F} \right)\ dt}
\geq (T-\tau-1)\left(\frac{1}{c_{F}}(\lambda R_0)^{\alpha}-c_{F} \right).	
\end{align*}
So \eqref{eq:ine3BIS} holds for $R_0\geq c_F^{2/\alpha}$.

We  estimate $J^T(\tilde \gamma)$ in the case $\tau< T-2$, the other case being similar and easier. Note that $|\dot \gamma(\tau)|\leq \lambda R_0$. By Lemma \ref{lem:reachablecurve} and the fact that $F\geq 0$, we have 
\begin{align*}
J^T(\tilde \gamma) &  = \int_0^\tau (\frac12|\ddot \gamma(t)|^2 + F(\gamma(t), \dot \gamma(t)))dt + \int_0^1 (\frac12|\ddot \sigma_1(t)|^2 + F( \sigma_1(t), \dot  \sigma_1(t)))dt\\ & \qquad \qquad + \int_0^{T-\tau-1} (\frac12|\ddot \sigma_2(t)|^2 + F( \sigma_2(t), \dot  \sigma_2(t)))dt \\ 
& \leq J^T(\gamma)+ C_{2}((\lambda R_0)^2+ (\lambda R_0)^\alpha+ R_0^2(T-\tau-1)^{-1}+ R_0^\alpha (T-\tau-1)). 
\end{align*}
In view of \eqref{eq:ine3BIS} this implies that
\begin{align*}
J^T(\tilde \gamma) \leq J^T(\gamma)+ C_3(\lambda^2 R_0^2 + R_0^{\alpha}\lambda^{-\alpha}T ) ,
\end{align*}
for a constant $C_3$ depending on $\alpha$ and $c_F$ only. 
\end{proof}

}

Next we prove that the $(V^T(0, \cdot, \cdot))$ have locally uniformly bounded oscillations:

\begin{lemma}\label{lem.bdoscVT}
There exists a constant $M_{1}(R) \geq 0$ such that for any $(x,v)$ and $(x_{0},v_{0})$ in $\T^{d} \times \overline{B}_{R}$ we have that	
\begin{equation*}
V^{T}(0,x,v) - V^{T}(0,x_{0},v_{0}) \leq M_{1}(R). 	
\end{equation*}
\end{lemma}
\begin{proof}
	Let $\gamma^{*}$ be a minimizer for $V^{T}(0,x_{0},v_{0})$ and let $\sigma: [0,1] \to \T^{d}$ be such that $\sigma(0)=x$, $\dot\sigma(0)=v$ and $\sigma(1)=x_0$, $\dot\sigma(1)=v_0$ as in Lemma \ref{lem:reachablecurve} for $\theta=1$. Define
	\begin{align*}
	\tilde\gamma(t)=
	\begin{cases}
		\sigma(t), & \quad t \in [0,1]
		\\
		\gamma^{*}(t-1), & \quad t \in [1,T]. 
	\end{cases}	
	\end{align*}
Then $\tilde\gamma\in \Gamma_0(x,v)$ and, by Lemma \ref{lem:reachablecurve} and the assumption that $F\geq 0$, we have that
\begin{align*}
& V^{T}(0,x,v)-V^{T}(0,x_{0},v_{0}) \leq \int_{0}^{1}{\left(\frac{1}{2}|\ddot\sigma(t)|^{2}+F(\sigma(t), \dot\sigma(t)) \right)\ dt}	
\\
&\qquad \qquad + \int_{1}^{T}{\left(\frac{1}{2}|\ddot\gamma^{*}(t-1)|^{2} + F(\gamma^{*}(t-1), \dot\gamma^{*}(t-1)) \right)\ dt}-V^{T}(0,x_{0},v_{0})
\\
& \qquad \qquad  \leq 2C_{2}R^2 + \int_{0}^{T-1}{\left(\frac{1}{2}|\ddot\gamma^{*}(t)|^{2} + F(\gamma^{*}(t), \dot\gamma^{*}(t)) \right)\ dt}-V^{T}(0,x_{0},v_{0})
\\ 
&\qquad \qquad   \leq 2C_{2}R^2 - \int_{T-1}^T{\left(\frac{1}{2}|\ddot\gamma^{*}(t)|^{2} + F(\gamma^{*}(t), \dot\gamma^{*}(t)) \right)\ dt} \; \leq \; 2C_{2}R^2,
\end{align*}
which is the claim. 
\end{proof}


	\begin{proposition}[{\bf Existence of the limit}]\label{prop:existslimit}
	Assume that $F$ satisfies {\bf (F1)} and {\bf (F2)}. Then, for any $(x, v) \in \T^{d} \times \R^{d}$, the following limits exist:
	\begin{equation*}
	\lim_{T \to +\infty} \frac{1}{T}V^{T}(0,x,v)=\lim_{T \to +\infty}\frac{1}{T}  \inf_{\gamma \in \Gamma_{0}(x,v)}J^{T}(\gamma).	
	\end{equation*}
	In addition the convergence is locally uniform in $(x,v)$ and the limit is independent of $(x,v)$. 
	\end{proposition}
\begin{proof} Fix $R_0\geq c_F^{2/\alpha}$ such that $|v|\leq R_0$. 
Let $\{T_n\}_{n \in \N}$ and let $\{ \gamma_{n}\}_{n \in \N}$ be a sequence of minimizers for $V^{T_{n}}(0,x,v)$  such that $T_{n} \to \infty$ as $n \to \infty$ and 
\begin{equation*}
\liminf_{T \to \infty}\frac{1}{T}V^{T}(0,x,v)=\lim_{n \to \infty} \frac{1}{T_{n}}J^{T_{n}}(\gamma_{n}).	
\end{equation*}
{ For $\lambda\geq 2$, let us define $\tilde \gamma_n$ is in Lemma \ref{lem.jkhzrnedg}. Then we know that  $\tilde \gamma_n(T)=x$, $\dot{\tilde \gamma}_n(T)=v$ and 
\begin{equation}\label{eq:mainine3BIS}
J^{T_n}(\tilde \gamma_n) \leq J^{T_n}(\gamma_n)+ C_3(\lambda^2 R_0^2 + R_0^{\alpha}\lambda^{-\alpha}T_n ) .
\end{equation}
Let us define $\hat\gamma_{n}$ as the periodic extension of the curve $\tilde\gamma_{n}$, i.e. $\hat\gamma_{n}$ is $T_n$-periodic and it is equal to $\tilde\gamma_{n}$ on $[0, T_n]$.
Then, taking $\hat\gamma_{n}$ as competitors for $J^{T}$ we obtain that
\begin{align*}
& \limsup_{T \to \infty}\inf_{\gamma \in \Gamma_{0}(x,v)}\frac{1}{T} J^{T}(\gamma) \leq \limsup_{T \to \infty}\frac{1}{T}J^{T}(\hat\gamma_{n})
\\
&\qquad =  \frac{1}{T_n}J^{T_n}(\tilde\gamma_{n}) \leq \left( \frac{1}{T_{n}}J^{T_{n}}(\gamma_{n}) + C_3(\lambda^2R_0^2T_n^{-1} + R_0^{\alpha}\lambda^{-\alpha})\right),
\end{align*}
where the equality holds true since we are taking the limit of a periodic function and the last inequality holds by \eqref{eq:mainine3BIS}. 

We get the conclusion letting $n \to \infty$ and then $\lambda \to \infty$, indeed: as $n \to \infty$ we deduce that
\begin{align*}
\limsup_{T \to \infty}\inf_{\gamma \in \Gamma_{0}(x, v)}\frac{1}{T}J^{T}(\gamma)&  \leq \lim_n  \frac{1}{T_{n}}J^{T_{n}}(\gamma_{n}) +C_3 R_0^{\alpha}\lambda^{-\alpha}\\  
& = 	
 \liminf_{T\to+\infty} \inf_{\gamma \in \Gamma_{0}(x, v)}\frac{1}{T} J^{T}(\gamma)+	C_3R_0^{\alpha}\lambda^{-\alpha}
\end{align*}
and then, taking the limit as $\lambda \to \infty$ we get
\begin{equation*}
	\limsup_{T \to \infty}\inf_{\gamma \in \Gamma_{0}(x, v)}\frac{1}{T}J^{T}(\gamma) \leq \liminf \inf_{\gamma \in \Gamma_{0}(x, v)}\frac{1}{T}J^{T}(\gamma).
\end{equation*}
As the $(V^T(0, \cdot,\cdot))$ have locally bounded oscillation (Lemma \ref{lem.bdoscVT}), the above convergence is locally uniform and the limit does not depend on $(x,v)$. }
\end{proof}

\subsection{Characterization  of the ergodic limit}

{ In this part we characterize the limit given in Proposition~\ref{prop:existslimit} in term of closed measures. The proof of the main result, Proposition \ref{prop:caract}, where this characterization is stated, is technical and requires several steps. Here are the main ideas of the proof. By using standard results on occupational measures, one can obtain in a relatively easy way that 
\begin{align*}
\lambda:= &\lim_{T \to \infty} \inf_{\gamma \in \Gamma_{0}(x_{0},v_{0})}\frac{1}{T}J^{T}(\gamma) 
\geq\  \inf_{\mu\in {\mathcal C}} \int_{\T^{d} \times \R^{d} \times \R^{d}}{\left(\frac{1}{2}|w|^{2}+F(x,v) \right)\ \mu(dx,dv,dw)}, 
\end{align*}
where $\mathcal{C}$ denotes the set of closed probability measures (see Definition \ref{def:closedmeasure}). The difficult part of the proof is the opposite inequality. The first step for this is a min-max formula (Theorem \ref{thm:minmaxform}) which gives, by  using the characterization of closed measures, that 
	\begin{align*}
	\begin{split}
	& \inf_{\mu \in \C} \int_{\T^{d} \times \R^{d} \times \R^{d}}{\left(\frac{1}{2}|w|^{2}+F(x,v) \right)\ \mu(dx,dv,dw)}
	\\
	& \qquad = \sup_{\varphi \in C^{\infty}_{c}(\T^{d} \times \R^{d})} \inf_{(x,v) \in \T^{d} \times \R^{d}}\left\{-\frac{1}{2}|D_{v}\varphi(x,v)|^{2}-\langle D_{x}\varphi(x,v), v \rangle +F(x,v)\right\}.
	\end{split}
	\end{align*}
In order to exploit this inequality, one just needs to find a map $\varphi\in C^{\infty}_{c}(\T^{d} \times \R^{d})$ for which 
$$
-\frac{1}{2}|D_{v}\varphi(x,v)|^{2}-\langle D_{x}\varphi(x,v), v \rangle +F(x,v)
$$
is almost equal to $\lambda$. This is not easy because the corrector of our ergodic problem does not seem to exist (at least in the usual sense) because of the lack of controllability and, if it existed, it certainly would not be smooth with a compact support. The standard idea in this set-up is to use instead the approximate corrector, i.e., the solution $V_\delta$ to 
$$
	\delta V_{\delta}(x,v)+\frac{1}{2}|D_{v}V_{\delta}(x,v)|^{2}+\langle D_{x}V_{\delta}(x,v), v \rangle = F(x,v)\qquad {\rm in }\; \T^d\times \R^d.
$$
However, this approximate corrector  has not a compact support either (it is even coercive, see Proposition \ref{prop.reguVdelta}) and $\delta V_\delta$ does not converge uniformly to $-\lambda$, but only locally uniformly. We overcome these issues by an extra approximation argument (Lemma \ref{lem:approx1}). \\
}

Let us first explain why closed measures pop up naturally in our problem. To see this, let $(x_{0}, v_{0}) \in \T^{d} \times \R^{d}$ be an initial position  { and let $\gamma^T_{(x_0,v_0)}$ be an optimal trajectory for $V^{T}(0,x_{0},v_{0})$.} We define the family of Borel probability measures $\{ \mu_{T}\}_{T > 0}$ as follows: for any function $\varphi \in C^{\infty}_{c}(\T^{d} \times \R^{d} \times \R^{d})$ 
\begin{equation}\label{eq:uniform}
	\int_{\T^{d} \times \R^{d} \times \R^{d}}{\varphi(x,v,w)\ \mu^{T}(dx,dv,dw)}=\frac{1}{T}\int_{0}^{T}{\varphi(\gamma^T_{(x_{0},v_{0})}(t), \dot\gamma^T_{(x_{0},v_{0})}(t), \ddot\gamma^T_{(x_{0},v_{0})}(t))\ dt}.
\end{equation}


\begin{lemma}\label{lem:measureconv}
Assume that $F$ satisfies {\bf (F1)} and {\bf (F2)}. Let the family of probability measures $\{\mu^{T}\}_{T >0}$ be defined by \cref{eq:uniform}. Then, $\{ \mu^{T}\}_{T>0}$ is tight and there exists a closed measure $\mu^{*}$ such that, up to a subsequence, $\mu^{T} \rightharpoonup^{*} \mu^{*}$ as $T \to +\infty$.  
\end{lemma}

\begin{proof}
We first prove that $\{\mu_{T}\}_{T>0}$ its a tight family of probability measures. Indeed, by assumption {\bf (F2)} for $(x_{0},v_{0}) \in \T^{d} \times \R^{d}$ we know that
	\begin{align*}
	\frac{1}{T}V^{T}(0,x_{0},v_{0})	= & \frac{1}{T} \int_{0}^{T}{\Big(\frac{1}{2}|\ddot\gamma^T_{(x_{0},v_{0})}(t)|^{2}+F(\gamma^T_{(x_{0},v_{0})}(t), \dot\gamma^T_{(x_{0},v_{0})}(t))	\Big)\ dt}
	\\
	= & \int_{\T^{d} \times \R^{d} \times \R^{d}}{\Big(\frac{1}{2}|w|^{2}+F(x,v)	\Big)\ \mu^{T}(dx,dv,dw)}
	\\
	\geq &  \int_{\T^{d} \times \R^{d} \times \R^{d}}{\Big(\frac{1}{2}|w|^{2}+\frac{1}{c_{F}}|v|^{\alpha}-c_F\Big)\ \mu^{T}(dx,dv,dw)}.
	\end{align*}
On the other hand, by \Cref{lem:upperbo} we have that
\begin{equation*}
\frac{1}{T}V^{T}(0,x_{0},v_{0}) \leq C_{1}	
\end{equation*}
where $C_{1}$ only depends on the initial point $(x_0,v_0)$. Therefore, we obtain that 
\begin{equation*}
	\int_{\T^{d} \times \R^{d} \times \R^{d}}{\Big(\frac{1}{2}|w|^{2}+\frac{1}{c_{F}}|v|^{\alpha}\Big)\ \mu^{T}(dx,dv,dw)} \leq C_{1}
\end{equation*}
which implies that $\{\mu^{T}\}_{T>0}$ is tight. By Prokhorov theorem there exists a measure $\mu^{*} \in \PP(\T^{d} \times \R^{d} \times \R^{d})$ such that up to a subsequence $\mu^{T} \rightharpoonup^{*} \mu^{*}$ as $T \to +\infty$. 

We now show  that the measure $\mu^{*}$ is closed in the sense of \Cref{def:closedmeasure}.
Let $\varphi \in C^{\infty}_{c}(\T^{d} \times \R^{d})$ be a test function and let $R \geq 0$ be such that $\varphi(x,v)=0$ for any $(x,v) \in \T^{d} \times B^{c}_{R}$. 
	Moreover, define
	\begin{align*}
\tau^{*}=
\begin{cases}
\sup\{t \in [0, T]: |\dot\gamma^T_{(x_{0}, v_{0})}(t)| \leq R\}, \quad & \text{if}\ |\dot\gamma_{(x_{0}, v_{0})}(T)| > R
\\
T, \quad &  \text{if}\ |\dot\gamma_{(x_{0}, v_{0})}(T)| \leq R
\end{cases}
\end{align*}
and let $\sigma^{*}: [\tau^{*}, \tau^{*}+1] \to \T^{d}$ be as in \Cref{lem:reachablecurve} such that $\sigma^{*}(\tau^{*})=\gamma^{T}_{(x_{0}, v_{0})}(\tau^{*})$, $\dot\sigma^{*}(\tau^{*})=\dot\gamma^{T}_{(x_{0}, v_{0})}(\tau^{*})$ and $\sigma^{*}(\tau^{*}+1)=x_0$, $\dot\sigma^{*}(\tau^{*}+1)=v_0$. 
Moreover, define
\begin{align*}
\tilde\gamma(t)=
\begin{cases}
	\gamma^{T}_{(x_{0}, v_{0})}(t), & \quad t \in [0,\tau^{*}]
	\\
	\sigma^{*}(t), & \quad t \in (\tau^{*}, \tau^{*}+1].
\end{cases}	
\end{align*}
Then we get
\begin{align*}
& \int_{\T^{d} \times \R^{d} \times \R^{d}}{\Big(\langle D_{x}\varphi(x,v), v \rangle + \langle D_{v}\varphi(x,v), w \rangle \Big)\ d\mu^{T}(x,v,w)} 
\\
=&\ \frac{1}{T}\int_{0}^{T}{\Big(\langle D_{x}\varphi(\gamma^{T}_{(x_{0}, v_{0})}(t),\dot\gamma^{T}_{(x_{0}, v_{0})}(t)), \dot\gamma^{T}_{(x_{0}, v_{0})}(t) \rangle + \langle D_{v}\varphi(\gamma^{T}_{(x_{0}, v_{0})}(t),\dot\gamma^{T}_{(x_{0}, v_{0})}(t)), \ddot\gamma^{T}_{(x_{0}, v_{0})}(t) \rangle \Big)\ dt}
\\
=&\ \underbrace{\frac{1}{T}\int_{0}^{\tau^{*}+1}{\Big(\langle D_{x}\varphi(\tilde{\gamma}^{T}(t),\dot{\tilde{\gamma}}^{T}(t)), \dot{\tilde{\gamma}}^{T}(t) \rangle + \langle D_{v}\varphi(\tilde{\gamma}^{T}(t),\dot{\tilde{\gamma}}^{T}(t)), \ddot{\tilde{\gamma}}^{T}(t) \rangle \Big)\ dt}}_{\bf{A}}
\\
-&\ \underbrace{\frac{1}{T}\int_{\tau^{*}}^{\tau^{*}+1}{\Big(\langle D_{x}\varphi(\sigma^{*}(t),\dot\sigma^{*}(t)), \dot\sigma^{*}(t) \rangle + \langle D_{v}\varphi(\sigma^{*}(t),\dot\sigma^{*}(t)), \ddot\sigma^{*}(t) \rangle \Big)\ dt}}_{\bf{B}}
\\
+& \underbrace{\int_{\tau^{*}}^{T}{\Big(\langle D_{x}\varphi(\gamma^{T}_{(x_{0}, v_{0})}(t),\dot\gamma^{T}_{(x_{0}, v_{0})}(t)), \dot\gamma^{T}_{(x_{0}, v_{0})}(t) \rangle + \langle D_{v}\varphi(\gamma^{T}_{(x_{0}, v_{0})}(t),\dot\gamma^{T}_{(x_{0}, v_{0})}(t)), \ddot\gamma^{T}_{(x_{0}, v_{0})}(t) \rangle \Big) \ dt}}_{\bf C}
\end{align*}

One can immediately observe that by construction {\bf C}$=0$ (since $\varphi$ has a support in $\T^d\times B_R$). By the  definition of $\tilde{\gamma}$ one also has that ${\bf A}=0$. 
The behavior of {\bf B} is also immediate because, as $\varphi$ is bounded, 
\begin{align*}
	& \frac{1}{T}\int_{\tau^{*}}^{\tau^{*}+1}{\Big(\langle D_{x}\varphi(\sigma^{*}(t),\dot\sigma^{*}(t)), \dot\sigma^{*}(t) \rangle + \langle D_{v}\varphi(\sigma^{*}(t),\dot\sigma^{*}(t)), \ddot\sigma^{*}(t) \rangle \Big)\ dt}
	\\
	=&\ \frac{1}{T}(\varphi(\sigma^{*}(\tau^{*}+1), \dot\sigma^{*}(\tau^{*}+1))-\varphi(\sigma^{*}(\tau^{*}), \sigma^{*}(\tau^{*})) \to 0, \quad \text{as}\ T \to +\infty.
\end{align*}
The proof is thus complete.
\end{proof}

The next step consists in formulating in two different ways the expected limit of Proposition \ref{prop:existslimit}. 

\begin{theorem}[{\bf Minmax formula}]\label{thm:minmaxform}
	Assume that $F$ satisfies {\bf (F1)} and {\bf (F2)}. Then, the following equality holds true:
	\begin{align}\label{eq:minimax}
	\begin{split}
	& \inf_{\mu \in \C} \int_{\T^{d} \times \R^{d} \times \R^{d}}{\left(\frac{1}{2}|w|^{2}+F(x,v) \right)\ \mu(dx,dv,dw)}
	\\
	& \qquad = \sup_{\varphi \in C^{\infty}_{c}(\T^{d} \times \R^{d})} \inf_{(x,v) \in \T^{d} \times \R^{d}}\left\{-\frac{1}{2}|D_{v}\varphi(x,v)|^{2}-\langle D_{x}\varphi(x,v), v \rangle +F(x,v)\right\}.
	\end{split}
	\end{align}
\end{theorem}

\begin{proof}
By definition of a closed measure we can write
\begin{align*}
& \inf_{ \mu \in \C}	\int_{\T^{d} \times \R^{d} \times \R^{d}}{\left(\frac{1}{2}|w|^{2}+F(x,v) \right)\ \mu(dx,dv,dw)} 
\\
= & \inf_{\mu \in \PP_{2, \alpha}(\T^{d} \times \R^{d} \times \R^{d})} \sup_{\varphi \in C^{\infty}_{c}(\T^{d} \times \R^{d})} \int_{\T^{d} \times \R^{d} \times \R^{d}}{\Big(\frac{1}{2}|w|^{2}+F(x,v)- \langle D_{x}\varphi(x,v),v \rangle-\langle D_{v}\varphi(x,v), w \rangle\Big)\ \mu(dx,dv,dw)}.
\end{align*}
Our aim is to use the min-max Theorem (see \Cref{thm:minmax} below). We use for this 
the notation introduced in Appendix A and set $\mathbb{A}= C^{\infty}_{c}(\T^{d} \times \R^{d})$, { $\mathbb{B}=\PP_{2, \alpha}(\T^{d} \times \R^{d} \times \R^{d})$} and for any $(\varphi, \mu) \in \mathbb{A} \times \mathbb{B}$
\begin{equation*}
\mathcal{L}(\varphi, \mu): = 	\int_{\T^{d} \times \R^{d} \times \R^{d}}{\Big(\frac{1}{2}|w|^{2}+F(x,v)-\langle D_{x}\varphi(x,v),v \rangle-\langle D_{v}\varphi(x,v), w \rangle \Big)\ \mu(dx,dv,dw)}.
\end{equation*}
Let us choose $\varphi^{*}(x,v)=0$ and 
{  
\begin{align*}
c^{*} & =1+ \inf_{\mu \in \PP_{2, \alpha}(\T^{d} \times \R^{d} \times \R^{d})} \sup_{\varphi \in C^{\infty}_{c}(\T^{d} \times \R^{d})} \int_{\T^{d} \times \R^{d} \times \R^{d}}{\Big(\frac{1}{2}|w|^{2}+F(x,v)}  
\\
&\qquad \qquad\qquad\qquad\qquad  -\ \langle D_{x}\varphi(x,v),v \rangle-\langle D_{v}\varphi(x,v), w \rangle \Big)\ \mu(dx,dv,dw).	
\end{align*} }
Note that $c^{*}$ is finite (since it is bounded below by assumption \cref{eq:Fgrowth} and bounded above for $\mu=\delta_{(x_0,0,0)}$ for any $x_0\in \T^d$). In addition, the set $\mathbb{B}^{*}=\left\{\mu \in \mathbb{B}: \mathcal{L}(\varphi^*, \mu) \leq c^{*} \right\}$ is nonempty and tight, and thus compact, in $\PP_{2, \alpha}(\T^{d} \times \R^{d} \times \R^{d})$ for the weak-$*$ convergence. Finally, we have  
\begin{align*}
 c^{*} & \geq 1+ \sup_{\varphi \in C^{\infty}_{c}(\T^{d} \times \R^{d})}\inf_{\mu \in \PP_{2, \alpha}(\T^{d} \times \R^{d} \times \R^{d})} \int_{\T^{d} \times \R^{d} \times \R^{d}}{\Big(\frac{1}{2}|w|^{2}+F(x,v)}
\\
&\qquad \qquad \qquad \qquad  -\ \langle D_{x}\varphi(x,v),v \rangle-\langle D_{v}\varphi(x,v), w \rangle \Big)\ \mu(dx,dv,dw).
\end{align*}
%
Therefore, the min-max Theorem \ref{thm:minmax} states that
\begin{align*}
 & \inf_{\mu \in \PP_{2, \alpha}(\T^{d} \times \R^{d} \times \R^{d})} \sup_{\varphi \in C^{\infty}_{c}(\T^{d} \times \R^{d})} \int_{\T^{d} \times \R^{d} \times \R^{d}}{\Big(\frac{1}{2}|w|^{2}+F(x,v)-  \langle D_{x}\varphi(x,v),v \rangle-\langle D_{v}\varphi(x,v), w \rangle \Big)\ \mu(dx,dv,dw)}
\\
&= \sup_{\varphi \in C^{\infty}_{c}(\T^{d} \times \R^{d})}\inf_{\mu \in \PP_{2}(\T^{d} \times \R^{d} \times \R^{d})} \int_{\T^{d} \times \R^{d} \times \R^{d}}{\Big(\frac{1}{2}|w|^{2}+F(x,v)- \langle D_{x}\varphi(x,v),v \rangle-\langle D_{v}\varphi(x,v), w \rangle \Big)\ \mu(dx,dv,dw)}
\end{align*}
\begin{align*}
 = &\sup_{\varphi \in C^{\infty}_{c}(\T^{d} \times \R^{d})} \inf_{(x,v,w) \in \T^{d} \times \R^{d} \times \R^{d}}{\left\{\frac{1}{2}|w|^{2}+F(x,v)-\langle D_{x}\varphi(x,v),v \rangle-\langle D_{v}\varphi(x,v), w \rangle \right\}}
 \\
 =& \sup_{\varphi \in C^{\infty}_{c}(\T^{d} \times \R^{d})} \inf_{(x,v) \in \T^{d} \times \R^{d} } \left\{-\frac{1}{2}|D_{v}\varphi(x,v)|^{2}-\langle D_{x}\varphi(x,v), v \rangle +F(x,v) \right\}.
\end{align*}
This complete the proof. 
\end{proof}

Next we introduce and study the discounted problem associated with \cref{eq:functional}. For any $\delta > 0$ and any $(x,v) \in \T^{d} \times \R^{d}$ we define $J_{\delta}: \Gamma \to \R\cup\{+\infty\}$ as 
\begin{equation*}
J_{\delta}(\gamma)= \int_{0}^{+\infty}{e^{-\delta t}\left(\frac{1}{2}|\ddot\gamma(t)|^{2}+F(\gamma(t), \dot\gamma(t)) \right) dt}	
\end{equation*}
if $\dot  \gamma$ is absolutely continuous with $\int_{0}^{+\infty}e^{-\delta t}\left(\frac{1}{2}|\ddot\gamma(t)|^{2}+|\dot\gamma(t))|^\alpha \right) dt<+\infty$, and $J_{\delta}(\gamma)=+\infty$ otherwise. 
We define the associated value function (the approximate corrector) 
\begin{equation}\label{eq:discounted}
V_{\delta}(x,v)=\inf_{\gamma \in \Gamma_{0}(x,v)}J_{\delta}(\gamma). 	
\end{equation}
We recall that $V_{\delta}$ is the unique continuous viscosity solution with a polynomial growth of the following Hamilton-Jacobi equation 
\begin{equation}\label{eq:discoundedHJ}
	\delta V_{\delta}(x,v)+\frac{1}{2}|D_{v}V_{\delta}(x,v)|^{2}+\langle D_{x}V_{\delta}(x,v), v \rangle = F(x,v).
\end{equation}

As the convergence of $V^T(0, \cdot,\cdot)/T$ is locally uniform (by Proposition \ref{lem.bdoscVT}), we can apply the Abelian-Tauberian Theorem of \cite{bib:MG} and we have that for any $(x,v) \in \T^{d} \times \R^{d}$
\begin{equation}\label{eq:tauberian}
\lim_{\delta \to 0^{+}} \delta V_{\delta}(x,v)=\lim_{T \to \infty} \frac{1}{T} V^{T}(0,x,v)=: \lambda.
\end{equation}

In the proof of the main result of this section (Proposition \ref{prop:caract}) we will have to smoothen the map $V^\delta$. This  involves some local regularity properties of $V^\delta$, which is the aim of the next result. 

\begin{proposition}\label{prop.reguVdelta}
	Assume that $F$ satisfies {\bf (F1)} -- {\bf (F3)}. Then, we have:
		\begin{itemize}
	\item[($i$)] $\{ \delta V_{\delta}(x,v)\}_{\delta > 0}$ is locally uniformly bounded;
	\item[($ii$)] $\{ V_{\delta}(x,v)\}_{\delta >0}$ has locally uniformly bounded oscillation, i.e. there exists a constant $M(R) \geq 0 $ such that for any $(x_{0}, v_{0}) , (x,v) \in \T^{d} \times \overline{B}_{R}$
	\begin{equation*}
	V_{\delta}(x,v) - V_{\delta}(x_{0},v_{0}) \leq M(R).
	\end{equation*}
	\item[($iii$)] there exists a constant $\tilde{C} \geq 0$ such that for any $(x,v) \in \T^{d} \times \R^{d}$
	\begin{equation}\label{eq:est2}
	\tilde{C}^{-1}|v|^{\alpha}-\tilde{C}\delta^{-1}\leq V_{\delta}(x,v) \leq c_F{\color{red} \delta^{-1}} (|v|^\alpha+1);	
	\end{equation}
\item[($iv$)] the map $x \mapsto V_{\delta}(x,v)$ is locally Lipschitz continuous and there exists a constant $C_{\delta} \geq 0$ such that for a.e. $(x,v) \in \T^{d} \times \R^{d}$ the following holds:
\begin{equation}\label{eq:est3}
	|D_{x}V_{\delta}(x,v)| \leq C_{\delta}(1+|v|^{\alpha}).
\end{equation}
	   \end{itemize}
	  \end{proposition}
	
	\begin{proof}
	\begin{itemize}
	\item[($i$)] 	 
	Fix $(x,v) \in \T^{d} \times \overline{B}_{R}$ and define a competitor $\gamma: [0,+\infty] \to \T^{d}$ such that $\gamma(t)=x+tv$. By definition and \cref{eq:Fgrowth} we get
	\begin{equation*}
		\delta V_{\delta}(x,v) \leq \delta \int_{0}^{\infty}{e^{-\delta t} F(\gamma(t), \dot\gamma(t))\ ds} \leq c_{F}(1+|v|^{\alpha})\; \leq \; c_{F}(1+R^{\alpha}).
	\end{equation*}
On the other hand,  we have by {\bf (F2)} that $F \geq 0$ and thus $V_\delta\geq0$, which completes the proof of ($i$).

\item[($ii$)]   Let $(x_{0},v_{0})$, $(x,v) \in \T^{d} \times \overline{B}_{R}$ be fixed points, let $\gamma^{*}$ be a minimizer for $V_\delta(x_0,v_0)$ and let $\sigma$ be defined as in \Cref{lem:reachablecurve} such that  $\sigma(0)=x$, $\dot\sigma(0)=v$ and $\sigma(1)=x_{0}$, $\dot\sigma(1)=v_{0}$. 
We define a new curve $\gamma: [0,+\infty) \to \T^{d}$ as follows
\begin{align*}
\gamma(t)=
\begin{cases}
	\sigma(t), & \quad t \in [0, 1]
	\\
	\gamma^{*}(t-1), & \quad t \in (1, +\infty).
\end{cases}
\end{align*}
Then 
	\begin{align}\label{eq:01}
	\begin{split}
	 V_{\delta}(x,v)-V_{\delta}(x_{0}, v_{0})\leq &\int_{0}^{1}{e^{-\lambda t}\left(\frac{1}{2}|\ddot\gamma(t)|^{2}+F(\gamma(t), \dot\gamma(t)) \right)\ dt}
	\\
	+ & \int_{1}^{+\infty}{e^{-\lambda t}\left(\frac{1}{2}|\ddot\gamma(t)|^{2} + F(\gamma(t), \dot\gamma(t))	\right)\ dt} -  V_{\lambda}(x_{0}, v_{0}).
	\end{split}
	\end{align}
By a change of variable, we have that
\begin{align*}
	& \int_{1}^{+\infty}{e^{-\delta t}\left(\frac{1}{2}|\ddot\gamma^{*}(t)|^{2} + F(\gamma^{*}(t), \dot\gamma^{*}(t))	\right)\ dt} 
	\\
	= &  e^{-\delta} \int_{0}^{\infty}{e^{-\delta s}\left(\frac{1}{2}|\ddot\gamma^{*}(s)|^{2} + F(\gamma^{*}(s), \dot\gamma^{*}(s)) \right)\ ds}=e^{-\delta} V_{\delta}(x_{0}, v_{0}). 
\end{align*}
	Therefore, we obtain that 
	\begin{align}\label{eq:1}
	\begin{split}
		& \left|\int_{1}^{+\infty}{e^{-\delta t}\left(\frac{1}{2}|\ddot\gamma(t)|^{2} + F(\gamma(t), \dot\gamma(t))	\right)\ dt} -V_{\delta}(x_{0}, v_{0})\right| \leq \left|e^{-\delta}-1\right|V_{\delta}(x_{0}, v_{0})
		\\
		& \qquad \leq\ \delta |V_{\delta}(x_{0}, v_{0})| \leq\  c_F(1+R^{\alpha}),
		\end{split}
	\end{align}
where the last inequality holds true by  ($i$). Moreover, by construction of $\sigma$ in \Cref{lem:reachablecurve} we have that
	\begin{align}\label{eq:11}
	\int_{0}^{1}{e^{-\delta t}\left(\frac{1}{2}|\ddot\sigma(t)|^{2}+F(\sigma(t), \dot\sigma(t)) \right)\ dt} \leq  J^{1}(\sigma) \leq C_{2}(R^2+R^\alpha). 
	\end{align}
	Combining together inequality \cref{eq:1} and \cref{eq:11} in \cref{eq:01} we get (ii):
	\begin{equation*}
	V_{\delta}(x,v)-V_{\delta}(x_{0}, v_{0}) \leq c_{F}(1+R^{\alpha})+C_{2}(R^2+R^\alpha)=:M(R).
	\end{equation*}

\item[($iii$)] For some constants $M_{1}$ and $M_{2}$ we have that the map $Z : \T^{d} \times \R^{d} \to \R$ such that $Z(x,v)=M_{1}^{-1}|v|^{\alpha}-M_{2}\delta^{-1}$ is a subsolution of \cref{eq:discoundedHJ}, indeed
	\begin{align*}
	& \delta Z(x,v) + \frac{1}{2}|D_{v}Z(x,v)|^{2}+\langle D_{x}Z(x,v), v \rangle - F(x,v) 
	\\
	& \qquad \leq  \delta M_{1}^{-1}|v|^{\alpha}-M_{2}+\frac{1}{2}M_{1}^{-2}\alpha^{2}|v|^{2(\alpha - 1)}	-c_{F}^{-1}|v|^{\alpha}+c_{F}.
	\end{align*}
As $2(\alpha-1) \leq \alpha$, since $\alpha \in (1, 2]$, we get, for $M_1$ and $M_2$ large enough,
\begin{equation*}
\delta Z(x,v) + \frac{1}{2}|D_{v} Z(x,v)|^{2}+\langle D_{x}Z(x,v), v \rangle -F(x,v) \leq 0.	
\end{equation*}
By comparison  we obtain $V_{\delta} \geq Z$, which proves the first inequality in \cref{eq:est2}.   

{ In the same way, considering the map $Z(x,v)=c_F\delta^{-1}(|v|^{\alpha}+1)$, we have 
	\begin{align*}
	& \delta Z(x,v) + \frac{1}{2}|D_{v}Z(x,v)|^{2}+\langle D_{x}Z(x,v), v \rangle - F(x,v) 
	\\
	& \qquad \geq   c_F(|v|^{\alpha}+1)+\frac{1}{2}\delta^{-2}(c_F\alpha)^{2}|v|^{2(\alpha - 1)}	-c_{F}|v|^{\alpha}-c_{F} \ \geq \ 0,
	\end{align*}
	so that $Z$ is a supersolution. By comparison we conclude that the second inequality in \eqref{eq:est2} holds. 
}

\item[($iv$)] 
Let $\gamma^{*}$ be optimal for $V_\delta(x,v)$ and let $h \in \R^{d}$. Then
\begin{align}\label{eq:devidelta}
\begin{split}
 V_{\delta}(x+h, v) & \leq \int_{0}^{+\infty}{e^{-\delta t} \left(\frac{1}{2}|\ddot\gamma^{*}(t)|^{2}+F(\gamma^{*}(t)+h, \dot\gamma^{*}(t)) \right)\ dt}
\\
& \leq\  V_{\delta}(x,v)+\int_{0}^{+\infty}{e^{-\delta t}\left(F(\gamma^{*}(t)+h, \dot\gamma^{*}(t))-F(\gamma^{*}(t), \dot\gamma^{*}(t)) \right)\ dt}	
\\
& \leq\  V_{\delta}(x,v) + \int_{0}^{+\infty}{e^{-\delta t}c_{F}(1+|\dot\gamma^{*}(t)|^{\alpha})|h|\ dt},
\end{split}
\end{align}
where the last inequality holds true by assumption {\bf (F3)}. Moreover, by \cref{eq:est2} we deduce that there exists a constant $C_{\delta} \geq 0$ such that
\begin{equation*}
\int_{0}^{+\infty}{e^{-\delta t}(c_F^{-1}|\dot\gamma^{*}(t)|^{\alpha}-c_F)\ dt} \leq  V_\delta(x,v)\leq C_{\delta}(1+|v|^{\alpha}).	
\end{equation*}
Therefore,  by \cref{eq:devidelta} we deduce that
\begin{equation*}
V_{\delta}(x+h,v)-V_{\delta}(x,v)\leq C_{\delta}(1+|v|^{\alpha})|h|,
\end{equation*}
which implies that $V_{\delta}$ is locally Lipschitz continuous in space and  proves (iv). 
	\end{itemize}
\end{proof}

We now strengthen a little the convergence in \eqref{eq:tauberian}: 

\begin{proposition}\label{prop:Inflemma}
Assume that $F$ satisfies {\bf (F1)}---{\bf (F3)}. Then
\begin{equation*}
\lambda=\lim_{\delta \to 0^{+}}  \inf_{(x,v) \in \T^{d} \times \R^{d}} \delta V_{\delta}(x,v),
\end{equation*}
with $\lambda$ defined in \cref{eq:tauberian}. 
\end{proposition}
\begin{proof} { First we note that, by (i) in Proposition \ref{prop.reguVdelta}, the convergence in \eqref{eq:tauberian} is locally uniform. Fix  $R \geq 0$  such that
\begin{equation}\label{defRzjkdn}
c_{F}^{-1}R^{\alpha}-c_{F} > \lambda. 	
\end{equation}
Then,  for any $\eps> 0$, there exists $\delta_{\eps} >0 $ such that for any $\delta \in(0, \delta_{\eps})$ we have that 
\begin{equation}\label{eq:tauberianbis}
\inf_{(x,v)\in \T^d\times B_R} \delta V_{\delta} (x,v)   \geq \lambda-\eps.
\end{equation}
Fix $(x,v) \in \T^{d} \times \R^{d}$ and let $\gamma^{*}_\delta$ be a minimizer for $V_{\delta}(x,v)$. We define
\begin{align*}
\tau_\delta=
\begin{cases}
	 \inf\{t \in [0, +\infty]: |\dot\gamma^*_\delta(t)| \leq R\}, \quad & \text{if}\ \{t \in [0, +\infty]: |\dot\gamma^*_\delta(t)| \leq R\} \not= \emptyset
\\
+\infty, \quad & \text{if}\  \{t \in [0, +\infty]: |\dot\gamma^*_\delta(t)| \leq R\} = \emptyset. 
\end{cases}	
\end{align*}
	By Dynamic Programming Principle we get
	\begin{equation*}
	V_{\delta}(x,v)=\int_{0}^{\tau_\delta}{e^{-\delta t}\left(\frac{1}{2}|\ddot\gamma^{*}_\delta(t)|^{2}+F(\gamma^{*}_\delta(t), \dot\gamma^{*}_\delta(t)) \right)\ dt}	+e^{-\delta \tau_\delta}V_{\delta}(\gamma^{*}_\delta(\tau_\delta), \dot\gamma^{*}_\delta(\tau_\delta))
	\end{equation*}
and by assumption \cref{eq:Fgrowth} and  definition of $\tau_\delta$ we deduce that
\begin{equation}\label{eq:ccc}
\delta V_{\delta}(x,v) \geq (c_{F}^{-1}R^{\alpha}-c_{F})(1-e^{-\delta \tau_\delta})+e^{-\delta \tau_\delta}\delta V_{\delta}(\gamma^{*}_\delta(\tau_\delta), \dot\gamma^{*}_\delta(\tau_\delta)).
\end{equation}
If $\tau_\delta$ is finite, we have that $|\dot\gamma^*_\delta(\tau_\delta)|$ is bounded by $R$ and thus, by \eqref{defRzjkdn} and \cref{eq:tauberianbis} we deduce that for any $\delta \in(0, \delta_{\eps})$ 
\begin{equation*}
\delta V_{\delta}(x,v) \geq \lambda (1-e^{-\delta \tau_\delta})+e^{-\delta \tau_\delta}(\lambda -\eps) \geq \lambda - \eps.	
\end{equation*}
By \eqref{defRzjkdn} and \eqref{eq:ccc} the same inequality also holds  if $\tau_\delta=+\infty$. Hence, we obtain that
\begin{equation*}
\lim_{\delta \to 0^{+}}  \inf_{(x,v) \in \T^{d} \times \R^{d}} \delta V_{\delta}(x,v) \geq \lambda - \eps.	
\end{equation*}
By \cref{eq:tauberian} we infer that 
\begin{equation*}
\lambda = \lim_{\delta \to 0^{+}} \delta V_{\delta}(0,0) \geq 	\lim_{\delta \to 0^{+}}  \inf_{(x,v) \in \T^{d} \times \R^{d}} \delta V_{\delta}(x,v) \geq \lambda - \eps,
\end{equation*}
which implies the desired result since $\eps$ is arbitrary. 
}\end{proof}

%
As $V_\delta$ is coercive, we cannot use it directly as a test function to test the fact that a measure is closed. To overcome this issue we approximate $V_\delta$ by family of Lipschitz maps $(V_\delta^R)$. 

\begin{lemma}[{\bf Approximate problem 1}]\label{lem:approx1}
	Assume that $F$ satisfies assumption {\bf (F1)}---{\bf (F3)}. Let $R>0$ and define $F_{R}(x,v)=\min\{F(x,v), R\}$ for any $(x,v) \in \T^{d} \times \R^{d}$. 
	Let $V_{\delta}^{R}$ be the unique continuous and bounded viscosity solution to
	\begin{equation}\label{eq:problem12}
	\delta V_{\delta}^{R}(x,v)+\frac{1}{2}|D_{v} V_{\delta}^{R}(x,v)|^{2}+\langle D_{x}V^{R}_{\delta}(x,v),v \rangle =F_{R}(x,v), \quad (x,v) \in \T^{d} \times \R^{d}. 
	\end{equation}
Then, the following holds:
\begin{itemize}
\item[(i)] $V_{\delta}^{R}$ is globally Lipschitz continuous;
\item[(ii)] there are two positive constants $\tilde{c}_{1,\delta}$ and $\tilde{c}_{2, \delta}$ such that
\begin{equation}\label{eq:estest}
\delta V_{\delta}^{R}(x,v) \geq \tilde{c}_{1,\delta}\big(1+\min\{|v|^{\alpha}, R\} \big)	-\tilde{c}_{2, \delta}
\end{equation}
for any $(x,v) \in \T^{d} \times \R^{d}$;
\item[(iii)] there is a constant $\tilde{C}_{\delta} \geq 0$ such that
 \begin{equation}\label{eq:estest2}
 	|D_{x} V_{\delta}^{R}(x,v)| \leq \tilde{C}_{\delta} \big( 1 + \min\{|v|^{\alpha}, R \} \big)
 \end{equation}
for a.e. $(x,v) \in \T^{d} \times \R^{d}$;
\item[(iv)] $V_{\delta}^{R}$ converge, as $R \to +\infty$, uniformly on compact subsets of $\T^{d} \times \R^{d}$ to the map $V_{\delta}$ defined in \eqref{eq:discounted}.
\end{itemize}
\end{lemma}
The proofs of (i) and (iv) are direct consequences of optimal control theory while the proofs of  \cref{eq:estest} and \cref{eq:estest2} follow the same argument as for  \cref{eq:est2} and \cref{eq:est3}, respectively and we omit these proofs.

\begin{lemma}\label{lem.Rapprox}
Assume that $F$ satisfies {\bf (F1)} -- {\bf (F3)}. Let $F_{R}$ and $V_{\delta}^{R}$ be defined  in \Cref{lem:approx1}. Then	we have that
\begin{equation}\label{eq:Rapprox}
\inf_{\mu \in \C} \int_{\T^{d} \times \R^{d} \times \R^{d}}{\left(\frac{1}{2}|w|^{2}+F_{R}(x,v) \right)\ \mu(dx,dv,dw)} \geq \inf_{(x,v) \in \T^{d} \times \R^{d}} \delta V_{\delta}^{R}(x,v).	
\end{equation}
\end{lemma}
\begin{proof}
Let $\xi^{1,\eps}\in C^\infty_c(\R^d)$ be such that $\supp(\xi^{1,\eps}) \subset B_{\eps}$, $\xi^{1,\eps}(x) \geq 0$ and $\int_{B_{\eps}}{\xi^{1,\eps}(x)\ dx}=1$, and define $V^{R,\eps}_{\delta}(x,v)=V_{\delta}^{R} \star_x \xi^{1,\eps} (x,v)$ where the mollification only holds in $x$. Then $V^{\eps}_{\delta}$ satisfies the following inequality in the viscosity sense
\begin{equation*}
\delta V^{\eps}_{\delta}(x,v) + \frac{1}{2}|D_{v}V^{\eps}_{\delta}(x,v)|^{2}+\langle D_{x}V^{\eps}_{\delta}(x,v),v \rangle \leq F_{R} \star \xi^{1,\eps}(x,v) \leq F_{R}(x,v)+C_{F}\eps(1+\min\{|v|^{\alpha}, R\})
\end{equation*}
where the last inequality holds true by {\bf (F3)} and the definition of $F_{R}$. 
Now, let $\xi^{2,\eps} \in C^\infty_c(\R^d)$ be such that $\supp(\xi^{2,\eps}) \subset B_{\eps}$, $\xi^{2,\eps}(v) \geq 0$ and $\int_{B_{\eps}}{\xi^{2,\eps}(v)\ dv}=1$  and define $\varphi^{\eps, \delta}_{R}(x,v)=\xi^{2,\eps} \star_v V^{R, \eps}_{\delta}(x,v)$ (where the  the mollification now only holds in $v$). 
Then, by \cref{eq:estest2} we have that 
\begin{equation*}
|\xi^{2, \eps}\star_v (\langle D_{x}V^{R, \eps}_{\delta}(x,\cdot), \cdot \rangle)(v) - \langle D_{x}\varphi^{\eps, \delta}_{R}(x,v), v \rangle | \leq \eps\|D_{x}V^{R, \eps}_{\delta} \|_{L^\infty(B_{\eps}(x,v))}\leq C_{\delta}\eps(1+\min\{|v|^{\alpha}, R\}),
\end{equation*}
 which implies that
\begin{align*}
& \delta \varphi^{\eps, \delta}_{R}(x,v) +\frac{1}{2}|D_{v}\varphi^{\eps, \delta}_{R}(x,v)|^{2} + \langle D_{x}\varphi^{\eps, \delta}_{R}(x,v), v \rangle\\
\leq\ & \delta \varphi^{\eps, \delta}_{R}(x,v) +\frac{1}{2}|D_{v}\varphi^{\eps, \delta}_{R}(x,v)|^{2} + 	\xi^{2, \eps}\star_v \langle D_{x}V^{R, \eps}_{\delta}(x,v), v \rangle +  C_{\delta}\eps(1+\min\{|v|^{\alpha}, R\}) \\
	\leq\ & F_{R} \star \xi^{2, \eps}(x,v) + C_{\delta}\eps(1+\min\{|v|^{\alpha}, R\}) \leq F_{R}(x,v) + C_{1, \delta}\eps(1+\min\{|v|^{\alpha}, R\}) 
\end{align*}
where the last inequality holds true by assumption {\bf (F3)}. 
Thus, so far we have proved that for any $(x,v) \in \T^{d} \times \R^{d}$
\begin{align}\label{eq:att1}
\begin{split}
	& \delta \varphi^{\eps, \delta}_{R}(x,v) +\frac{1}{2}|D_{v}\varphi^{\eps, \delta}_{R}(x,v)|^{2} + \langle D_{x}\varphi^{\eps, \delta}_{R}(x,v), v \rangle
\\
& \qquad \qquad \leq \  F_{R}(x,v) + C_{1, \delta}\eps(1+\min\{|v|^{\alpha}, R\}) . 
	\end{split}
\end{align}
Moreover, in view of 	\cref{eq:estest} we deduce that there exists a constant $C_{2,\delta} \geq 0$ such that for any $(x,v) \in \T^{d} \times \R^{d}$ we have that
\begin{equation}\label{eq:att2}
\delta \varphi^{\eps, \delta}_{R}(x,v) \geq C^{-1}_{2,\delta}\min\{|v|^{\alpha}, R\}-C_{2,\delta}.	
\end{equation}
We claim that for $\eps > 0$ small enough, the following holds:
\begin{align}\label{eq:claim}
\begin{split}
& \inf_{\mu \in \C}\int_{\T^{d} \times \R^{d} \times \R^{d}}{\left(\frac{1}{2}|w|^{2}+F_{R}(x,v) \right)\ d\mu(x,v,w)} 
\\
& \qquad \qquad \geq\  \inf_{(x,v) \in \T^{d} \times \R^{d}}\left(\delta \varphi^{\eps, \delta}_{R}(x,v)-C_{1,\delta}\eps\big(1+\min\{|v|^{\alpha}, R\}\big) \right). 
\end{split}
\end{align}
By \Cref{rem:lipschitzclosed} below, we can test the fact that a measure is closed by smooth and globally Lipschitz continuous maps. Let $\E(\T^{d} \times \R^{d})$ be such a  set. Then 
\begin{align*}
& \inf_{\mu \in \C}\int_{\T^{d} \times \R^{d} \times \R^{d}}{\left(\frac{1}{2}|w|^{2}+F_{R}(x,v) \right)\ \mu(dx,dv,dw)} 
\\
=\ &  \inf_{\mu \in \PP_{\alpha,2}(\T^{d} \times \R^{d} \times \R^{d})}\sup_{\psi \in \E(\T^{d} \times \R^{d})} \int_{\T^{d} \times \R^{d} \times \R^{d}}{\Big(\frac{1}{2}|w|^{2}+F_{R}(x,v) - \langle D_{x}\psi(x,v), v \rangle - \langle D_{v}\psi(x,v),w \rangle \Big)\ \mu(dx,dv,dw)}
\\
\geq &  \sup_{\psi \in \E(\T^{d} \times \R^{d})} \inf_{\mu \in \PP_{\alpha,2}(\T^{d} \times \R^{d} \times \R^{d})} \int_{\T^{d} \times \R^{d} \times \R^{d}}{\Big(\frac{1}{2}|w|^{2}+F_{R}(x,v) - \langle D_{x}\psi(x,v), v \rangle - \langle D_{v}\psi(x,v),w \rangle \Big)\ \mu(dx,dv,dw)}
\\
\geq & \inf_{\mu \in \PP_{\alpha,2}(\T^{d} \times \R^{d} \times \R^{d})} \int_{\T^{d} \times \R^{d} \times \R^{d}}{\Big(\frac{1}{2}|w|^{2}+F_{R}(x,v) - \langle D_{x}\varphi^{\eps, \delta}_{R}(x,v), v \rangle - \langle D_{v}\varphi^{\eps, \delta}_{R}(x,v),w \rangle \Big)\ \mu(dx,dv,dw)}
\\
= & \inf_{(x,v) \in \T^{d} \times \R^{d}}\left\{-\frac{1}{2}|D_{v}\varphi^{\eps, \delta}_{R}(x,v)|^{2}+F_{R}(x,v) - \langle D_{x}\varphi^{\eps, \delta}_{R}(x,v), v \rangle\right\} , 
\end{align*}
which proves  \cref{eq:claim} thanks to \cref{eq:att1}. Recalling \eqref{eq:att2}, the right hand side of \cref{eq:claim} is  coercive in $v$ uniformly in $\eps$ for $\eps$ small. As in addition $\varphi^{\eps, \delta}_{R}$ converges locally uniformly to $V_{\delta}^{R}$ as $\eps\to 0$, we obtain
\begin{equation*}
\lim_{\eps \to 0} \inf_{(x,v) \in \T^{d} \times \R^{d}} \left(\delta \varphi^{\eps, \delta}_{R}(x,v)-C_{2,\delta}\eps\big(1+\min\{|v|^{\alpha}, R\}\big) \right) = \inf_{(x,v) \in \T^{d} \times \R^{d}} \delta V_{\delta}^{R}(x,v).
\end{equation*}
So we can let $\eps\to 0$ in \eqref{eq:claim} to obtain the result. 
\end{proof}

In the proof we used the following: 
\begin{remarks}\label{rem:lipschitzclosed}\em
	Note that we can allow for a larger class of test functions in \Cref{def:closedmeasure}, i.e. $\varphi \in W^{1,\infty}(\T^{d} \times \R^{d}) \cap C^{\infty}(\T^{d} \times \R^{d})$. Indeed, let $\varphi \in W^{1,\infty}(\T^{d} \times \R^{d}) \cap C^{\infty}(\T^{d} \times \R^{d})$ and for $R > 1$ let $\xi_{R}\in C^\infty_c(\R^d)$ be such that $\xi_{R}(x,v)=1$ for $(x,v) \in \T^{d} \times B_{R}$, $\xi_{R}(x,v)=0$ for $(x,v) \in \T^{d} \times \R^{d} \backslash B_{2R}$, $0 \leq \xi_{R}(x,v) \leq 1$ for $\T^{d} \times B_{2R} \backslash B_{R}$ and there exists a constant $M \geq 0$ such that $|D\xi_{R}(x,v)| \leq MR^{-1}$ for any $(x,v) \in \T^{d} \times \R^{d}$. Set $\varphi_{R}=\varphi  \xi_{R}$. Then, we have that $\varphi_{R} \in C^{\infty}_{c}(\T^{d} \times \R^{d})$, $D\varphi_R$ is uniformly bounded and converges locally uniformly to $D\varphi$. For $\mu \in \C$ we have: 
	\begin{equation}\label{eq:approxclosed}
	\int_{\T^{d} \times \R^{d} \times \R^{d}}{\Big(\langle D_{x}\varphi_{R}(x,v), v \rangle + \langle D_{v}\varphi_{R}(x,v), w \rangle \Big)\ \mu(dx,dv,dw)}=0	.
	\end{equation}
Since $\mu \in \PP_{2, \alpha}(\T^{d} \times \R^{d} \times \R^{d})$, we can pass to the limit in \cref{eq:approxclosed} as $R\to +\infty$ by dominate convergence. This  proves that
\begin{equation*}
	\int_{\T^{d} \times \R^{d} \times \R^{d}}{\Big(\langle D_{x}\varphi(x,v), v \rangle + \langle D_{v}\varphi(x,v), w \rangle \Big)\ \mu(dx,dv,dw)}=0	
\end{equation*}
  for $\varphi \in W^{1,\infty}(\T^{d} \times \R^{d}) \cap C^{\infty}(\T^{d} \times \R^{d})$.\qed
\end{remarks}

{

In the next step, we let $R\to +\infty$ in \eqref{eq:Rapprox}: 

\begin{lemma}\label{lem.Rapprox2}
Assume that $F$ satisfies {\bf (F1)} -- {\bf (F3)}. Let $V_{\delta}$ be defined  in \eqref{eq:discounted}. Then
\begin{equation}\label{eq:Rapprox2}
\inf_{\mu \in \C} \int_{\T^{d} \times \R^{d} \times \R^{d}}{\left(\frac{1}{2}|w|^{2}+F(x,v) \right)\ \mu(dx,dv,dw)} \geq \inf_{(x,v) \in \T^{d} \times \R^{d}} \delta V_{\delta}(x,v).	
\end{equation}
\end{lemma}

\proof We first consider the left-hand side of \eqref{eq:Rapprox}, for which we obviously have, by the definition of $F_R$ in Lemma \ref{lem:approx1}, 
\begin{align}\label{zkebsdfc}
& \inf_{\mu \in \C} \int_{\T^{d} \times \R^{d} \times \R^{d}}{\left(\frac{1}{2}|w|^{2}+F_{R}(x,v) \right)\ \mu(dx,dv,dw)} \notag
\\
& \qquad \leq \inf_{\mu \in \C} \int_{\T^{d} \times \R^{d} \times \R^{d}}{\left(\frac{1}{2}|w|^{2}+F(x,v) \right)\ \mu(dx,dv,dw)}.	
\end{align}
As for the right hand side of \cref{eq:Rapprox}, we note that, if $(x_R, v_R)\in \T^d\times \R^d$ satisfies 
$$
V^R_\delta(x_R,v_R)\leq \inf_{(x,v)\in \T^d\times \R^d} V^R_\delta(x,v) +R^{-1},
$$
then, as $V^R_\delta\leq V_\delta$ and \eqref{eq:estest} holds, we have 
$$
 \tilde{c}_{1,\delta}\big(1+\min\{|v_R|^{\alpha}, R\} \big)	-\tilde{c}_{2, \delta} \leq  \inf_{(x,v)\in \T^d\times \R^d} V_\delta(x,v) +R^{-1}. 
$$
This proves that $v_R$ remains bounded in $R$ and we can find a subsequence of $(x_R,v_R)$,  denoted in the same way, which converges to some $(\bar x,\bar v)\in \T^d\times \R^d$ as $R\to +\infty$. Then by local uniform convergence of $V^R_\delta$ to $V_\delta$, we obtain that 
\begin{equation}\label{zkebsdfc2}
 \inf_{(x,v)\in \T^d\times \R^d} V_\delta(x,v) \leq  V_\delta(\bar x,\bar v) = \lim_{R\to+\infty}  V^R_\delta(x_R,v_R) = \lim_{R\to+\infty} \inf_{(x,v)\in \T^d\times \R^d} V^R_\delta(x,v) .
\end{equation}
Passing to the limit as $R\to+\infty$ in \eqref{eq:Rapprox} proves the Lemma thanks to \eqref{zkebsdfc} and \eqref{zkebsdfc2}. 
%
\qed

}

We are now ready to prove the main result of this section.

\begin{proposition}[{\bf Characterization with closed measures}]\label{prop:caract}
Assume that $F$ satisfies {\bf (F1)} --- {\bf (F3)}. For any $(x_{0},v_{0}) \in \T^{d}\times \R^{d}$ we have that
\begin{equation*}
\lim_{T \to \infty} \frac{1}{T}V^{T}(0,x_{0},v_{0}) = \inf_{\mu \in \mathcal{C}} \int_{\T^{d} \times \R^{d} \times \R^{d}}{\left(\frac{1}{2}|w|^{2}+F(x,v) \right)\ \mu(dx,dv,dw)}.
\end{equation*}
\end{proposition}
\proof
Let $\gamma^T_{(x_{0},v_{0})}$ be a minimum for the problem
\begin{equation*}
\inf_{\gamma \in \Gamma_{0}(x_{0},v_{0}) }J^{T}(\gamma).	
\end{equation*}
Let us define the probability measures $\mu_{T}$ by
 \begin{equation*}
 	\int_{\T^{d} \times \R^{d} \times \R^{d}}{\varphi(x,v,w)\ d\mu^{T}(x,v,w)}=\frac{1}{T}\int_{0}^{T}{\varphi(\gamma^T_{(x_{0},v_{0})}(t), \dot\gamma^T_{(x_{0},v_{0})}(t), \ddot\gamma^T_{(x_{0},v_{0})}(t))\ dt}
 \end{equation*}
 for any $\varphi \in \C^{\infty}_{c}(\T^{d} \times \R^{d} \times \R^{d})$. By Lemma \ref{lem:measureconv}, the $(\mu^T)$ converge, up to a subsequence $(T_n)$, weak-$*$ to a closed  measure $\mu^*$. Therefore 
%
%
\begin{align*}
&\lim_{T \to \infty} \inf_{\gamma \in \Gamma_{0}(x_{0},v_{0})}\frac{1}{T}J^{T}(\gamma) =  \lim_{n \to \infty} \frac{1}{T_n}\int_{0}^{T_n}{\left(\frac{1}{2}|\ddot\gamma^{T_n}_{(x_{0},v_{0})}(t)|^{2}+F(\gamma^{T_n}_{(x_{0},v_{0})}(t), \dot\gamma^{T_n}_{(x_{0},v_{0})}(t)) \right)\ dt} 
\\
=\ & \lim_{n \to \infty}	\int_{\T^{d} \times \R^{d} \times \R^{d}}{\left(\frac{1}{2}|w|^{2}+F(x,v) \right)\ \mu^{T_n}(dx,dv,dw)}
\geq\  \int_{\T^{d} \times \R^{d} \times \R^{d}}{\left(\frac{1}{2}|w|^{2}+F(x,v) \right)\ \mu^{*}(dx,dv,dw)}. 
\end{align*}
Thus, taking the infimum over the set of closed measures $\C$ we obtain that
\begin{equation*}
	\lim_{T \to \infty} \inf_{\gamma \in \Gamma_{0}(x_{0},v_{0})}\frac{1}{T} J^{T}(\gamma) \geq\ \inf_{\mu \in \C} \int_{\T^{d} \times \R^{d} \times \R^{d}}{\left(\frac{1}{2}|w|^{2}+F(x,v) \right)\ d\mu(x,v,w)}. 
\end{equation*}

To obtain the opposite inequality, we note that, by \eqref{eq:Rapprox2} (which holds for any $\delta>0$) and Proposition \ref{prop:Inflemma}, we have 
$$
\inf_{\mu \in \C} \int_{\T^{d} \times \R^{d} \times \R^{d}}{\left(\frac{1}{2}|w|^{2}+F(x,v) \right)\ \mu(dx,dv,dw)} \geq \lim_{\delta\to 0^+} \inf_{(x,v) \in \T^{d} \times \R^{d}} \delta V_{\delta}(x,v) = \lambda,
$$
where $\lambda$ defined in \cref{eq:tauberian}. Then we can conclude thanks to \eqref{eq:tauberian}.
\qed

\proof[Proof of Theorem \ref{theo:main1}]  The existence of the limit and the fact that it does not depend on $(x,v)$ is the main statement of Proposition \ref{prop:existslimit} while the characterization of this limit is given by Proposition \ref{prop:caract}. 
\qed 

\section{Asymptotic behavior of MFG with acceleration}\label{sec:AsymptoticMFG}

 We now turn to MFG problems of acceleration. In order to study the asymptotic behavior of these problems, we first need to describe the expected limit: the ergodic MFG problems of acceleration. The difficulty here is that, as explained in the previous part, we do not expect the existence of a corrector and therefore the ergodic MFG problem cannot be phrased in these terms. We overcome this issue by using the characterization of the ergodic limit given by Theorem \ref{theo:main1} in terms of closed measures. This suggests the definition of equilibria for ergodic MFG  of acceleration (Definition \ref{def.ergoMFGeq}). We prove the existence and the uniqueness of a solution in Proposition \ref{prop.exuniErgoMFGeq}. In order to pass to the limit in the time-dependent MFG system of acceleration, we first need to rephrase the solution of this system in terms of closed measures (more precisely in terms of the so-called $T-$closed measures, see \Cref{def:TclosedDef}). This is the aim of the second part of the section (\Cref{thm:lintegral}). Thanks to this characterization, we are then able to conclude on the long time average and complete the proof of \Cref{def.ergoMFGeq}. 

\subsection{Ergodic MFG with acceleration}

Following Definition \ref{def:closedmeasure} we recall that $\C \subset  \PP_{\alpha,2}(\T^{d} \times \R^{d} \times \R^{d})$ denotes the set of closed measures, i.e. $\mu \in \C$ if it satisfies for any test function $\varphi \in \C^{\infty}_{c}(\T^{d} \times \R^{d})$ the following condition:
\begin{equation*}
\int_{\T^{d} \times \R^{d} \times \R^{d}}{\Big(\langle D_{x}\varphi(x,v), v \rangle + \langle D_{v}\varphi(x,v) , w \rangle  \Big)\ \mu(dx,dv,dw)}	=0.
\end{equation*}

The candidate limit problem that we are going to study is the following fixed point problem: we look for a measure $\mu \in \C$ such that
   \begin{equation}\label{eq:Ergodicproblem}
\mu \in \argmin_{\eta \in \C}\left\{\int_{\T^{d} \times \R^{d} \times \R^{d}}{\left(\frac{1}{2}|w|^{2} +F(x,v, \pi \sharp \mu)\right)\ \eta(dx,dv,dw)} \right\}
\end{equation}
where $\pi: \T^{d} \times \R^{d} \times \R^{d}$, defined as $\pi(x,v,w)=(x,v)$, is the projection function.

%
%
%
%
%

\begin{proposition}\label{prop.exuniErgoMFGeq}
Assume that $F$ satisfies {\bf (F1')} and {\bf (F2')}. Then, there exists at least one solution ${ (\bar\lambda, \bar\mu) \in \R \times \C}$ of the ergodic MFG problem. 

Moreover, if $F$ satisfies the  monotonicity assumption \eqref{monotonicity}
and if $(\bar\lambda_{1}, \bar\mu_{1})$ and $(\bar\lambda_{2}, \bar\mu_{2})$ are two solutions of the ergodic MFG problem, then $\bar\lambda_{1}=\bar\lambda_{2}$.
\end{proposition}

\proof  {
Let $\mathcal K$ be the set of  probability measures $\mu\in \C$ such that
\begin{equation*}
\int_{\T^{d} \times \R^{d}}(\frac12 |w|^2+c_F^{-1}|v|^\alpha)\ \mu(dx,dv,dw) \leq 2c_F,	
\end{equation*}
where $\alpha$ and $c_F$ are given by assumption {\bf (F2')}. We endow $\mathcal K$ with the $d_1$ distance and define, for any $\mu\in \mathcal K$, the set $\Psi(\mu)$ as the set of minimizers $\bar \eta\in \C$ of the map defined on $\C$
\begin{equation}\label{eoizgrdmngc}
\eta \to \int_{\T^{d} \times \R^{d} \times \R^{d}}{\left(\frac{1}{2}|w|^{2} +F(x,v, \pi\sharp \mu)\right)\ \eta(dx,dv,dw)} 
\end{equation}
We also denote by $\Lambda(\mu)$ the value of this minimum. 
First, we show that  the set-valued map $\Psi$ is well-defined from $\mathcal K$ into $\mathcal K$. Indeed, if $\mu \in \mathcal K$ and $\bar \eta\in \C$ is any minimum of \eqref{eoizgrdmngc}, we have by assumption {\bf (F2')}  (setting $\tilde \eta = \delta_{(x_0, 0,0)} \in \C$ for an arbitrary point $x_0\in \T^d$):
\begin{align*}
& \int_{\T^d\times \R^d\times \R^d} (\frac12|w|^2+c_F^{-1}|v|^\alpha-c_F)\ \bar \eta(dx,dv,dw) \leq
 \int_{\T^d\times \R^d\times \R^d} (\frac12|w|^2 +F(x,v, \pi\sharp \mu)) \  \bar \eta(dx,dv,dw)\\
& \qquad \leq  \int_{\T^d\times \R^d\times \R^d} (\frac12|w|^2 +F(x,v, \pi\sharp \mu)) \  \tilde \eta(dx,dv,dw) \ \leq \ c_F.
\end{align*}
So $\bar \eta$ belongs to $\mathcal K$. Moreover, we observe that a solution of the ergodic MFG problem exists if the set-valued map $\Psi$ has a fixed-point and we prove that this is the case using the Kakutani fixed-point theorem. 
Since $\alpha > 1 $, by the above considerations, we know that the space $\mathcal K$ is compact with respect to the $d_{1}$ distance. Thus,  for any $\mu \in \mathcal K$, the set $\Psi(\mu)$ is convex and compact. It remains to check that $\Psi$ has closed graph. Fix a sequence $\{ \mu_{j}\}_{j \in \N} \subset  \mathcal K$ and a sequence $\{\eta_{j}\}_{j \in \N} \subset  \mathcal K$ such that 
\begin{equation*}
\mu_{j} \rightharpoonup^{d_1} \mu, \quad \eta_{j} \rightharpoonup^{d_1} \bar \eta, \quad \text{and}\quad \eta_{j} \in \Psi(\mu_{j})\ \forall j \in \N.
\end{equation*}
Let us show that $\bar \eta \in \Psi(\mu)$. Note that $\bar \eta\in \C$. 
It remains to check that $\bar \eta$ minimizes \eqref{eoizgrdmngc}. By standard lower-semi continuity arguments and continuity of $F$, we have: 
\begin{align}\label{kjzhebsrdfc}
\int_{\T^d\times \R^d\times \R^d}( \frac12|w|^2 +F(x,v, \pi\sharp \mu)) \  \bar \eta(dx,dv,dw) \leq \liminf_j \int_{\T^d\times \R^d\times \R^d} (\frac12|w|^2 +F(x,v, \pi\sharp \mu_j)) \  \eta_j(dx,dv,dw). 
 \end{align}
 We now check that the right-hand side is not larger that $\Lambda(\mu)$. Indeed, let $\tilde \eta$ belong to $\Psi(\mu)$ and fix $\eta>0$. As $\tilde \eta$ belongs to $\mathcal K$ we can find $R>0$ such that 
 $$
 \int_{(\T^d\times \R^d\times \R^d)\backslash B_R}( \frac12|w|^2 + c_F|v|^\alpha+c_F)\ \tilde \eta(dx,dv,dw) \leq \eps. 
 $$
 As $\pi\sharp \mu_j$ converges to $\pi\sharp \mu$ for the $d_1$ distance, we have by assumption  {\bf (F1')} that, for $j$ large enough,  
 $$
 \lim_{j \to +\infty} \sup_{(x,v)\in  B_R} |F(x,v,\pi\sharp \mu_j)- F(x,v,\pi\sharp \mu)|\leq \eps .
 $$
 So, by optimality of $\eta_j$ and the estimates above,  
 \begin{align*}
&  \int_{\T^d\times \R^d\times \R^d} (\frac12|w|^2 +F(x,v, \pi\sharp \mu_j)) \  \eta_j(dx,dv,dw)=\Lambda(\mu_j) \\
& \leq \int_{\T^d\times \R^d\times \R^d} (\frac12|w|^2 +F(x,v, \pi\sharp \mu_j)) \  \tilde \eta(dx,dv,dw)\\
& \leq \int_{B_R} (\frac12|w|^2 +F(x,v, \pi\sharp \mu_j)) \ \tilde  \eta(dx,dv,dw)
+ \int_{B_R^c} (\frac12|w|^2 +c_F|v|^\alpha+c_F) \ \tilde \eta(dx,dv,dw)\\
&\leq \int_{B_R} (\frac12|w|^2 +F(x,v, \pi\sharp \mu)) \ \tilde  \eta(dx,dv,dw)+2\eps \leq \Lambda(\mu)+2\eps.
 \end{align*}
Coming back to \eqref{kjzhebsrdfc}, this shows that  
$$
\int_{\T^d\times \R^d\times \R^d}( \frac12|w|^2 +F(x,v, \pi\sharp \mu)) \ \bar \eta(dx,dv,dw) \leq \Lambda(\mu),
$$
and therefore that $\bar \eta$ belongs to $\Psi(\mu)$. Therefore, applying Kakutani fixed-point theorem we have that there exists a fixed point $\bar{\eta}$ of $\Psi$ and this is a solution of the ergodic MFG problem. 

Now, we prove that under the monotonicity assumption \cref{monotonicity} the critical value is unique. Let $(\blambda_{1}, \bar\mu_{1})$ and $(\blambda_{2}, \bar\mu_{2})$ be two solutions of the ergodic MFG problem. Then, by definition we have that, for $i=1$ or $i=2$, 
\begin{align}\label{eq:def1}
\begin{split}
\bar\lambda_{i} =\ & \inf_{\mu \in \C} \int_{\T^{d} \times \R^{d} \times \R^{d}}{\left(\frac{1}{2}|w|^{2} +F(x,v, \pi \sharp \bar{\mu}_{i})\right)\ \mu(dx,dv,dw)}	
\\
=\ & \int_{\T^{d} \times \R^{d} \times \R^{d}}{\left(\frac{1}{2}|w|^{2} +F(x,v, \pi \sharp \bar{\mu}_{i})\right)\ \bar{\mu}_{i}(dx,dv,dw)}.
\end{split}
\end{align}
Thus, exchanging the role of $\bar{\mu}_{1}$ and $\bar{\mu}_{2}$ as competitor for $\blambda_{1}$ and $\blambda_{2}$, respectively, we get
\begin{equation}\label{eq:lambda1}
\bar\lambda_{1} \leq \int_{\T^{d} \times \R^{d} \times \R^{d}}{\left(\frac{1}{2}|w|^{2}+F(x,v,\pi \sharp \bar\mu_{1}) \right)\ \bar{\mu}_{2}}(dx,dv,dw)	
\end{equation}
and 
\begin{equation}\label{eq:lambda2}
\bar\lambda_{2} \leq \int_{\T^{d} \times \R^{d} \times \R^{d}}{\left(\frac{1}{2}|w|^{2}+F(x,v,\pi \sharp \bar\mu_{2}) \right)\ \bar{\mu}_{1}}(dx,dv,dw).
\end{equation}
We first take the difference between \cref{eq:lambda1} and \cref{eq:def1} for $i=2$ and we get
\begin{align*}
\bar\lambda_{1}-\bar\lambda_{2} \leq & 	
 \int_{\T^{d} \times \R^{d} \times \R^{d}}{\big(F(x,v,\pi \sharp \bar{\mu}_{1})-  F(x,v,\pi \sharp \bar{\mu}_{2})\big)\ d\bar{\mu}_{2}}(dx,dv,dw).
\end{align*}
Taking the difference between \cref{eq:lambda1} for $i=1$  and  \cref{eq:lambda2} we get
\begin{align*}
	\bar\lambda_{1}-\bar\lambda_{2} \geq & 	
 \int_{\T^{d} \times \R^{d} \times \R^{d}}{\big(F(x,v,\pi \sharp \bar{\mu}_{1})-  F(x,v,\pi \sharp \bar{\mu}_{2})\big)\ d\bar{\mu}_{1}}(dx,dv,dw).
\end{align*}
Thus, taking the difference of the above expressions we deduce that
\begin{equation*}
0 \geq 	 \int_{\T^{d} \times \R^{d} \times \R^{d}}{\big(F(x,v,\pi \sharp \bar{\mu}_{1})-  F(x,v,\pi \sharp \bar{\mu}_{2})\big)\ (\bar{\mu}_{1}}(dx,dv,dw)-\bar{\mu}_{2}(dx,dv,dw))
\end{equation*}
which implies by monotonicity assumption \cref{monotonicity} that $F(x,v,\pi \sharp \bar{\mu}_{1}) = F(x,v,\pi \sharp \bar{\mu}_{2})$. Coming back to \eqref{eq:lambda1}, it follows that $\bar\lambda_{1}=\bar\lambda_{2}$.
\qed

\subsection{Representation of the solution of the time-dependent MFG system}

We now consider the  time-dependent MFG system \eqref{eq:accMFG}. In \cite{bib:CM, bib:YA}, it has been proved that such system has a solution $(u^{T}, m^{T})$ and that the function $u^{T}$ can be represented as
\begin{equation}\label{eq:MFGvaluefunction}
u^{T}(t,x,v)=\inf_{\gamma \in \Gamma_{t}(x,v)}\left\{\int_{t}^{T}{\Big(\frac{1}{2}|\ddot\gamma(s)|^{2}+F(\gamma(s), \dot\gamma(s), m^{T}_{s}) \Big)\ ds} + g(\gamma(T), \dot\gamma(T), m^{T}_{T}) \right\}.	
\end{equation}
In order to compare the solution of this time-dependent problem with the solution of the ergodic MFG problem, which is written in terms of closed measures, we need to rewrite the time-dependent problem  in term of flows of Borel probability measures on $\T^d\times \R^d\times \R^d$. The following definition mirrors the definition of closed measure in the ergodic setting: 

\begin{definition}[{\bf T-Closed measures}]\label{def:TclosedDef}
		Let $T$ be a finite time horizon and let $m_{0}\in \PP_1(\T^d\times \R^d)$. If $\eta \in C([0,T]; \PP_{1}(\T^{d} \times \R^{d} \times \R^{d}))$, we say that $\eta$ is a $T$-closed measure associated with $m_{0}$ if for any test function $\varphi \in C^{\infty}_{c}([0,T] \times \T^{d} \times \R^{d})$ the following holds
		\begin{align}\label{eq:Tclosedcondition}
		\begin{split}
			& \int_{0}^{T}\int_{\T^{d} \times \R^{d} \times \R^{d}}{\Big(\partial_{t} \varphi(t,x,v) + \langle D_{x}\varphi(t,x,v), v \rangle + \langle D_{v} \varphi(t,x,v), w \rangle  \Big)\ \eta_{t}(dx,dv,dw)dt}
			\\
			=\ & \int_{\T^{d} \times \R^{d} \times \R^{d}}{\varphi(T,x,v)\ \eta_{T}(dx,dv,dw)}- \int_{\T^{d} \times \R^{d} \times \R^{d}}{\varphi(0,x,v)\ m_{0}(dx,dv)}.
			\end{split}
		\end{align}
	\end{definition}
\noindent We denote by $\C^{T}(m_{0})$ the set of $T$-closed measures associated with $m_{0} \in \PP_1(\T^{d} \times \R^{d})$. 

The goal of the subsection is to prove the following equality: 
\begin{theorem}\label{thm:lintegral}
Assume that $F$ satisfies {\bf (F1')}, {\bf (F2')} and $g$ satisfies {\bf (G1)}. Let $M \geq 0$ and assume that
\begin{equation}\label{eq:Mbound}
	\int_{\T^{d} \times \R^{d}}{|v|^{\alpha}\ m_{0}(dx,dv)} \leq M. 
\end{equation}
{ Let $(u^T, m^T)$ be a solution to \eqref{eq:accMFG}. } Then
\begin{align}\label{eq:lintegral}
	\begin{split}
& \inf_{\mu \in \C^{T}(m_{0})} \Bigl\{  \int_{0}^{T}\int_{\T^{d} \times \R^{d} \times \R^{d}}{\left(\frac{1}{2}|w|^{2}+F(x,v,m^{T}_{t}) \right)\ \mu_{t}(dx,dv,dw)dt}	
\\
&\qquad\qquad\qquad + \int_{\T^{d} \times \R^{d} \times \R^{d}}{g(x,v,m^{T}_{T})\ \mu_{T}(dx,dv,dw)} \Bigr\}
\\
 &\qquad =  \int_{\T^{d} \times \R^{d}}{u^{T}(0,x,v)\ m_{0}(dx,dv)}. 
\end{split} 
\end{align}

In addition, there exists a minimizer $\bar \mu^T\in \C^T(m_0)$ of the problem in the left-hand side of \eqref{eq:lintegral} such that $m^T_{t} = \pi \sharp \bar\mu^T_{t}$, where $\pi:  \T^{d} \times \R^{d} \times \R^{d} \to \T^{d} \times \R^{d}$ is the canonical projection on the two first coordinates, i.e. such that $\pi(x,v,w)=(x,v)$.
\end{theorem}

The proof of Theorem \ref{thm:lintegral}  follows standard arguments but is slightly technical because the problem is stated in the whole space in velocity. The main problem is to regularize the map $u^T$ in order to have a smooth function with a compact support which satisfies a suitable (approximate) Hamilton-Jacobi inequality. The first step towards this aim is the following Lemma: 

\begin{lemma}[{\bf Approximate problem 2}]\label{lem:approx2}
Let $f: \T^{d} \times \R^{d} \times [0,T] \to \R$ be a continuous map with at most a polynomial growth and which is locally Lipschitz continuous in space locally uniformly in time and  $g: \T^{d} \times \R^{d} \to \R$ be a locally Lipschitz continuous map with at most a polynomial growth.
Let $R > 0$ and let $\xi^{R}$ be a smooth cut--off function such that $\xi^{R} \geq 0$, $\xi^{R}(x,v)=1$ if $(x,v) \in \T^{d} \times  \overline{B}_{R}$, $ 0 \leq \xi^{R}(x,v) \leq 1$ if $(x,v) \in \T^{d} \times \overline{B}_{2R} \backslash \overline{B}_{R}$ and $\xi^{R}(x,v)=0$ if $(x,v) \in \T^{d} \times \overline{B}_{R}^{c}$. 
Define $f_{R}: \T^{d} \times \R^{d} \times [0,T] \to \R$ and $g_{R}: \T^{d} \times \R^{d} \to \R$ as $f_{R}=\xi^{R}  f$ and $g_{R}=\xi^{R} g$.
Let $u^{T}_{R}$ be the viscosity solution of the following problem
	\begin{align}\label{eq:problem1}
	\begin{cases}
	-\partial_{t}u^{T}_{R}(t,x,v)+\frac{1}{2}|D_{v} u_{R}^{T}(t,x,v)|^{2}-\langle D_{x}u^{T}_{R}(t,x,v),v \rangle &\\
	\hspace{6cm}   =f_{R}(t,x,v), &\;  {\rm in }\; [0,T]\times  \T^{d} \times \R^{d} 
	\\
	u^{T}(T,x,v)=g_{R}(x,v), & \; {\rm in}\;  \T^{d} \times \R^{d}.
	\end{cases}
	\end{align}
Then, the following hold:
\begin{enumerate}
\item $u^{T}_{R}$ has compact support;
\item $u^{T}_{R}$ is Lipschitz continuous in space and velocity variable;
\item $u^{T}_{R}$ converge, as $R \to +\infty$, locally uniformly to the solution $u^{T}$ of the following problem
\begin{align*}
\begin{cases}
	-\partial_{t}u^{T}(t,x,v)+\frac{1}{2}|D_{v} u^{T}(t,x,v)|^{2}-\langle D_{x}u^{T}(t,x,v),v \rangle =f(t,x,v), & \quad {\rm in}\;  [0,T]\times \T^{d} \times \R^{d}
	\\
	u^{T}(T,x,v)=g(x,v), & \quad  {\rm in}\;  \T^{d} \times \R^{d}.
	\end{cases}
	\end{align*}
\end{enumerate}	
\end{lemma}

The proof of the Lemma follows standard argument in  optimal control and we omit it. Next we prove Theorem \ref{thm:lintegral} in the simpler case where $F$ and $g$ are replaced by $F_R$ and $g_R$: 

\begin{proposition}\label{thm:initialintegral}
	Assume that $F$ satisfies {\bf (F1')} and {\bf (F2')} and $g$ satisfies {\bf (G1)}. Let $(u^{T}, m^{T})$ be a solution of system \cref{eq:accMFG}. 
	For $R > 0$, let $\xi^{R}$ be a smooth cut--off function  as in Lemma \ref{lem:approx2} and let us set $F_{R}=\xi^{R} F$ and $g_{R}=\xi^{R} g$. Let $u^{T}_{R}$ be the continuous viscosity solution of the following problem
	\begin{align}\label{eq:MFG-HJ}
	\begin{split}
		\begin{cases}
			-\partial_{t}u^{T}_{R}(t,x,v)+\frac{1}{2}|D_{v}u^{T}_{R}(t,x,v)|^{2}-\langle D_{x}u^{T}_{R}(t,x,v),v \rangle & \\
			\hspace{6cm}  =F_{R}(x,v,m^{T}_{t}), & \quad {\rm in}\; [0,T] \times \T^{d} \times \R^{d}
			\\
			u^{R}(T,x,v)=g_{R}(x,v, m^{T}_{T}), & \quad {\rm in}\;  \T^{d} \times \R^{d}. 
		\end{cases}
	\end{split}
	\end{align}
 Then
	\begin{align*}
		\inf_{\mu \in \C^{T}(m_{0})} & \Bigl\{  \int_{0}^{T}\int_{\T^{d} \times \R^{d} \times \R^{d}}{\Big(\frac{1}{2}|w|^{2}+F_{R}(x,v,m^{T}_{t}) \Big)\ \mu_{t}(dx,dv,dw)dt} 
	\\
	 & \qquad +\int_{\T^{d} \times \R^{d} \times \R^{d}}{g_{R}(x,v,m^{T}_{T})\ \mu_{T}(dx,dv,dw)} \Bigr\}
	\\
	= & \int_{\T^{d} \times \R^{d}}{u^{T}_{R}(0,x,v)\ m_{0}(dx,dv)}. 	
	\end{align*}
\end{proposition}

\begin{proof}
	We first prove that
	\begin{align}\label{eq:hand1}
	\begin{split}
	& \inf_{\mu \in \C^{T}(m_{0})} \Bigl\{  \int_{0}^{T}\int_{\T^{d} \times \R^{d} \times \R^{d}}{\Big(\frac{1}{2}|w|^{2}+F_{R}(x,v,m^{T}_{t}) \Big)\ \mu_{t}(dx,dv,dw)dt} 
	\\
	& \qquad\qquad +\  \int_{\T^{d} \times \R^{d} \times \R^{d}}{g_{R}(x,v,m^{T}_{T})\ \mu_{T}(dx,dv,dw)} \Bigr\}
	\\
	 &\qquad  \geq \int_{\T^{d} \times \R^{d}}{u^{T}_{R}(0,x,v)\ m_{0}(dx,dv)}.
	\end{split} 	
	\end{align}
	 
	 We have that
	 \begin{align*}
	&  \inf_{\mu \in \C^{T}(m_{0})} \Bigl\{\int_{0}^{T}\int_{\T^{d} \times \R^{d} \times \R^{d}}{\Big(\frac{1}{2}|w|^{2}+F_{R}(x,v,m^{T}_{t}) \Big)\ \mu_{t}(dx,dv,dw)dt} 
	\\
	& \qquad \qquad +\ \int_{\T^{d} \times \R^{d} \times \R^{d}}{g_{R}(x,v,m^{T}_{T})\ \mu_{T}(dx,dv,dw)}\Bigr\}
	\\
	= & \inf_{\mu \in C([0,T]; \PP_1(\T^{d} \times \R^{d} \times \R^{d}))}\sup_{\varphi \in C^{\infty}_{c}([0,T] \times \T^{d} \times \R^{d})}\int_{0}^{T}\int_{\T^{d} \times \R^{d} \times \R^{d}}{\Big(\frac{1}{2}|w|^{2} + F_{R}(x,v, m^{T}_{t}) + \partial_{t} \varphi(t,x,v)}
	\\
	 & \qquad +\ \langle D_{x}\varphi(t,x,v), v \rangle + \langle D_{v}\varphi(t,x,v), w \rangle \Big)\ \mu_{t}(dx,dv,dw)dt 
	\\
	+\ & \int_{\T^{d} \times \R^{d} \times \R^{d}}{\Big(g_{R}(x,v,m^{T}_{T})-\varphi(T,x,v) \Big)\ \mu_{T}(dx,dv,dw)} + \int_{\T^{d} \times \R^{d}}{\varphi(0,x,v)\ m_{0}(dx,dv)}
	\\
	\geq & \sup_{\varphi \in C^{\infty}_{c}([0,T] \times \T^{d} \times \R^{d})} \inf_{\mu \in C([0,T]; \PP_1(\T^{d} \times \R^{d} \times \R^{d}))}\int_{0}^{T}\int_{\T^{d} \times \R^{d} \times \R^{d}}{\Big(\frac{1}{2}|w|^{2} + F_{R}(x,v, m^{T}_{t}) + \partial_{t} \varphi(t,x,v)}
	\\
	& \qquad +\  \langle D_{x}\varphi(t,x,v), v \rangle + \langle D_{v}\varphi(t,x,v), w \rangle \Big)\ \mu_{t}(dx,dv,dw)dt 
	\\
	+\ & \int_{\T^{d} \times \R^{d} \times \R^{d}}{\Big(g_{R}(x,v,m^{T}_{T})-\varphi(T,x,v) \Big)\ \mu_{T}(dx,dv,dw)} + \int_{\T^{d} \times \R^{d}}{\varphi(0,x,v)\ m_{0}(dx,dv)}. 
	 \end{align*}

	In the argument below, the constant $c_R$ depends on $R$ and on the data and may change from line to line.  Let $\xi^{1, \eps}= \xi^{1,\eps}(x)$ be a smooth mollifier  such that $\supp(\xi^{1,\eps}) \subset B_{\eps}$, $\xi^{1,\eps}(x) \geq 0$ and $\int_{B_{\eps}}{\xi^{1,\eps}(x)\ dx}=1$, and define $u^{\eps, R}_{1} =  u^{T}_{R} \star_x \xi^{1, \eps}(t,x,v)$ (the convolution being in the $x$ variable only). Let $R^{'} \geq R$ be such that $\supp(u^{T}_{R})$, $\supp(F_{R})$ and $\supp(g_{R})$ are contained in $B_{R^{'}}$.  Then, we have that $u^{\eps, R}_{1}$ satisfies the following inequality in the viscosity sense
	 \begin{align*}
	 & -\partial_{t} u^{\eps, R}_{1}(t,x,v) + \frac{1}{2}|D_{v} u^{\eps, R}_{1}(t,x,v)|^{2} - \langle D_{x} u^{\eps, R}_{1}(t,x,v), v \rangle  \leq F_{R} \star \xi^{1, \eps}(t,x,v) 
	 \\
	  & \qquad \leq\ F_{R}(x,v, m^{T}_{t}) + C_{F}\eps(1+|v|^{\alpha}){\bf 1}_{(x,v) \in \T^{d} \times B_{R^{'}}}. 	
	 \end{align*}
	 Now, let $\xi^{2, \eps}=\xi^{2, \eps}(v)$ be a smooth mollifier such that $\supp(\xi^{2,\eps}) \subset B_{\eps}$, $\xi^{2,\eps}(v) \geq 0$ and $\int_{B_{\eps}}{\xi^{2,\eps}(v)\ dv}=1$ and define $u^{R, \eps}_{2}=\xi^{2,\eps} \star_v u^{R, \eps}_{1}(t,x,v)$ (the convolution being now in the $v$ variable only). Then, by the Lipschitz regularity of $u^T_R$ stated in \Cref{lem:approx2} we have that
	 \begin{equation*}
	 |\xi^{2, \eps} \star_v \langle D_{x} u^{R, \eps}_{1}(t,x,\cdot), \cdot \rangle(v) - \langle D_{x} u^{R, \eps}_{2}(t,x,v), v \rangle | \leq  \eps \| D_{x} u^{R, \eps}_{1} \|_{\infty} \leq c_R \eps  {\bf 1}_{(x,v) \in \T^{d} \times B_{R^{'}}}. 
	 \end{equation*}
Hence $u^{\eps, R}_{2}$ satisfies in the viscosity sense: 
	 \begin{align*}
	 & -\partial_{t} u^{\eps, R}_{2}(t,x,v) + \frac{1}{2}|D_{v} u^{\eps, R}_{2}(t,x,v)|^{2} - \langle D_{x} u^{\eps, R}_{2}(t,x,v), v \rangle 
	 \\
	  & \qquad \leq\ F_{R}(x,v, m^{T}_{t}) + c_R\eps{\bf 1}_{(x,v) \in \T^{d} \times B_{R^{'}}}. 	
	 \end{align*}
We finally regularize $u^{\eps, R}_{2}$ in time. 
Let $\xi^{3, \eps}=\xi^{3,\eps}(t)$ be a smooth mollifier  such that $\supp(\xi^{2,\eps}) \subset B_{\eps}$, $\xi^{2,\eps}(t) \geq 0$ and $\int_{B_{\eps}}{\xi^{2,\eps}(t)\ dt}=1$ and define $u^{R, \eps}_{3}=\xi^{3, \eps} \star_t u^{R, \eps}_{2}(t,x,v)$ (convolution in time). Thus, $u^{R, \eps}_{3}$, for any $(t,x,v) \in (-\infty, T-\eps] \times \T^{d} \times \R^{d}$,  satisfies (in the classical sense)
\begin{align*}
	& -\partial_{t} u^{R, \eps}_{3}(t,x,v) + \frac{1}{2}|D_{v} u^{R, \eps}_{3}(t,x,v)|^{2} - \langle D_{x}u^{R, \eps}_{3}(t,x,v), v \rangle 
	\\
	& \qquad \leq  \xi^{3,\eps}\star_t F_{R}(x,v, m^{T}_{\cdot})(t) + c_R\eps {\bf 1}_{(x,v) \in \T^{d} \times B_{R^{'}}}.
\end{align*}
By \cite[Theorem 5.9]{bib:CM} we know that $m^{T}$ is Lipschitz continuous in time with respect to the $d_{1}$ distance. Setting  $\hat{u}^{R}_{\eps}(t,x,v) = u^{R, \eps}_{3}(t-\eps, x,v)$,  $\hat{u}^{R}_{\eps}$ satisfies therefore 
%
\begin{align}\label{eq:appHJ}
\begin{split}
	-\ & \partial_{t} \hat{u}^{R}_{\eps}(t,x,v) + \frac{1}{2}|D_{v} \hat{u}^{R}_{\eps}(t,x,v)|^{2} - \langle D_{x}\hat{u}^{R}_{\eps}(t,x,v), v \rangle 
	\\
	 &\qquad \leq\  F_{R}(x,v, m^{T}_{t}) + c_R \eps {\bf 1}_{(x,v) \in \T^{d} \times B_{R^{'}}} .
	\end{split} 
\end{align}
We note that $\hat{u}^{R}_{\eps}$ is smooth and has a compact support and converges uniformly to $u^R$ as $\eps\to 0$. 
Using $\hat{u}^{R}_{\eps}$ as test function we get
\begin{align*}
	& \sup_{\varphi \in C^{\infty}_{c}([0,T] \times \T^{d} \times \R^{d})} \inf_{\mu \in C([0,T]; \PP_1(\T^{d} \times \R^{d} \times \R^{d}))}\int_{0}^{T}\int_{\T^{d} \times \R^{d} \times \R^{d}}{\Big(\frac{1}{2}|w|^{2} + F_{R}(x,v, m^{T}_{t}) + \partial_{t} \varphi(t,x,v)}
	\\
	&\qquad +  \langle D_{x}\varphi(t,x,v), v \rangle + \langle D_{v}\varphi(t,x,v), w \rangle \Big)\ \mu_{t}(dx,dv,dw)dt 
	\\
	+ & \int_{\T^{d} \times \R^{d} \times \R^{d}}{\Big(g_{R}(x,v,m^{T}_{T})-\varphi(T,x,v) \Big)\ \mu_{T}(dx,dv,dw)} + \int_{\T^{d} \times \R^{d}}{\varphi(0,x,v)\ m_{0}(dx,dv)}
	\\
	\geq &   \inf_{\mu \in C([0,T]; \PP_1(\T^{d} \times \R^{d} \times \R^{d}))}\int_{0}^{T}\int_{\T^{d} \times \R^{d} \times \R^{d}}{\Big(\frac{1}{2}|w|^{2} + F_{R}(x,v, m^{T}_{t}) + \partial_{t} \hat{u}^{R}_{\eps}(t,x,v)}
	\\
	& \qquad +  \langle D_{x}\hat{u}^{R}_{\eps}(t,x,v), v \rangle + \langle D_{v}\hat{u}^{R}_{\eps}(t,x,v), w \rangle \Big)\ \mu_{t}(dx,dv,dw)dt 
	\\
	+ & \int_{\T^{d} \times \R^{d} \times \R^{d}}{\Big(g_{R}(x,v,m^{T}_{T})-\hat{u}^{R}_{\eps}(T, x,v) \Big)\ \mu_{T}(dx,dv,dw)} + \int_{\T^{d} \times \R^{d}}{\hat{u}^{R}_{\eps}(0,x,v)\ m_{0}(dx,dv)}
	\\
	=\ &   \inf_{(t,x,v) \in [0,T] \times \T^{d} \times \R^{d}} \Bigl\{  \frac{1}{2}|D_{v} \hat{u}^{R}_{\eps}(t,x,v)|^{2} + F_{R}(x,v, m^{T}_{t}) - \partial_{t} \hat{u}^{R}_{\eps}(t,x,v) + \langle D_{x}\hat{u}^{R}_{\eps}(t,x,v), v \rangle
	\\
	& \qquad +\  g_{R}(x,v,m^{T}_{T})-\hat{u}^{R}_{\eps}(T,x,v) \Bigr\} + \int_{\T^{d} \times \R^{d}}{\hat{u}^{R}_{\eps}(0,x,v)\ m_{0}(dx,dv)}. 
\end{align*}
	 By \cref{eq:appHJ} we obtain that
	 \begin{align*}
	 	&   \inf_{(t,x,v) \in [0,T] \times \T^{d} \times \R^{d}}\Bigl\{ \Big(\frac{1}{2}|D_{v} \hat{u}^{R}_{\eps}(t,x,v)|^{2} + F_{R}(x,v, m^{T}_{t}) + \partial_{t} \hat{u}^{R}_{\eps}(t,x,v) + \langle D_{x}\hat{u}^{R}_{\eps}(t,x,v), v \rangle \Big)
	\\
	 & \qquad +\ g_{R}(x,v,m^{T}_{T})- \hat{u}^{R}_{\eps}(T,x,v)\Bigr\}  + \int_{\T^{d} \times \R^{d}}{\hat{u}^{R}_{\eps}(0,x,v)\ m_{0}(dx,dv)}
	\\
	 \geq & -c_R \eps +  \inf_{(x,v) \in  \T^{d} \times \R^{d}} \Bigl\{ \ g_{R}(x,v,m^{T}_{T})- \hat{u}^{R}_{\eps}(T,x,v)\Bigr\}	 + \int_{\T^{d} \times \R^{d}}{\hat{u}^{R}_{\eps}(0,x,v)\ m_{0}(dx,dv)}.
	 \end{align*}
As $\eps \to 0^{+}$ we obtain \cref{eq:hand1}.

	On the other hand, since $u^{T}_{R}$ is a continuous viscosity solution we know that it can be represented as follows:
	\begin{equation}\label{eq:appValuefunction}
		u^{T}_{R}(0,x,v)=\inf_{\gamma \in \Gamma_{0}(x,v)}\left\{\int_{0}^{T}{\left(\frac{1}{2}|\ddot\gamma(t)|^{2}+F_{R}(\gamma(t), \dot\gamma(t), m^{T}_{t}) \right)\ dt} + g_{R}(\gamma(T), \dot\gamma(T), m^{T}_{T}) \right\}.
	\end{equation}
We define the measure $\nu \in C([0,T]; \PP_{1}( \T^{d} \times \R^{d} \times \R^{d}))$ as 
	\begin{equation*}
		\int_{\T^{d} \times \R^{d} \times \R^{d}}{\varphi(x,v,w)\ \nu_{t}(dx,dv,dw)}=\int_{\T^{d}\times \R^{d} \times \R^{d}}{\varphi(\gamma_{(x,v)}(t), \dot\gamma_{(x,v)}(t), \ddot\gamma_{(x,v)}(t))\ m_{0}(dx,dv)},
	\end{equation*}
	for any $\varphi \in C^{\infty}_{c}(\T^{d} \times \R^{d} \times \R^{d})$ and any $t \in [0,T]$, where $\gamma_{(x,v)}$ is a measurable selection of minimizers of problem \cref{eq:appValuefunction}, see \Cref{lem:miniselection}.
By the regularity of the minimizers it is not difficult to prove that $\nu \in \C^{T}(m_{0})$. Moreover, integrating the equality
\begin{equation*}
u^{T}_{R}(0,x,v)= \int_{0}^{T}{\Big(\frac{1}{2}|\ddot\gamma_{(x,v)}(t)|^{2}+F_{R}(\gamma_{(x,v)}(t), \dot\gamma_{(x,v)}(t), m^{T}_{t}) \Big)\ dt} + g_{R}(\gamma_{(x,v)}(T),\dot\gamma_{(x,v)}(T),m^{T}_{T})	
\end{equation*}
against the measure $m_{0}$ we deduce that
\begin{align*}
& \int_{\T^{d} \times \R^{d}}{u^{T}_{R}(0,x,v)\ m_{0}(dx,dv)}  
\\
= & \int_{\T^{d} \times \R^{d}}\int_{0}^{T}{\Big(\frac{1}{2}|\ddot\gamma_{(x,v)}(t)|^{2}+F_{R}(\gamma_{(x,v)}(t), \dot\gamma_{(x,v)}(t), m^{T}_{t}) \Big)\ dt\ m_{0}(dx,dv)} 
\\
& \qquad + \int_{\T^{d} \times \R^{d}}{g_{R}(\gamma_{(x,v)}(T),\dot\gamma_{(x,v)}(T),m^{T}_{T})\ m_{0}(dx,dv)}
\\
=\ & 	\int_{0}^{T}\int_{\T^{d} \times \R^{d} \times \R^{d}}{\Big(\frac{1}{2}|w|^{2}+F_{R}(x, v, m^{T}_{t}) \Big)\ \nu_{t}(dx,dv,dw)dt} + \int_{\T^{d} \times \R^{d} \times \R^{d}}{g_{R}(x,v,m^{T}_{T})\ \nu_{T}(dx,dv,dw)}
\\
\geq & \inf_{\mu \in \C^{T}(m_{0})} \int_{0}^{T}\int_{\T^{d} \times \R^{d} \times \R^{d}}{\Big(\frac{1}{2}|w|^{2}+F_{R}(x,v,m^{T}_{t}) \Big)\ \mu_{t}(dx,dv,dw)dt} 
\\
 &\qquad + \int_{\T^{d} \times \R^{d} \times \R^{d}}{g_{R}(x,v,m^{T}_{T})\ \mu_{T}(dx,dv,dw)}.
\end{align*}
This completes the proof. 
\end{proof}

\begin{proof}[Proof of Theorem \ref{thm:lintegral}]
 Using the notation of \Cref{thm:initialintegral} we know that for any $R \geq 0$
   \begin{align*}
\inf_{\mu \in \C^{T}(m_{0})} & \int_{0}^{T}\int_{\T^{d} \times \R^{d} \times \R^{d}}{\left(\frac{1}{2}|w|^{2}+F_{R}(x,v,m^{T}_{t}) \right)\ \mu_{t}(dx,dv,dw)dt}	
\\
+& \int_{\T^{d} \times \R^{d} \times \R^{d}}{g_{R}(x,v,m^{T}_{T})\ \mu_{T}(dx,dv,dw)}
\\
= & \int_{\T^{d} \times \R^{d}}{u^{T}_{R}(0,x,v)\ m_{0}(dx,dv)}. 
\end{align*} 
Then, on the one hand it is easy to see, by standard optimal control arguments, that for any $(x,v) \in \T^{d} \times \R^{d}$ we have that $|u^{T}_{R}(0,x,v)| \leq \tilde{C}_{1}(1+|v|^{\alpha})$. 
By Dominated Convergence Theorem we get
\begin{equation*}
\lim_{R \to + \infty} \int_{\T^{d} \times \R^{d}}{u^{T}_{R}(0,x,v)\ m_{0}(dx,dv)}=\int_{\T^{d} \times \R^{d}}{u^{T}(0,x,v)\ m_{0}(dx,dv)}. 	
\end{equation*}
On the other hand, without loss of generality we can define a cut-off function $\xi_{R}$ as in \Cref{thm:initialintegral} such that $F_{R}$ and $g_{R}$ are non-decreasing in $R$. Thus
\begin{align*}
	\limsup_{R \to +\infty} \inf_{\mu \in \C^{T}(m_{0})}\ & \int_{0}^{T}\int_{\T^{d} \times \R^{d} \times \R^{d}}{\left(\frac{1}{2}|w|^{2}+F_{R}(x,v,m^{T}_{t}) \right)\ \mu_{t}(dx,dv,dw)dt}	
\\
+\ & \int_{\T^{d} \times \R^{d} \times \R^{d}}{g_{R}(x,v,m^{T}_{T})\ \mu_{T}(dx,dv,dw)}
\\
\leq\ & \inf_{\mu \in \C^{T}(m_{0})} \int_{0}^{T}\int_{\T^{d} \times \R^{d} \times \R^{d}}{\left(\frac{1}{2}|w|^{2}+F(x,v,m^{T}_{t}) \right)\ \mu_{t}(dx,dv,dw)dt}	
\\
+\ & \int_{\T^{d} \times \R^{d} \times \R^{d}}{g(x,v,m^{T}_{T})\ \mu_{T}(dx,dv,dw)}.
\end{align*}
To prove the reverse inequality, let $\{R_{j}\}_{j \in \N}$ and $\{\mu^{j}_{t}\}_{j \in \N} \subset \C^{T}(m_{0})$ be such that
\begin{align*}
		\liminf_{R \to +\infty} \inf_{\mu \in \C^{T}(m_{0})}\ & \int_{0}^{T}\int_{\T^{d} \times \R^{d} \times \R^{d}}{\left(\frac{1}{2}|w|^{2}+F_{R}(x,v,m^{T}_{t}) \right)\ \mu_{t}(dx,dv,dw)dt}	
\\
+\ & \int_{\T^{d} \times \R^{d} \times \R^{d}}{g_{R}(x,v,m^{T}_{T})\ \mu_{T}(dx,dv,dw)}
\\
= & 	\lim_{j \to +\infty} \inf_{\mu \in \C^{T}(m_{0})}\ \int_{0}^{T}\int_{\T^{d} \times \R^{d} \times \R^{d}}{\left(\frac{1}{2}|w|^{2}+F_{R_{j}}(x,v,m^{T}_{t}) \right)\ \mu_{t}(dx,dv,dw)dt}	
\\
+\ & \int_{\T^{d} \times \R^{d} \times \R^{d}}{g_{R_{j}}(x,v,m^{T}_{T})\ \mu_{T}(dx,dv,dw)}
\\
= & 	\lim_{j \to +\infty} \ \int_{0}^{T}\int_{\T^{d} \times \R^{d} \times \R^{d}}{\left(\frac{1}{2}|w|^{2}+F_{R_{j}}(x,v,m^{T}_{t}) \right)\ \mu^{j}_{t}(dx,dv,dw)dt}	
\\
+\ & \int_{\T^{d} \times \R^{d} \times \R^{d}}{g_{R_{j}}(x,v,m^{T}_{T})\ \mu^{j}_{T}(dx,dv,dw)}.
\end{align*}
We claim that $\{ \mu_{t}^{j}\}_{j \in \N}$ is tight. Indeed,  the lower bound on $F$ and $g$, there exists a constant $C \geq 0$ such that
\begin{equation}\label{eq:wbound}
\sup_j \int_{0}^{T}\int_{\T^{d} \times \R^{d} \times \R^{d}}{|w|^{2}\ \mu_{t}^{j}(dx,dv,dw)dt} \leq C
\end{equation}
and thus it is enough to prove that the moment with respect to $v$ is also bounded. In order to prove this bound, let $\psi \in C^{\infty}_{c}(\R^{d})$ with $\psi(0)=0 $ and such that $|D\psi(p)| \leq 1$. For $\varphi(t,x,v)=(T-t)\psi(v)$, we have, by the definition of a $T-$closed measure in \cref{eq:Tclosedcondition},
\begin{align}\label{eq:attempt1}
\int_{0}^{T}\int_{\T^{d} \times \R^{d} \times \R^{d}}{\big(-\psi(v) + (T-t)\langle D\psi(v), w \rangle \big)\ \mu_{t}^{j}(dx,dv,dw)dt}= -T\int_{\T^{d} \times \R^{d}}{\psi(v)\ m_{0}(dx,dv)}	
\end{align}
and by \cref{eq:wbound} and Cauchy-Schwarz inequality we get
\begin{align*}
\Bigl|\int_{0}^{T}\int_{\T^{d} \times \R^{d} \times \R^{d}}{(T-t)\langle D\psi(v), w \rangle\ \mu_{t}^{j}(dx,dv,dw)dt} \Bigr| \leq TC^{1/2}.	
\end{align*}
Thus, by \cref{eq:attempt1} we obtain that
\begin{align*}
\Bigl|\int_{\T^{d} \times \R^{d} \times \R^{d}}{\psi(v)\ \mu_{t}^{j}(dx,dv,dw)dt}	\Bigr| \leq C, 
\end{align*}
for some new constant $C$. 
If we choose  $\psi_{n}$ such that $\psi_n(v)$ increases in $n$ and converges to $|v|$, we get therefore
\begin{align*}
\int_{0}^{T}\int_{\T^{d} \times \R^{d} \times \R^{d}}{|v|\ \mu_{t}^{j}(dx,dv,dw)dt} \leq C. 
\end{align*}
This implies that $\{\mu_{t}^{j}\}_{j \in \N}$ is tight and, up to a subsequence still denoted by $\mu_{t}^{j}$, converges to some $\bar\mu \in \C^{T}(m_{0})$. Then, we have that
\begin{align*}
& \inf_{\mu \in \C^{T}(m_{0})} \int_{0}^{T}\int_{\T^{d} \times \R^{d} \times \R^{d}}{\left(\frac{1}{2}|w|^{2}+F(x,v, m^{T}_{t}) \right)\ \mu_{t}(dx,dv,dw)dt}	
\\
+ & \int_{\T^{d} \times \R^{d} \times \R^{d}}{g(x,v,m^{T}_{T})\ \mu_{T}(dx,dv,dw)}
\\
\leq & \int_{0}^{T}\int_{\T^{d} \times \R^{d} \times \R^{d}}{\left(\frac{1}{2}|w|^{2}+F(x,v, m^{T}_{t}) \right)\ \bar\mu_{t}(dx,dv,dw)dt}	
\\
+ & \int_{\T^{d} \times \R^{d} \times \R^{d}}{g(x,v,m^{T}_{T})\ \bar\mu_{T}(dx,dv,dw)}
\\
\leq  & \lim_{j \to +\infty}  \int_{0}^{T}\int_{\T^{d} \times \R^{d} \times \R^{d}}{\left(\frac{1}{2}|w|^{2}+F_{R_{j}}(x,v, m^{T}_{t}) \right)\ \mu_{t}^{j}(dx,dv,dw)dt}	
\\
+ & \int_{\T^{d} \times \R^{d} \times \R^{d}}{g_{R_{j}}(x,v,m^{T}_{T})\ \mu_{T}^{j}(dx,dv,dw)}
\\
= & \liminf_{j \to +\infty} \inf_{\mu \in \C^{T}(m_{0})}  \int_{0}^{T}\int_{\T^{d} \times \R^{d} \times \R^{d}}{\left(\frac{1}{2}|w|^{2}+F_{R_{j}}(x,v, m^{T}_{t}) \right)\ \mu_{t}(dx,dv,dw)dt}	
\\
+ & \int_{\T^{d} \times \R^{d} \times \R^{d}}{g_{R_{j}}(x,v,m^{T}_{T})\ \mu_{T}(dx,dv,dw)}.
\end{align*}
This completes the proof of  equality \eqref{eq:lintegral}.

{ It remain to check the existence of a minimizer $\bar \mu^T\in \C^T(m_0)$ of the problem in the left-hand side such that $m^T_{t} = \pi \sharp \bar\mu^T_{t}$. For this, let $\gamma_{(x,v)}$ denote the measurable selection of minimizers of $u^T(0, x,v)$ in \eqref{eq:MFGvaluefunction} as  in \Cref{lem:miniselection} below and define the measure
\begin{equation*}
\bar \mu^T_{t} = \left((x,v)\to (\gamma_{(x,v)}(t), \dot\gamma_{(x,v)}(t), D_{v}u^{T}(t, \gamma_{(x,v)}(t), \dot\gamma_{(x,v)}(t))) \right) \sharp m_{0}	
\end{equation*}
for any $t \in [0,T]$. Note that by \cite[Lemma 3.5]{bib:YA} $\bar \mu^T_{t}$ is well-defined since $u(t,x,\cdot)$ is differentiable along the optimal trajectory $\gamma_{(x,v)}$ with 
\begin{equation*}
\ddot\gamma_{(x,v)}(t)=D_{v}u^{T}(t, \gamma_{(x,v)}(t), \dot\gamma_{(x,v)}(t)), \quad t \in [0,T]
\end{equation*}
In particular, it is easy to see that $\bar \mu^T \in \C^{T}(m_{0})$ and moreover, by \cite[Proposition 4.2]{bib:YA} we have that $m^{T}_{t}=\pi \sharp \bar \mu^T_{t}$ since $m^{T}_{t}=((x,v)\to (\gamma_{(x,v)}(t), \dot \gamma_{(x,v)}(t))) \sharp m_{0}$. 
%
%
By the representation formula of the value function we have that
\begin{align*}
u^{T}(0,x,v) = & \int_{0}^{T}{\Big(\frac{1}{2}|\ddot\gamma_{(x,v)}(t)|^{2} + F(\gamma_{(x,v)}(t), \dot\gamma_{(x,v)}(t), m^{T}_{t}) \Big)\ dt} + g(\gamma_{(x,v)}(T), \dot\gamma_{(x,v)}(T), m^{T}_{T})	
\\
= & \int_{0}^{T}{\Big(\frac{1}{2}|D_{v}u^{T}(t, \gamma_{(x,v)}(t), \dot\gamma_{(x,v)}(t))|^{2} + F(\gamma_{(x,v)}(t), \dot\gamma_{(x,v)}(t), m^{T}_{t}) \Big)\ dt} 
\\
& \qquad +\  g(\gamma_{(x,v)}(T), \dot\gamma_{(x,v)}(T), m^{T}_{T}).	
\end{align*}
Integrating both side against the measure $m_{0}$ and using the definition of $\bar \mu^T$, we obtain that $\bar \mu^T$ satisfies the equality in \eqref{eq:lintegral} and therefore is optimal.  
}
\end{proof}

\begin{lemma}\label{lem:miniselection}
	Assume that $F$ satisfies {\bf (F1')} and {\bf (F2')}  and $g$ satisfies {\bf (G1)}. For $(x,v) \in \T^{d} \times \R^{d}$ let $\Gamma^{*}(x,v) \subset \Gamma_{0}(x,v)$ be the set of minimizers of problem \cref{eq:MFGvaluefunction} for $t=0$. Then, the set-valued map
		\begin{equation*}
	\Gamma^{*}: (\T^{d} \times \R^{d}, |\cdot|) \rightrightarrows (\Gamma, \| \cdot\|_{\infty}), \quad (x,v) \mapsto \Gamma^{*}(x,v)
	\end{equation*}
 has a measurable selection $\gamma_{(x,v)}$, i.e. $(x,v)\to \gamma_{(x,v)}$ is measurable and, for any $(x,v) \in \T^{d} \times \R^{d}$, $\gamma_{(x,v)} \in \Gamma^{*}(x,v)$. 
\end{lemma}
\begin{proof}
By using classical results from optimal control theory it is not difficult to see that $\Gamma^*$ has a closed graph, see for instance \cite[Lemma 4.1]{bib:CM}. Therefore, by \cite[Proposition 9.5]{bib:C} the set-valued map $(x,v) \rightrightarrows \Gamma^{*}(x,v)$ is measurable with closed values. This implies  by \cite[Theorem A 5.2]{bib:SC} the existence of a measurable selection $\gamma_{(x,v)} \in \Gamma^{*}(x,v)$.	
\end{proof}

\subsection{Convergence of the solution of the time dependent MFG system}  We now investigate the limit as the horizon $T\to +\infty$ of the time-dependent MFG problem. The main result of this subsection is the following proposition:  

\begin{proposition}[{\bf Convergence of MFG solution}]\label{prop:convergenceMFG}
	Assume that $F$ satisfies {\bf (F1')}, {\bf (F2')}, {\bf (F3')} with $\alpha=2$  and the monotonicity condition \eqref{monotonicity}, that $g$ satisfies {\bf (G1)} and that the initial distribution $m_{0}$ in \cref{eq:accMFG} belongs to $\mathcal P_2(\T^d\times \R^d)$. 
	Let $(u^{T}, m^{T})$ be a solution of the MFG system \cref{eq:accMFG} and let $(\bar\lambda, \bar\mu)$ be the solution of the ergodic MFG problem \cref{eq:Ergodicproblem}. Then
	\begin{equation*}
	\lim_{T \to +\infty} \frac{1}{T} \int_{\T^{d} \times \R^{d}}{u^{T}(0,x,v)\ m_{0}(dx,dv)}= \bar\lambda.
	\end{equation*}
\end{proposition}

Throughout the section, we assume that the assumption of Proposition \ref{prop:convergenceMFG} are in force. 
The proof of the proposition---given at the end of the subsection---is made at the level of the closed and $T-$closed measures. For this we first need to discuss how to manipulate them. The first lemma is a straightforward application of the definition of $T-$closed measures:

\begin{lemma}[{\bf Concatenation of $T$-closed measure}]\label{lem:concatenation}
	Let $T,T'>0$, $m_{0} \in \mathcal P_2(\T^d\times \R^d)$,  $\mu_{1} \in \C^{T}(m_{0})$ and  $\mu_{2} \in \C^{T'}(m_{1})$ with $m_{1}=\pi \sharp \mu_{1}(T)$. Then, the measure 
	\begin{align*}
	\mu_{t}:=
	\begin{cases}
		\mu_{1}(t),& \quad t \in [0,T]
		\\
		\mu_{2}(t-T),& \quad t \in (T, T+T']
	\end{cases}
	\end{align*}
belongs to $\C^{T+T'}(m_{0})$. 
\end{lemma}

Next we explain how to link two measures by a $T-$closed measure:

\begin{lemma}[{\bf Linking two measures by a $T$-closed measure}]\label{lem:linkmeasure}
	Let $m_{0}^{1}$ and $m_{0}^{2}$ belong to $\mathcal P_2(\T^d\times \R^d)$. Then, there exists $\mu^{m^1_0\to m^2_0}\in \C^{T=1}(m_{0}^1)$ such that $m_{0}^{2}=\pi \sharp \mu^{m^1_0\to m^2_0}_1$ and
\begin{align}\label{estimum1m2}
& \int_0^1 \int_{\T^d\times \R^{2d}} (\frac12 |w|^2+ c_F(1+|v|^2)) \ \mu^{m^1_0\to m^2_0}_t(dx,dv,dw)dt
\leq C_2(1+ M_2(m_0^1)+M_2(m_0^2)), 
\end{align}
where $M_2(m)= \int_{\T^d\times \R^d} |v|^2dm(x,v)$ (for $m\in \PP_2(\T^d\times \R^d)$) and where $C_2$ depends only on $\alpha$ and $c_F$. 
\end{lemma}
\begin{proof}
Let $(x_0,v_0) \in \supp(m_{0}^{1})$ and let $(x,v) \in \supp(m_{0}^{2})$. Then, following the proof of  \Cref{lem:reachablecurve},  there exists a curve $\sigma _{(x_0,v_0)}^{(x,v)}:[0,1] \to \T^{d}$	such that $\sigma _{(x_0,v_0)}^{(x,v)}(0)=x_0$, $\dot\sigma _{(x_0,v_0)}^{(x,v)}(0)=v_0$ and $\sigma _{(x_0,v_0)}^{(x,v)}(1)=y$, $\dot\sigma _{(x_0,v_0)}^{(x,v)}(1)=w$ with 
\begin{equation}\label{kajhezskndc}
\int_0^1 (\frac12 |\ddot \sigma _{(x_0,v_0)}^{(x,v)}(t)|^2 + c_F(1+ |\dot \sigma _{(x_0,v_0)}^{(x,v)}(t)|^2)) dt \leq 
 C_{2}(1+|v|^2+|v_0|^2). 
\end{equation}
 Moreover, by construction, 
$\sigma$ depends continuously on $(x_0,v_0,x,v)$. 
Let $\lambda\in \Pi(m_0^1,m_0^2)$ be a transport plan between $m_{0}^{1}$ and $m_{0}^{2}$ (see \eqref{def.transportplan}). We define the measure $\mu^{m^1_0\to m^2_0} \in \C^1(m_0^1)$ 
by 
$$
\int_{\T^d\times \R^d\times \R^d} \varphi(x,v,w) \mu^{m^1_0\to m^2_0}_t (dx,dv,dw) = 
\int_{(T^d\times \R^d)^2} \varphi(\sigma _{(x_0,v_0)}^{(x,v)}(t), \dot \sigma _{(x_0,v_0)}^{(x,v)}(t), \ddot \sigma _{(x_0,v_0)}^{(x,v)}(t))\ \lambda(dx_0,dv_0,dx,dv)
$$   
for any $\varphi\in C^\infty_c(\T^d\times \R^d\times \R^d)$. 
Then, on easily checks that  $m_{0}^{2}=\pi \sharp \mu^{m^1_0\to m^2_0}_1$ and that,  by \eqref{kajhezskndc}: 
\begin{align*}
& \int_0^1 \int_{\T^d\times \R^{2d}} (\frac12 |w|^2+ c_F(1+|v|^2))\ \mu^{m^1_0\to m^2_0}_t(dx,dv,dw)dt\\
& \qquad \leq 
C_{2}\int_0^1\int_{\T^d\times \R^{2d}}(1+|v|^2+|v_0|^2)\ \mu^{m^1_0\to m^2_0}_t(dx,dv,dw)dt = 
C_2(1+ M_2(m_0^1)+M_2(m_0^2)). 
\end{align*}
\end{proof}

\begin{proposition}[{\bf Energy estimate}]\label{prop:energyestimate}
Under the notation and assumption of Proposition \ref{prop:convergenceMFG}, 
 there exists a constant $C \geq 0$ (independent of $T$) such that 
\begin{equation}\label{eq:energyest}
\int_0^T \sup_{(x,v)\in \T^d\times \R^d} \frac{| F(x,v, m^T_t)- F(x,v, \bar{m})|^{2d+2}}{(1+|v|^2)^{2d}} dt \leq\  C T^{\frac{1}{2}},
\end{equation}
where $\bar{m}=\pi \sharp \bar\mu$, with $\pi(x,v,w)=(x,v)$. 
	\end{proposition}
	
\begin{proof} The proof consists in building from $\bar \mu$ and $\mu^T$ competitors in problems \eqref{eq:Ergodicproblem} and \eqref{eq:lintegral} respectively. Let us recall that $\mu^T$ and $\bar \mu$ are minimizers for these respective problems. \\

We start with problem \eqref{eq:lintegral}. Fix $T\geq 2$. We define the measure $\tilde\mu^T$ by 
		\begin{align}\label{eq:measure1}
		\tilde\mu^T_{t}=
		\begin{cases}
			\mu^{m_{0} \to \bar{m}}_{t}, & \quad t \in [0,1] \\
			\bar\mu, & \quad t \in (1, T],
		\end{cases}
		\end{align}
where $\mu^{m_{0} \to \bar{m}}$ is the measure defined by Lemma \ref{lem:linkmeasure}. We know by \Cref{lem:concatenation} that $\tilde\mu^T$ belongs to $\C^{T}(m_{0})$. So we can use $\tilde\mu^T$ as a competitor  in problem  \eqref{eq:lintegral} to get
	\begin{align}\label{eq:first}
	\begin{split}
	& \int_{0}^{T}\int_{\T^{d} \times \R^{d} \times \R^{d}}{\left(\frac{1}{2}|w|^{2} + F(x,v, m^{T}_{t})\right)\ \mu^{T}_{t}(dx,dv,dw)dt}
	\\
	& \qquad \qquad +\ \int_{\T^{d} \times \R^{d} \times \R^{d}}{g(x,v,m^{T}_{T})\ \mu^{T}_{T}(dx,dv,dw)}	
	\\
	&\qquad \leq\  \int_{0}^{1}\int_{\T^{d} \times \R^{d} \times \R^{d}}{\left(\frac{1}{2}|w|^{2} + F(x,v, m^{T}_{t})\right)\ \mu^{m_{0} \to \bar{m}}_{t}(dx,dv,dw)dt}
	\\
	& \qquad \qquad +\ \int_1^T\int_{\T^{d} \times \R^{d} \times \R^{d}}{\left(\frac{1}{2}|w|^{2} + F(x,v, m^{T}_{t})\right)\ \bar{\mu}(dx,dv,dw)}dt
	\\
	& \qquad \qquad +\ \int_{\T^{d} \times \R^{d} \times \R^{d}}{g(x,v,m^{T}_{T})\ \bar\mu(dx,dv,dw)}	. 
	\end{split}
	\end{align}
	
Next we build from $\mu^T$ a competitor for the minimization problem \cref{eq:Ergodicproblem} for which $\bar \mu$ is a minimizer. 
In view of \cite[Proposition 4.2]{bib:YA} (see also \cite[Theorem 6.3]{bib:CM}) there exists a Borel measurable maps $(x,v)\to \gamma_{(x,v)}$ such that, for each $(x,v)\in \T^d\times \R^d$, $\gamma_{(x,v)}$ is a minimizer for $u^T(0,x,v)$ in \eqref{eq:lintegral} and satisfies 
\begin{align}\label{rep.muT}
&  \int_0^T\int_{\T^d\times \R^{2d}} \varphi(x,v,w)\  \mu^T_t(dx,dv,dw) dt =  \int_{\T^d\times \R^d} \int_0^T \varphi( \gamma_{(x,v)}(t),
\dot{ \gamma}_{(x,v)}(t),\ddot{  \gamma}_{(x,v)}(t)) dtm_{0}(dx,dv)
\end{align}
for any test function $\varphi\in C^0_b(\T^d\times \R^{2d})$. By Lemma \ref{lem.jkhzrnedg} and  Remark \ref{rem.jkhzrnedg}, for any $\lambda\geq 2$, there exist Borel measurable maps $(x,v)\to \tilde \gamma_{(x,v)}$ and $(x,v)\to \tau_{(x,v)}$ such that 
\begin{equation}\label{eq.cloclo}
\mbox{\rm $\tilde \gamma_{(x,v)}(0)=
\tilde \gamma_{(x,v)}(T)=x$, $\dot{\tilde \gamma}_{(x,v)}(0)=
\dot{\tilde \gamma}_{(x,v)}(T)=v$ and $\tilde \gamma_{(x,v)}=\gamma_{(x,v)}$ on $[0,\tau_{(x,v)}]$ }
\end{equation} and 
\begin{align}\label{ineq.tildegamma}
 \int_{\tau_{(x,v)}}^T (\frac12 |\ddot{\tilde \gamma}_{(x,v)}(t)|^2+c_F(1+ |\dot{\tilde \gamma}_{(x,v)}(t)|^2)dt \leq  C_3(1+|v|)^2 (\lambda^2 + \lambda^{-2}T ) .
\end{align}
Let us define $\hat \mu^T$ by 
\begin{equation}\label{eq:measure2BIS}
\int_{\T^d\times \R^{2d}} \varphi(x,v,w)\ \hat \mu^T(dx,dv,dw) = T^{-1}  \int_{\T^d\times \R^d}\int_0^T \varphi(t,\tilde \gamma_{(x,v)}(t),
\dot{\tilde \gamma}_{(x,v)}(t),\ddot{ \tilde \gamma}_{(x,v)}(t))dt m_0(dx,dv) 
\end{equation}
for any test function $\varphi\in C^0_b( \T^d\times \R^{2d})$. Note that, by \eqref{eq.cloclo}, $ \hat\mu^T$ belongs to $\C$. 
So using the closed measure $\hat \mu^T$ as a competitor in problem \cref{eq:Ergodicproblem} we deduce that
\begin{align}\label{eq:second}
\begin{split}
	& \int_{\T^{d} \times \R^{d} \times \R^{d}}{\left(\frac{1}{2}|w|^{2} + F(x,v, \bar{m})\right)\ \bar\mu(dx,dv,dw)} 
	\\
	 &\qquad 
	 \leq\ \int_{\T^{d} \times \R^{2d}}{\left(\frac{1}{2}|w|^{2} + F(x,v, \bar{m})\right)\ \hat \mu^{T}(dx,dv,dw)}.
	\end{split}
\end{align}
Note that by the definition of $\hat \mu^T$ in \eqref{eq:measure2BIS} and by \eqref{eq.cloclo} and \eqref{ineq.tildegamma},  we have
\begin{align*}
& T \int_{\T^{d} \times \R^{2d}}{\left(\frac{1}{2}|w|^{2} + F(x,v, \bar{m})\right)\ \hat \mu^{T}(dx,dv,dw)}\\
& \; = \int_{\T^d\times \R^d} \int_0^T (\frac12 |\ddot{\tilde \gamma}_{(x,v)}(t)|^2+F(\tilde \gamma_{(x,v)}(t), \dot{\tilde \gamma}_{(x,v)}(t), \bar m))dt \ m_0(dx,dv)\\ 
& \; \leq \int_{\T^d\times \R^d} \Bigl( \int_0^{\tau_{(x,v)}} (\frac12 |\ddot{ \gamma}_{(x,v)}(t)|^2+F( \gamma_{(x,v)}(t), \dot{ \gamma}_{(x,v)}(t), \bar m) )dt \\
& \qquad \qquad +\int_{\tau_{(x,v)}}^T (\frac12 |\ddot{\tilde \gamma}_{(x,v)}(t)|^2+c_F(1+ |\dot{\tilde \gamma}_{(x,v)}(t)|^2))dt \Bigr) \ m_0(dx,dv) \\ 
& \; \leq \int_{\T^d\times \R^d} \Bigl( \int_0^T (\frac12 |\ddot{ \gamma}_{(x,v)}(t)|^2+F( \gamma_{(x,v)}(t), \dot{ \gamma}_{(x,v)}(t), \bar m) )dt \\
& \qquad \qquad +C_3(1+|v|)^2(\lambda^2  +\lambda^{-2}T )  \Bigr) \ m_0(dx,dv).
\end{align*}
Plugging  this inequality into \eqref{eq:second} and using the representation of $\mu^T$ in \eqref{rep.muT} then gives
\begin{align}
\begin{split}\label{eq:secondBIS}
& \int_{\T^{d} \times \R^{d} \times \R^{d}}{\left(\frac{1}{2}|w|^{2} + F(x,v, \bar{m})\right)\ \bar\mu(dx,dv,dw)}  \\
& \qquad \leq  \int_{\T^{d} \times \R^{2d}}{\left(\frac{1}{2}|w|^{2} + F(x,v, \bar{m})\right)\ \hat \mu^{T}(dx,dv,dw)}\\
& \qquad \leq  T^{-1}\int_{0}^{T}\int_{\T^{d} \times \R^{d} \times \R^{d}}{\left(\frac{1}{2}|w|^{2} + F(x,v, \bar m)\right)\ \mu^{T}_{t}(dx,dv,dw)dt}\\
& \qquad \qquad + 2C_3(1+M_2(m_0))(\lambda^2T^{-1}+\lambda^{-2}), 
\end{split}
\end{align}
where $M_2(m_0)=\int_{\T^d\times \R^d} |v|^2dm_0(x,v)$.
Putting together \eqref{eq:first} and \eqref{eq:secondBIS} (multiplied by $T$)  then implies that 
\begin{align*}
&\int_0^T \int_{\T^d\times \R^{2d}} (\frac12 |w|^2+F(x,v, m^T_t)) d \mu^T_t(x,v,w)  +\int_0^T  \int_{\T^{d} \times \R^{2d}}{\left(\frac{1}{2}|w|^{2} + F(x,v, \bar{m})\right)\ \bar\mu(dx,dv,dw)}dt \\
& \leq\  \int_{0}^{1}\int_{\T^d\times \R^{2d}}{\left(\frac{1}{2}|w|^{2} + F(x,v, m^{T}_{t})\right)\ \mu^{m_{0} \to \bar{m}}_{t}(dx,dv,dw)dt}
	\\
	& \qquad +\ \int_1^T\int_{\T^d\times \R^{2d}}{\left(\frac{1}{2}|w|^{2} + F(x,v, m^{T}_{t})\right)\ \bar{\mu}(dx,dv,dw)}dt
	\\
	& \qquad \qquad +\ \int_{\T^d\times \R^{2d}}{g(x,v,m^{T}_{T})\ \bar\mu(dx,dv,dw)}- \int_{\T^d\times \R^{2d}}{g(x,v,m^{T}_{T})\ \mu^T_T(dx,dv,dw)}	
	\\
	&\qquad+  \int_{0}^{T}\int_{\T^d\times \R^{2d}}{\left(\frac{1}{2}|w|^{2} + F(x,v, \bar m)\right)\ \mu^{T}_{t}(dx,dv,dw)dt}\\
	&\qquad + 2C_3(1+M_2(m_0))(\lambda^2+\lambda^{-\alpha}T).
\end{align*}
Using \eqref{estimum1m2} to bound the first term in the right-hand side (note that $\bar m$ belongs to $\PP_{\alpha,2}(\T^d\times \R^d\times \R^d)$ with $\alpha=2$, so that $\bar m\in \PP_2(\T^d\times \R^d)$) we obtain therefore
\begin{align*}
&\int_0^T \int_{\T^d\times \R^{2d}} (F(x,v, m^T_t)- F(x,v, \bar{m}))\ ( \mu^T_t(dx,dv,dw) -\bar\mu(dx,dv,dw))dt\\
&\qquad  \leq\  C_2(1+ M_2(m_0)+M_2(\bar m)) +2\|g\|_\infty
	+ 2C_3(1+M_2(m_0))(\lambda^2+\lambda^{-2}T).
\end{align*}
We now use the  strong monotonicity condition \cref{monotonicity} and choose $\lambda= T^{1/4}$ to get
\begin{align*}
&\int_0^T \int_{\T^d\times \R^{2d}} (F(x,v, m^T_t)- F(x,v, \bar{m}))^2dxdvdt \leq\  C T^{\frac{1}{2}}
\end{align*}
 for a constant $C$ independent of $T$. Recalling that $F$ satisfies {\bf (F3')}, we obtain \eqref{eq:energyest} by the interpolation inequality \Cref{lem.interpol} in the Appendix. 
\end{proof}

\proof[Proof of Proposition \ref{prop:convergenceMFG}] Throughout the proof, $C$ denotes a constant independent of $T$ and which may change from line to line. Let $\mu^{T}\in \C^T(m_0)$ be associated with a solution $(u^T,m^T)$ of the MFG system \eqref{eq:accMFG} as in \Cref{thm:lintegral}. By \Cref{thm:lintegral} we have that
	\begin{align}\label{eq:infpb}
	\begin{split}
	& \frac{1}{T}\int_{\T^{d} \times \R^{d}}{u^{T}(0,x,v)\ m_{0}(dx,dv)}
	\\
	& \qquad =\  \frac{1}{T}\Bigl\{ \int_{0}^{T}\int_{\T^{d} \times \R^{d} \times \R^{d}}{\left(\frac{1}{2}|w|^{2} + F(x,v, m^{T}_{t}) \right)\ \mu^T_{t}(dx,dv,dw)dt}
	\\
& \qquad \qquad 	+\  \int_{\T^{d} \times \R^{d} \times \R^{d}}{g(x,v,m^{T}_{T})\ \mu^T_{T}(dx,dv,dw)}\Bigr\}\\
	& \qquad =\  \inf_{\mu\in \C^T(m_0)} \frac{1}{T}\Bigl\{ \int_{0}^{T}\int_{\T^{d} \times \R^{d} \times \R^{d}}{\left(\frac{1}{2}|w|^{2} + F(x,v, m^{T}_{t}) \right)\ \mu_{t}(dx,dv,dw)dt}
	\\
& \qquad \qquad 	+\  \int_{\T^{d} \times \R^{d} \times \R^{d}}{g(x,v,m^{T}_{T})\ \mu_{T}(dx,dv,dw)}\Bigr\}.
	\end{split}
	\end{align}
	
We first claim that 
\begin{align}\label{uhezljrsdf2}
\begin{split}
& \limsup_{T\to+\infty} \inf_{\mu\in \C^T(m_0)} \frac{1}{T}\Bigl\{ \int_{0}^{T}\int_{\T^{d} \times \R^{d} \times \R^{d}}{\left(\frac{1}{2}|w|^{2} + F(x,v, m^{T}_{t}) \right)\ \mu_{t}(dx,dv,dw)dt} \\
& \qquad \qquad 	+\  \int_{\T^{d} \times \R^{d} \times \R^{d}}{g(x,v,m^{T}_{T})\ \mu_{T}(dx,dv,dw)}\Bigr\}\\
&\qquad  \leq   \inf_{\tilde \mu\in \C} \Bigl\{ \int_{\T^{d} \times \R^{d} \times \R^{d}}\left(\frac{1}{2}|w|^{2} + F(x,v, \bar{m}) \right)\ \tilde \mu(dx,dv,dw) \Bigr\}.
\end{split}
\end{align}
In order to prove the claim, we first note that, by Young's inequality and  \Cref{prop:energyestimate}, we have, for any $\mu\in \C^T(m_0)$, 
\begin{align}\label{eq:11bis}
\begin{split}
	&   \Bigl|\int_{0}^{T}\int_{\T^{d} \times \R^{d} \times \R^{d}}{\left(F(x,v, m^{T}_{t}) -F(x,v, \bar{m}) \right)\ \mu_{t}(dx,dv,dw)dt}\Bigr| \\
	&\leq  \int_{0}^{T}\int_{\T^{d} \times \R^{d} \times \R^{d}}{ \sup_{(x',v')\in \T^d\times \R^d} \frac{| F(x',v', m^T_t)- F(x',v', \bar{m})|}{(1+|v'|^2)^{\frac{d}{d+1}}} \ (1+|v|^2)^{\frac{d}{d+1}} \ \mu_{t}(dx,dv,dw)dt} \\
	&\leq  \frac{T^{\frac{1}{4}}}{2d+2} \int_{0}^{T} \sup_{(x,v)\in \T^d\times \R^d} \frac{| F(x,v, m^T_t)- F(x,v, \bar{m})|^{2d+2}}{(1+|v|^2)^{2d}}dt 
	\\
	&\qquad \qquad + 
	\frac{(2d+1)T^{-\frac{1}{4(2d+1)}}}{2d+2} \int_0^T\int_{\T^{d} \times \R^{d} \times \R^{d}}{(1+|v|^2)^{\frac{2d}{(2d+1)}} \ \mu_{t}(dx,dv,dw)dt} \\
	& \leq C T^{\frac{3}{4}} + T^{-\frac{1}{4(2d+1)}} \int_0^T\int_{\T^{d} \times \R^{d} \times \R^{d}}{(1+|v|^2) \ \mu_{t}(dx,dv,dw)dt}.
	\end{split}
\end{align}
As $g$ is bounded,  we have therefore, for any $\mu\in \C^T(m_0)$, 
\begin{align}\label{zQLESNRDFV}
\begin{split}
&  \frac{1}{T}\Bigl\{ \int_{0}^{T}\int_{\T^{d} \times \R^{d} \times \R^{d}}{\left(\frac{1}{2}|w|^{2} + F(x,v, m^{T}_{t}) \right)\ \mu_{t}(dx,dv,dw)dt}
	\\
& \qquad \qquad 	+\  \int_{\T^{d} \times \R^{d} \times \R^{d}}{g(x,v,m^{T}_{T})\ \mu_{T}(dx,dv,dw)}\Bigr\}\\ 
& \leq   \frac{1}{T}\Bigl\{ \int_{0}^{T}\int_{\T^{d} \times \R^{d} \times \R^{d}}{\left(\frac{1}{2}|w|^{2} + F(x,v, \bar m) \right)\ \mu_{t}(dx,dv,dw)dt}
	\\
& \qquad \qquad 	+T^{-\frac{1}{4(2d+1)}} \int_0^T\int_{\T^{d} \times \R^{d} \times \R^{d}}{(1+|v|^2) \ \mu_{t}(dx,dv,dw)dt} \Bigr\} + CT^{-\frac14}+T^{-1}\|g\|_\infty. 
\end{split}
\end{align}
Given $\tilde \mu \in \C$, we know from \Cref{lem:linkmeasure} that there exists $\mu^{m_0\to \pi\sharp \tilde \mu}$ such that 
\begin{align}\label{zrabesdf}
& \int_0^1 \int_{\T^d\times \R^{2d}} (\frac12 |w|^2+ c_F(1+|v|^2)) \mu^{m_0\to \pi\sharp \tilde \mu}_t(dx,dv,dw)dt
\leq C_2(1+ M_2(m_0)+M_2(\pi\sharp \tilde \mu)).
\end{align}
Let us then define $\tilde \mu^T$ by
		\begin{align*}
		\tilde\mu^T_{t}=
		\begin{cases}
			\mu^{m_0\to \pi\sharp \tilde \mu}_t, & \quad t \in [0,1] \\
			\tilde \mu, & \quad t \in (1, T],
		\end{cases}
		\end{align*}
By \Cref{lem:linkmeasure}, $\tilde \mu^T$ belongs to $\C^T(m_0)$ and we have, in view of \eqref{zrabesdf},  
\begin{align*}
& T^{-1} \int_{0}^{T}\int_{\T^{d} \times \R^{d} \times \R^{d}}\left(\frac{1}{2}|w|^{2} + F(x,v, \bar{m}) \right)\ \tilde \mu^T_{t}(dx,dv,dw)dt \\
&  \leq C_2T^{-1}(1+ M_2(m_0)+M_2(\pi\sharp \tilde \mu)) + T^{-1}(T-1) \int_{\T^{d} \times \R^{d} \times \R^{d}}\left(\frac{1}{2}|w|^{2} + F(x,v, \bar{m}) \right)\ \tilde  \mu(dx,dv,dw)
\end{align*}
while 
\begin{align*}
&\int_0^T\int_{\T^{d} \times \R^{d} \times \R^{d}}{(1+|v|^2) \ \tilde \mu^T_{t}(dx,dv,dw)dt} 
 \leq C_2(1+ M_2(m_0)+M_2(\pi\sharp \tilde  \mu)) + (T-1) M_2(\pi\sharp\tilde  \mu). 
\end{align*}
Therefore, coming back to \eqref{zQLESNRDFV} and using the $\tilde \mu^T$ built as above from the $\tilde  \mu\in \C$ as competitors, we have  
\begin{align}\label{uhezljrsdf}
\begin{split}
& \inf_{\mu\in \C^T(m_0)} \frac{1}{T}\Bigl\{ \int_{0}^{T}\int_{\T^{d} \times \R^{d} \times \R^{d}}{\left(\frac{1}{2}|w|^{2} + F(x,v, m^{T}_{t}) \right)\ \mu_{t}(dx,dv,dw)dt} \\
& \qquad \qquad 	+\  \int_{\T^{d} \times \R^{d} \times \R^{d}}{g(x,v,m^{T}_{T})\ \mu_{T}(dx,dv,dw)}\Bigr\}\\ 
&\qquad  \leq   \inf_{\tilde  \mu\in \C} \Bigl\{ \int_{\T^{d} \times \R^{d} \times \R^{d}}\left(\frac{1}{2}|w|^{2} + F(x,v, \bar{m}) \right)\ \tilde  \mu(dx,dv,dw) \\
& \qquad\qquad + CT^{-\frac{1}{4(2d+1)}} (1+ M_2(m_0)+M_2(\pi\sharp \tilde \mu)) \Bigr\} + CT^{-\frac14}+T^{-1}\|g\|_\infty.
\end{split}
\end{align}
Since, by assumption {\bf (F2')},
$$
\int_{\T^{d} \times \R^{d} \times \R^{d}} F(x,v, \bar{m})\ \tilde \mu(dx,dv,dw) \geq 
c_F^{-1} M_2(\pi\sharp \tilde \mu)-c_F, 
$$
one easily checks that the limit of the right-hand side of \eqref{uhezljrsdf} as $T\to +\infty$ is 
$$
\inf_{\tilde \mu\in \C} \Bigl\{ \int_{\T^{d} \times \R^{d} \times \R^{d}}\left(\frac{1}{2}|w|^{2} + F(x,v, \bar{m}) \right)\ \tilde \mu(dx,dv,dw)\Bigr\}.
$$
This proves our claim \eqref{uhezljrsdf2}. \\

Next we claim that there exists a closed measure $\hat \mu\in \C$ such that 
\begin{align}\label{iazehrsdfg}
\begin{split}
& \liminf_{T\to+\infty} \frac{1}{T}\Bigl\{ \int_{0}^{T}\int_{\T^{d} \times \R^{d} \times \R^{d}}{\left(\frac{1}{2}|w|^{2} + F(x,v, m^{T}_{t}) \right)\ \mu^T_{t}(dx,dv,dw)dt}	\\
& \qquad \qquad 	+\  \int_{\T^{d} \times \R^{d} \times \R^{d}}{g(x,v,m^{T}_{T})\ \mu^T_{T}(dx,dv,dw)}\Bigr\}\\
& \geq  \int_{\T^{d} \times \R^{d} \times \R^{d}}\left(\frac{1}{2}|w|^{2} + F(x,v, \bar{m}) \right)\ \hat \mu(dx,dv,dw).
\end{split}
\end{align}
For the proof of \eqref{iazehrsdfg}, we work with a subsequence of $T\to+\infty$ (still denoted by $T$) along which the lower limit in the left-hand side is achieved. 
Coming back to \eqref{eq:11bis}, we have 
\begin{align*}
\begin{split}
& \frac{1}{T}\Bigl\{ \int_{0}^{T}\int_{\T^{d} \times \R^{d} \times \R^{d}}{\left(\frac{1}{2}|w|^{2} + F(x,v, m^{T}_{t}) \right)\ \mu^T_{t}(dx,dv,dw)dt}	\\
& \qquad \qquad 	+\  \int_{\T^{d} \times \R^{d} \times \R^{d}}{g(x,v,m^{T}_{T})\ \mu^T_{T}(dx,dv,dw)}\Bigr\}\\
& \geq  \frac{1}{T}\Bigl\{ \int_{0}^{T}\int_{\T^{d} \times \R^{d} \times \R^{d}}{\left(\frac{1}{2}|w|^{2} + F(x,v, \bar m) \right)\ \mu^T_{t}(dx,dv,dw)dt}	\\
& \qquad \qquad -T^{-\frac{1}{4(2d+1)}} \int_0^T (1+M_2(\pi\sharp \mu^T_t))dt	\Bigr\} -CT^{-\frac14}-\|g\|_\infty T^{-1}.
\end{split}
\end{align*}
By the coercivity of $F$ in assumption {\bf (F2')}, we can absorb the second term in the right-hand side into the first one and obtain: 
\begin{align}\label{ualksjfdgBIS}
\begin{split}
& \frac{1}{T}\Bigl\{ \int_{0}^{T}\int_{\T^{d} \times \R^{d} \times \R^{d}}{\left(\frac{1}{2}|w|^{2} + F(x,v, m^{T}_{t}) \right)\ \mu^T_{t}(dx,dv,dw)dt}	\\
& \qquad \qquad 	+\  \int_{\T^{d} \times \R^{d} \times \R^{d}}{g(x,v,m^{T}_{T})\ \mu^T_{T}(dx,dv,dw)}\Bigr\}\\
& \geq  \frac{1}{T} (1- C^{-1}T^{-\frac{1}{4(2d+1)}})  \int_{0}^{T}\int_{\T^{d} \times \R^{d} \times \R^{d}}{\left(\frac{1}{2}|w|^{2} + F(x,v, \bar m) \right)\ \mu^T_{t}(dx,dv,dw)dt}	\\
& \qquad \qquad -CT^{-\frac{1}{4(2d+1)}} -CT^{-\frac14}-\|g\|_\infty T^{-1}.
\end{split}
\end{align}

As in the proof of \Cref{prop:energyestimate} (see \eqref{eq:secondBIS}), for any $\lambda\geq 1$, we can find a closed measure $\hat \mu^T\in \C$ such that 
\begin{align*}
\begin{split}
&  \int_{\T^{d} \times \R^{2d}}{\left(\frac{1}{2}|w|^{2} + F(x,v, \bar{m})\right)\ \hat \mu^{T}(dx,dv,dw)}\\
& \qquad \leq  T^{-1}\int_{0}^{T}\int_{\T^{d} \times \R^{d} \times \R^{d}}{\left(\frac{1}{2}|w|^{2} + F(x,v, \bar m)\right)\ \mu^{T}_{t}(dx,dv,dw)dt}\\
& \qquad \; + 2C_3(1+M_2(m_0))(\lambda^2T^{-1}+\lambda^{-2}). 
\end{split}
\end{align*}
Plugging this inequality into \eqref{ualksjfdgBIS} we find therefore 
\begin{align*}
\begin{split}
& \frac{1}{T}\Bigl\{ \int_{0}^{T}\int_{\T^{d} \times \R^{d} \times \R^{d}}{\left(\frac{1}{2}|w|^{2} + F(x,v, m^{T}_{t}) \right)\ \mu^T_{t}(dx,dv,dw)dt}	\\
& \qquad \qquad 	+\  \int_{\T^{d} \times \R^{d} \times \R^{d}}{g(x,v,m^{T}_{T})\ \mu^T_{T}(dx,dv,dw)}\Bigr\}\\
& \geq  (1- C^{-1}T^{-\frac{1}{4(2d+1)}})  \int_{\T^{d} \times \R^{2d}}{\left(\frac{1}{2}|w|^{2} + F(x,v, \bar{m})\right)\ \hat \mu^{T}(dx,dv,dw)}	 \\
& \qquad \qquad -2C_3(1+M_2(m_0))(\lambda^2T^{-1}+\lambda^{-2}) -CT^{-\frac{1}{4(2d+1)}}.
\end{split}
\end{align*}
By assumption {\bf (F2')}, the functional in the right-hand side of the inequality is coercive for $T$ large enough. So $\hat\mu^{T}$ weakly-$*$ converges (up to a subsequence) to a closed measure $\hat \mu$. Taking the lower-limit in the last inequality then implies  \eqref{iazehrsdfg}. \\

Putting together \eqref{uhezljrsdf2} and \eqref{iazehrsdfg}, we find that $\hat \mu$ is a minimizer in the right-hand side of \eqref{uhezljrsdf2} and that the semi-limits  and the inequalities in  \eqref{uhezljrsdf2} and \eqref{iazehrsdfg} are in fact limits and equalities. So coming back to \eqref{eq:infpb} we find that 
	\begin{align*}
	\begin{split}
	& \lim_{T\to+\infty} \frac{1}{T}\int_{\T^{d} \times \R^{d}}{u^{T}(0,x,v)\ m_{0}(dx,dv)} =  \inf_{\tilde \mu\in \C} \Bigl\{ \int_{\T^{d} \times \R^{d} \times \R^{d}}\left(\frac{1}{2}|w|^{2} + F(x,v, \bar{m}) \right)\ \tilde \mu(dx,dv,dw) \Bigr\}.
	\end{split}
	\end{align*}
The right-hand side of this equality is nothing than but $\bar \lambda$ since $(\bar \lambda,\bar \mu)$ is a solution to the the ergodic MFG problem with $\bar m=\pi\sharp\bar \mu$: this completes the proof of the proposition.   
\qed

To complete the proof of Theorem \ref{thm:theo.main2}, we need  estimates on the oscillation of $u^T$. This comes next: 

\begin{lemma}\label{lem.unifcontuT} For any $R\geq 1$ and $(x,v), (x',v')\in \T^d\times B_R$, we have 
$$
|u^T(0,x,v)-u^T(0,x',v')| \leq CR^2T^{\frac{4d+3}{4(d+1)}},
$$
where $C$ is independent of $T$ and $R$. 
\end{lemma}

\begin{proof} Let $\gamma\in \Gamma(x,v)$ be optimal for $u^T(0,x,v)$ in \eqref{eq:MFGvaluefunction}. We define $\tilde \gamma\in \Gamma(x',v')$ by 
$$
\tilde \gamma(t) = \left\{\begin{array}{ll}
\sigma(t) & {\rm if }\; t\in [0,1]\\
\gamma(t-1) & {\rm if }\; t\in [1,T].
\end{array}\right.
$$
where $\sigma$ is as in \Cref{lem:reachablecurve} with $\sigma(0)=x'$, $\dot\sigma(0)=v'$, $\sigma(1)=x$, $\dot\sigma(1)=v$ and 
$$
\int_0^1 \Bigl( \frac12 |\ddot{\sigma}(t)|^2+ F(\sigma(t), \dot{\sigma}(t), m^T_t)\Bigr)dt \leq 2C_2R^2.
$$
Note that, as the problem for $u^T$ depends on time through $(m^T_t)$,  the cost associated with $\tilde \gamma$ could be quite far from the cost associated with $\gamma$. To overcome this issue, we use in a crucial way \Cref{prop:energyestimate}. Indeed, applying \eqref{eq:energyest} in \Cref{prop:energyestimate}, we have
\begin{align*}
&\int_0^T \left| F(\gamma(t), \dot \gamma(t),m^T_t)-F(\gamma(t), \dot \gamma(t),\bar m))\right|\ dt \\
& \qquad \leq \int_0^T (1+|\dot \gamma(t)|^2)^{\frac{d}{d+1}} \sup_{(y,z)\in \T^d\times \R^d} \frac{| F(y,z, m^T_t)- F(y,z, \bar{m})|}{(1+|v|^2)^{\frac{d}{d+1}}} dt\\
& \qquad \leq \left( \int_0^T (1+|\dot \gamma(t)|^2)^{\frac{2d}{2d+1}}dt \right)^{\frac{2d+1}{2d+2}}
\left( \int_0^T \sup_{(y,z)\in \T^d\times \R^d} \frac{| F(y,z, m^T_t)- F(y,z, \bar{m})|^{2d+2}}{(1+|v|^2)^{2d}} dt\right)^{\frac{1}{2d+2}}\\ 
& \qquad \leq  
C T^{\frac{1}{4(d+1)}}\left( \int_0^T (1+|\dot \gamma(t)|^2)dt \right)^{\frac{2d+1}{2d+2}}. 
\end{align*}
We have by assumption {\bf (F2')} and \Cref{lem:upperbo} that 
\begin{equation}\label{ezjlkrfdg}
\int_0^T (c_F^{-1}|\dot \gamma(t)|^2-c_F) dt \leq u^T(0,x,v) \leq c_FT(1+|v|^2 ). 
\end{equation}
Therefore 
\begin{align}\label{sdflfngcdfjkcv}
&\int_0^T \left| F(\gamma(t), \dot \gamma(t),m^T_t)-F(\gamma(t), \dot \gamma(t),\bar m))\right| \leq  C  T^{\frac{4d+3}{4(d+1)}}
(1+R^2)^{\frac{2d+1}{2d+2}}.
\end{align}
For the very same reason  we also have 
\begin{align}\label{sdflfngcdfjkcvBIS}
&\int_1^T \left| F(\gamma(t-1), \dot \gamma(t-1),m^T_t)-F(\gamma(t-1), \dot \gamma(t-1),\bar m))\right| \leq  C  T^{\frac{4d+3}{4(d+1)}}
(1+R^2)^{\frac{2d+1}{2d+2}},
\end{align}
because we only used the optimality of $\gamma$ only in the estimate \eqref{ezjlkrfdg}. So, by \eqref{sdflfngcdfjkcv} and \eqref{sdflfngcdfjkcvBIS} we obtain
\begin{align*}
& u^T(0,x',v')\leq \int_0^T \Bigl(\frac12 |\ddot{\tilde \gamma}(t)|^2+ F(\tilde \gamma(t), \dot{\tilde \gamma}(t), m^T_t)\Bigr)dt \\
& =  \int_0^1  \Bigl( \frac12 |\ddot{\sigma}(t)|^2+ F(\sigma(t), \dot{\sigma}(t), m^T_t)\Bigr)dt + \int_0^{T-1} \Bigl(\frac12 |\ddot{ \gamma}(t-1)|^2+ F( \gamma(t-1), \dot{ \gamma}(t-1), m^T_{t})\Bigr)dt \\
& \qquad \leq 2C_2R^2+ \int_0^{T} \Bigl(\frac12 |\ddot{ \gamma}(t)|^2+ F( \gamma(t), \dot{ \gamma}(t), \bar m)\Bigr)dt +C T^{\frac{4d+3}{4(d+1)}} (1+R^2)^{\frac{2d+1}{2d+2}}\\
&  \leq 2C_2R^2+ \int_0^{T} \Bigl(\frac12 |\ddot{ \gamma}(t)|^2+ F( \gamma(t), \dot{ \gamma}(t), m^T_{t})\Bigr)dt+2C T^{\frac{4d+3}{4(d+1)}} (1+R^2)^{\frac{2d+1}{2d+2}} \\
&  
\leq u^T(0,x,v) +2C_2R^2 + 2C T^{\frac{4d+3}{4(d+1)}} (1+R^2)^{\frac{2d+1}{2d+2}}, 
\end{align*}
from which the result derives easily. 
\end{proof}


\begin{proof}[Proof of Theorem \ref{thm:theo.main2}] Proposition \ref{prop.exuniErgoMFGeq} states the existence of a solution for the ergodic MFG system and its uniqueness under assumption \eqref{monotonicity}. From Proposition \ref{prop:convergenceMFG} we know  that 
$$
\lim_{T\to+\infty} \frac{1}{T} \int_{\T^d\times \R^d} u^T(0,x,v)m_0(dx,dv) = \bar \lambda.
$$
It remains to prove the local uniform convergence of $u^T$ to $\bar \lambda$. Fix $R>0$ and $\eps>0$.  We have by \Cref{lem:upperbo} that 
\begin{equation}\label{kuzhjjneqsdfcxv}
0\leq u^T(0,x,v)\leq c_FT(1+|v|^2 ). 
\end{equation}
As $m_0\in \PP_2(\T^d\times \R^d)$, there exists $R'\geq R$ such that 
\begin{equation}\label{kuzhjjneqsdfcxvBIS}
 \int_{\T^d\times (\R^d\backslash B_{R'})}(1+|v|^2) m_0(dx,dv) \leq \eps.
\end{equation}
Then, for any $(x_0,v_0)\in \T^d\times B_R$, we have, by Lemma \ref{lem.unifcontuT},  \eqref{kuzhjjneqsdfcxv} and \eqref{kuzhjjneqsdfcxvBIS}, 
\begin{align*}
&|\frac{1}{T}u^T(0,x_0,v_0) - \bar \lambda | \leq  \left|  \frac{1}{T} \int_{\T^d\times \R^d} u^T(0,x,v)m_0(dx,dv) -\bar \lambda\right| \\
& \qquad  + 
 \frac{1}{T} \int_{\T^d\times B_{R'}} \left|u^T(0,x,v)- u^T(0,x_0,v_0)\right|m_0(dx,dv)\\
 & \qquad  +  \frac{1}{T} \int_{\T^d\times (\R^d\backslash B_{R'})} (|u^T(0,x,v)|+|u^T(0,x_0,v_0)|)m_0(dx,dv)   \\
& \qquad  \leq \left|  \frac{1}{T} \int_{\T^d\times \R^d} u^T(0,x,v)m_0(dx,dv) -\bar \lambda\right|  + CT^{-1}(R')^2T^{\frac{4d+3}{4(d+1)}} +c_F\eps(2+R^2), 
\end{align*}
from which the local uniform convergence of $u^T(0,\cdot,\cdot)/T$ to $\bar \lambda$ can be obtained easily. 
\end{proof}

\appendix

\section{Von Neumann minmax theorem}

Let $\mathbb{A}$, $\mathbb{B}$ be convex sets of some vector spaces and let us suppose that $\mathbb{B}$ is endowed with some Hausdorff topology. Let $\mathcal{L}: \mathbb{A} \times \mathbb{B} \rightarrow \mathbb{R}$ be a saddle function satisfying
\begin{enumerate}
\item $a \mapsto \mathcal{L}(a,b)$ is concave in $\mathbb{A}$ for every $b \in \mathbb{B}$, 
\item $b \mapsto \mathcal{L}(a,b)$ is convex in $\mathbb{B}$ for every $a \in \mathbb{A}$. 	
\end{enumerate}

It is always true that
\begin{equation*}
\inf_{b \in \mathbb{B}}\sup_{a \in \mathbb{A}} \mathcal{L}(a,b) \geq \sup_{a \in \mathbb{A}}\inf_{b \in \mathbb{B}} \mathcal{L}(a,b)	.
\end{equation*}

\begin{theorem}[\cite{bib:OPA}]\label{thm:minmax}
	Assume that there exists $a^{*} \in \mathbb{A}$ and $c^{*} > \sup_{a \in \mathbb{A}} \inf_{b \in\mathbb{B}} \mathcal{L}(a,b)$ such that
	\begin{equation*}
	\mathbb{B}^{*}:= \left\{b \in \mathbb{B}: \mathcal{L}(a^{*}, b) \leq c^{*} \right\}	
	\end{equation*}
is not empty and compact in $\mathbb{B}$, and that $b \mapsto \mathcal{L}(a,b)$ is lower semicontinuous in $\mathbb{B}^{*}$ for every $a \in \mathbb{A}$. 

Then
\begin{equation*}
\min_{b \in \mathbb{B}}\sup_{a \in \mathbb{A}}\mathcal{L}(a,b) = \sup_{a \in \mathbb{A}} \inf_{b \in \mathbb{B}} \mathcal{L}(a,b). 	
\end{equation*}
\end{theorem}

\section{An interpolation inequality}

\begin{lemma}\label{lem.interpol} Assume that $f:\T^d\times \R^d\to \R$ is locally Lipschitz continuous with 
\begin{equation}\label{hyphypapp2}
|f(x,v)|+ |D_xf(x,v)|+|D_vf(x,v)|\leq c_0(1+|v|^\alpha)\qquad \mbox{\rm for a.e. $(x,v)\in \T^d\times \R^d$}
\end{equation}
for some constants $c_0>0$ and $\alpha\in (1, 2]$. There exists a constants $C_d>0$ (depending on dimension only) such that  
$$
\sup_{(x,v)\in \T^d\times \R^d} \frac{|f(x,v)|^{2d+2}}{(1+|v|^\alpha)^{2d}} \leq C_d c_0^{2d} \int_{\T^d\times \R^d} |f(x,v)|^2dxdv.
$$
\end{lemma}

\begin{proof}  Let $(x_0, v_0)\in \T^d\times \R^d$ be such that $f(x_0, v_0)\neq 0$ and let $R= \frac{|f(x_0,v_0)|}{2c_0(3+2|v_0|^\alpha)}$. Note that, by our assumption on $|f|$ in \eqref{hyphypapp2}, $R$ is less than $1$.  Then, for any $(x,v)\in B_R(x_0,v_0)$, we have by assumption \eqref{hyphypapp2} that 
$$
|D_xf(x,v)|+|D_vf(x,v)|\leq  c_0(1+(1+|v_0|)^\alpha) \leq c_0(1+2^{\alpha-1}+ 2^{\alpha-1}|v_0|^\alpha) \leq c_0(3+2|v_0|^\alpha),
$$
(where we used the fact that $R\leq 1$ and that $(a+b)^\alpha\leq 2^{\alpha-1}(a^\alpha+b^\alpha)$ in the first inequality and the fact that $\alpha\leq 2$ in the second one). Therefore 
$$
|f(x,v)|\geq |f(x_0,v_0)| - c_0(3+2|v_0|^\alpha)R = \frac{|f(x_0, v_0)|}{2}.
$$
Taking the square and integrating over $B_R(x_0,v_0)$ gives
$$
\int_{\T^d\times \R^d} |f(x,v)|^2dxdv \geq |B_1| R^{2d} \frac{|f(x_0,v_0)|^2}{4}= |B_1| \frac{|f(x_0,v_0)|^{2d+2}}{2^{2d+2}c_0^{2d}(3+2|v_0|^\alpha)^{2d}}\ ,
$$
 which implies the result. 
\end{proof}

\bibliographystyle{plain}

\end{document}